\documentclass[reqno,english,12pt]{amsart}
 \usepackage[margin=1in]{geometry}
 
\usepackage{amssymb,amsmath,amsthm,mathtools,amsfonts}

\usepackage{graphicx}
\usepackage[usenames,dvipsnames]{xcolor}
\usepackage{cite}
\usepackage{url}
\usepackage{hyperref}
\usepackage{xcolor}

\usepackage[version=4]{mhchem}

\hypersetup{
    bookmarks=true,         	
    unicode=false,          		
    pdftoolbar=true,        		
    pdfmenubar=true,        	
    pdffitwindow=false,     		
    pdfstartview={FitH},    		
   pdftitle={Long-Time Asymptotics of Solutions to the defocussing Modified Korteweg-de vries Equation for Soliton-Free Initial Data},    					
   pdfauthor={Gong Chen, Jiaqi Liu},     	
    linktocpage=true,			
    pdfnewwindow=true,      	
    colorlinks=true,       			
    linkcolor=red,          		
    citecolor=PineGreen,    	
    filecolor=magenta,      		
    urlcolor=cyan           		
}

\usepackage{tikz}
\usetikzlibrary{decorations.pathmorphing}

\usepackage{float}
\usepackage[section]{placeins}		

\theoremstyle{plain}
\newtheorem{theorem}{Theorem}[section]
\newtheorem{corollary}[theorem]{Corollary}
\newtheorem{lemma}[theorem]{Lemma}
\newtheorem{proposition}[theorem]{Proposition}

\theoremstyle{definition}
\newtheorem{definition}[theorem]{Definition}
\newtheorem{problem}[theorem]{Problem}

\theoremstyle{remark}
\newtheorem{remark}[theorem]{Remark}

\numberwithin{figure}{section}
\numberwithin{equation}{section}

\DeclareMathOperator{\ad}{ad}

\DeclareMathOperator{\Real}{Re}
\DeclareMathOperator{\Imag}{Im}

\newcommand{\vertiii}[1]{{\left\vert\kern-0.25ex\left\vert\kern-0.25ex\left\vert #1 
    \right\vert\kern-0.25ex\right\vert\kern-0.25ex\right\vert}}
\allowdisplaybreaks

\newenvironment{doublecases}
{
	\left\{ 
			\begin{array}{lllll}
}
{			
			\end{array} 
			\right.
}

\begin{document}

\title[Long-Time Asymptotics of MKdV]{Long-Time Asymptotics of the Modified KdV Equation in Weighted Sobolev Spaces}
\author{Gong Chen}
\author{Jiaqi Liu}
\address[Chen]{Department of Mathematics, University of Toronto, Toronto, Ontario M5S 2E4, Canada }
\email{gc@math.toronto.edu}
\address[Liu]{School of Mathematics, University of the Chinese Academy of sciences, Beijing, China }
\email{jqliu@ucas.ac.cn}


\maketitle
\begin{abstract}
The long time behavior of  solutions to  the defocussing modified Korteweg-de Vries (MKdV) equation is established for initial 
conditions in  some weighted Sobolev spaces. Our approach is based on the nonlinear steepest descent method of Deift and Zhou and its reformulation by Dieng and McLaughlin through $\overline{\partial}$-derivatives. To extend the asymptotics to solutions with initial data in lower regularity spaces, we apply a global approximation via PDE techniques.

\medskip

\end{abstract}

\tableofcontents

%
%

\newcommand{\eps}{\varepsilon}
\newcommand{\lam}{\lambda}

\newcommand{\bfN}{\mathbf{N}}
\newcommand{\calbR}{\mathcal{ \breve{R}}}
\newcommand{\rhobar}{\overline{\rho}}
\newcommand{\zetabar}{\overline{\zeta}}

\newcommand{\rarr}{\rightarrow}
\newcommand{\darr}{\downarrow}

\newcommand{\dee}{\partial}
\newcommand{\dbar}{\overline{\partial}}

\newcommand{\dint}{\displaystyle{\int}}

\newcommand{\dotarg}{\, \cdot \, }

%
%

\newcommand{\RHP}{\mathrm{LC}}			
\newcommand{\PC}{\mathrm{PC}}
\newcommand{\w}{w^{(2)}}
%
%

\newcommand{\zbar}{\overline{z}}

\newcommand{\bbC}{\mathbb{C}}
\newcommand{\bbR}{\mathbb{R}}

\newcommand{\calB}{\mathcal{B}}
\newcommand{\calC}{\mathcal{C}}
\newcommand{\calR}{\mathcal{R}}
\newcommand{\calS}{\mathcal{S}}

\newcommand{\ba}{\breve{a}}
\newcommand{\bb}{\breve{b}}

\newcommand{\balpha}{\breve{\alpha}}
\newcommand{\brho}{\breve{\rho}}

\newcommand{\tPhi}{{\widetilde{\Phi}}}

\newcommand{\bfe}{\mathbf{e}}
\newcommand{\bfn}{\mathbf{n}}

\newcommand{\bphi}{\breve{\Phi}}
\newcommand{\bN}{\breve{N}}
\newcommand{\bV}{\breve{V}}
\newcommand{\bR}{\breve{R}}
\newcommand{\bdelta}{\breve{\delta}}
\newcommand{\bzeta}{\breve{\zeta}}
\newcommand{\bbeta}{\breve{\beta}}
\newcommand{\bm}{\breve{m}}
\newcommand{\br}{\breve{r}}
\newcommand{\bnu}{\breve{\nu}}
\newcommand{\bbfN}{\breve{\mathbf{N}}}
\newcommand{\rbar}{\overline{r}}

\newcommand{\One}{\mathbf{1}}

%
%

\newcommand{\bigO}[2][ ]
{
\mathcal{O}_{#1}
\left(
{#2}
\right)
}

\newcommand{\littleO}[1]{{o}\left( {#1} \right)}

\newcommand{\norm}[2]
{
\left\Vert		{#1}	\right\Vert_{#2}
}

%
%

\newcommand{\rowvec}[2]
{
\left(
	\begin{array}{cc}
		{#1}	&	{#2}	
	\end{array}
\right)
}

\newcommand{\uppermat}[1]
{
\left(
	\begin{array}{cc}
	0		&	{#1}	\\
	0		&	0
	\end{array}
\right)
}

\newcommand{\lowermat}[1]
{
\left(
	\begin{array}{cc}
	0		&	0	\\
	{#1}	&	0
	\end{array}
\right)
}

\newcommand{\offdiagmat}[2]
{
\left(
	\begin{array}{cc}
	0		&	{#1}	\\
	{#2}	&	0
	\end{array}
\right)
}

\newcommand{\diagmat}[2]
{
\left(
	\begin{array}{cc}
		{#1}	&	0	\\
		0		&	{#2}
		\end{array}
\right)
}

\newcommand{\Offdiagmat}[2]
{
\left(
	\begin{array}{cc}
		0			&		{#1} 	\\
		\\
		{#2}		&		0
		\end{array}
\right)
}

\newcommand{\twomat}[4]
{
\left(
	\begin{array}{cc}
		{#1}	&	{#2}	\\
		{#3}	&	{#4}
		\end{array}
\right)
}

\newcommand{\unitupper}[1]
{	
	\twomat{1}{#1}{0}{1}
}

\newcommand{\unitlower}[1]
{
	\twomat{1}{0}{#1}{1}
}

\newcommand{\Twomat}[4]
{
\left(
	\begin{array}{cc}
		{#1}	&	{#2}	\\[10pt]
		{#3}	&	{#4}
		\end{array}
\right)
}

%
%
%

\newcommand{\JumpMatrixFactors}[6]
{
	\begin{equation}
	\label{#2}
	{#1} =	\begin{cases}
					{#3} {#4}, 	&	z \in (-\infty,\xi) \\
					\\
					{#5}{#6},	&	z \in (\xi,\infty)
				\end{cases}
	\end{equation}
}


%
%
%

\newcommand{\RMatrix}[9]
{
\begin{equation}
\label{#1}
\begin{aligned}
\left. R_1 \right|_{(\xi,\infty)} 	&= {#2} &	\qquad\qquad		
\left. R_1 \right|_{\Sigma_1}		&= {#3} 
\\[5pt]
\left. R_3 \right|_{(-\infty,\xi)} 	&= {#4} 	&	
\left. R_3 \right|_{\Sigma_2} 	&= {#5}
\\[5pt]
\left. R_4 \right|_{(-\infty,\xi)} 	&= {#6} &	
\left. R_4 \right|_{\Sigma_3} 	&= {#7} 
\\[5pt]
\left. R_6 \right|_{(\xi,\infty)}  	&= {#8} &	
\left. R_6 \right|_{\Sigma_4} 	&= {#9}
\end{aligned}
\end{equation}
}

%
%

%
%
%
%
%
%

\newcommand{\SixMatrix}[6]
{
\begin{figure}
\centering
\caption{#1}
\vskip 15pt
\begin{tikzpicture}
[scale=0.75]
%
%
\draw[thick]	 (-4,0) -- (4,0);
\draw[thick] 	(-4,4) -- (4,-4);
\draw[thick] 	(-4,-4) -- (4,4);
%
%
\draw	[fill]		(0,0)						circle[radius=0.075];
\node[below] at (0,-0.1) 				{$z_0$};
%
%
\node[above] at (3.5,2.5)				{$\Omega_1$};
\node[below]  at (3.5,-2.5)			{$\Omega_6$};
\node[above] at (0,3.25)				{$\Omega_2$};
\node[below] at (0,-3.25)				{$\Omega_5$};
\node[above] at (-3.5,2.5)			{$\Omega_7^+$};
\node[below] at (-3.5,-2.5)			{$\Omega_8^+$};
%
%
\node[above] at (0,1.25)				{$\twomat{1}{0}{0}{1}$};
\node[below] at (0,-1.25)				{$\twomat{1}{0}{0}{1}$};
%
%
\node[right] at (1.20,0.70)			{$#3$};
\node[left]   at (-1.20,0.70)			{$#4$};
\node[left]   at (-1.20,-0.70)			{$#5$};
\node[right] at (1.20,-0.70)			{$#6$};
\end{tikzpicture}
\label{#2}
\end{figure}
}

\newcommand{\sixmatrix}[6]
{
\begin{figure}
\centering
\caption{#1}
\vskip 15pt
\begin{tikzpicture}
[scale=0.75]
%
%
\draw[thick]	 (-4,0) -- (4,0);
\draw[thick] 	(-4,4) -- (4,-4);
\draw[thick] 	(-4,-4) -- (4,4);
%
%
\draw	[fill]		(0,0)						circle[radius=0.075];
\node[below] at (0,-0.1) 				{$-z_0$};
%
%
\node[above] at (3.5,2.5)				{$\Omega_7^-$};
\node[below]  at (3.5,-2.5)			{$\Omega_8^-$};
\node[above] at (0,3.25)				{$\Omega_2$};
\node[below] at (0,-3.25)				{$\Omega_5$};
\node[above] at (-3.5,2.5)			{$\Omega_3$};
\node[below] at (-3.5,-2.5)			{$\Omega_4$};
%
%
\node[above] at (0,1.25)				{$\twomat{1}{0}{0}{1}$};
\node[below] at (0,-1.25)				{$\twomat{1}{0}{0}{1}$};
%
%
\node[right] at (1.20,0.70)			{$#3$};
\node[left]   at (-1.20,0.70)			{$#4$};
\node[left]   at (-1.20,-0.70)			{$#5$};
\node[right] at (1.20,-0.70)			{$#6$};
\end{tikzpicture}
\label{#2}
\end{figure}
}

%
%
%
%

%
%
%
%

\newcommand{\JumpMatrixRightCut}[6]
{
\begin{figure}
\centering
\caption{#1}
\vskip 15pt
\begin{tikzpicture}[scale=0.85]
%
%
\draw [fill] (4,4) circle [radius=0.075];						
\node at (4.0,3.65) {$\xi$};										
%
%
\draw 	[->, thick]  	(4,4) -- (5,5) ;								
\draw		[thick] 		(5,5) -- (6,6) ;
\draw		[->, thick] 	(2,6) -- (3,5) ;								
\draw		[thick]		(3,5) -- (4,4);	
\draw		[->, thick]	(2,2) -- (3,3);								
\draw		[thick]		(3,3) -- (4,4);
\draw		[->,thick]	(4,4) -- (5,3);								
\draw		[thick]  		(5,3) -- (6,2);
%
%
\draw [  thick, blue, decorate, decoration={snake,amplitude=0.5mm}] (4,4)  -- (8,4);				
\node at (1.5,4) {$0 < \arg (\zeta-\xi) < 2\pi$};
%
%
\node at (8.5,8.5)  	{$\Sigma_1$};
\node at (-0.5,8.5) 	{$\Sigma_2$};
\node at (-0.5,-0.5)	{$\Sigma_3$};
\node at (8.5,-0.5) 	{$\Sigma_4$};
%
%
\node at (7,7) {${#3}$};						
\node at (1,7) {${#4}$};						
\node at (1,1) {${#5}$};						
\node at (7,1) {${#6}$};						
\end{tikzpicture}
\label{#2}
\end{figure}
}

%
%

\section{Introduction}

%
%
In this paper we calculate the  
long-time asymptotics of  solutions  to  the defocussing  modified KdV equation (MKdV):
\begin{equation}
\label{MKDV}
u_t + u_{xxx} - 6u^2u_x=0 \qquad (x, t)\in (\bbR, \bbR^+).
\end{equation}
There is a vast body of literature regarding the MKdV equation, in
particular with the local and global well-posedness of the Cauchy
problem. For a summary of known results we refer the reader to Linares-Ponce
\cite{LP}. Without trying to be exhaustive, we mention the works by Kato \cite{Kat}, Kenig-Ponce-Vega
\cite{KPV}, Colliander-Keel-Staffilani-Takaoka-Tao \cite{CKSTT},
Guo \cite{Guo} and Kishimoto \cite{Kis}. In particular, we know
that the MKdV for both the focussing and defocussing cases on the line
is locally well-posed (cf. Kenig-Ponce-Vega \cite{KPV}), and globally
well-posed, (cf. Colliander-Keel-Staffilani-Takaoka-Tao \cite{CKSTT},
Guo \cite{Guo} and Kishimoto \cite{Kis}),  in $H^{s}\left(\mathbb{R}\right)$for
$s\ge\frac{1}{4}$. These results are complemented by several ill-posedness
results (cf. Christ-Colliander-Tao \cite{CCT} and references
therein)   {which establish that  $H^{\frac{1}{4}}(\mathbb{R})$ is optimal if one requires that solutions depend uniformly continuously on the initial data.
	After the completion  of the first version of the current paper, there have been significant progresses regarding the global well-posedness of integrable PDEs on the real line,  in particular for the KdV, mKdV and NLS equations, see Killip-Visan \cite{KV}, Harrop-Griffiths-Killip-Visan \cite{HGKV}. In \cite{HGKV}, for the mKdV equation, the global well-posedness is obtained in $H^\tau(\mathbb{R})$ for $\tau>-\frac{1}{2}$. It is also known that instantaneous norm inflation happens in  $H^\tau(\mathbb{R})$ for $\tau=-\frac{1}{2}$. We again refer to  \cite{HGKV} for details.  }

\smallskip

Besides well-posedness, another fundamental question for dispersive
PDEs is the long-time asymptotics. Using the complete integrability of the MKdV equation,  Deift and Zhou in their seminal work \cite{DZ93} developed the celebrated nonlinear steepest descent method for oscillatory Riemann-Hilbert problems. In the same paper, the authors give explicit asymptotic formulae and error terms for Schwartz class initial data. Since then, analysis of  long-time behavior
of integrable systems have been extensively treated by
many authors. The nonlinear steepest descent method provides a systematic way to 
reduce the original RHP 
to a canonical model RHP whose solution is calculated in terms of special functions.  This reduction is done through a sequence of transformations whose effects do not change  the  long-time behavior of the recovered solution at leading order. In this way, one  obtains the asymptotic behavior of the solution in terms of the spectral data (thus in terms of the initial conditions).

\smallskip

A natural question to ask is whether it is possible to study the asymptotic
behavior of the MKdV equation without relying on the completely integrable
structure. A proof of global existence and a (partial) derivation
of the asymptotic behavior for small localized solutions was later given by Hayashi and
Naumkin in \cite{HN1,HN2} using the method of factorization of operators.
Recently, Germain-Pusateri-Rousset \cite{GPR} use the idea of the space-time resonance to study the long-time asymptotics of
small data and soliton stability problem. Also a precise derivation of
asymptotics and a proof of asymptotic completeness, was given by Harrop-Griffiths
\cite{HGB} using wave packets analysis. Overall, although PDE techniques
do not rely on the complete integrability, to our best knowledge,  certain
smallness assumptions on the initial data are required. 

\smallskip
In the present paper, we use the inverse scattering transform/nonlinear steepest descent
to study the long-time asymptotics of  solution to the MKdV equation without smallness assumption on the initial data. We give a full description of the long-time behavior of solutions in the weighted Sobolev space $H^{2,1}$ which is necessary to construct the solution { via inverse scattering}  and extend these results to other Sobolev spaces including $H^{1,1}$, $H^{\frac{1}{4},1}$   and $L^{2,1}$ via a global approximation argument. 

\smallskip

In Deift-Zhou \cite{DZ93}, a key step in the nonlinear steepest descent method consists of  deforming the contour
associated to  the RHP in such a way that  the phase function  with oscillatory 
dependence on parameters become exponential decay.
In general the entries of the jump matrix are not analytic, so direct analytic extension off the real axis is not possible. Instead
they must be approximated by rational functions and this results in some error term in the recovered solution. Therefore, in the context of nonlinear steepest descent, most results are carried out
under the assumptions that the initial data belong to the Schwartz
space.

\smallskip

In \cite{Zhou98},  Xin Zhou developed a rigorous analysis of the direct and inverse scattering transform of the AKNS system
for a class of initial conditions $u_0(x)=u(x,t=0)$ belonging to the space  $H^{i,j}(\bbR)$.
Here,  $H^{i,j}(\bbR)$  denotes  
the completion of $C_0^\infty(\bbR)$ in the norm
\begin{equation}
\label{sp: weighted}
    \norm{u}{H^{i,j}(\bbR)}
= \left( \norm{(1+|x|^j)u}{2}^2 + \norm{u^{(i)}}{2}^2 \right)^{1/2}. 
\end{equation}
Recently, much effort has been devoted to relax the regularities
of the initial data. In particular, among the most celebrated results
concerning nonlinear Schr\"odinger equations, we point out the work
of Deift-Zhou \cite{DZ03} where they provide the asymptotics for
the NLS in the weighted space $L^{2,1}$. This topology is more or
less optimal from the views of PDE and inverse scattering transformations.
The global $L^{2}$ existence of the cubic NLS can be carried out
by the $L_{t}^{4}L_{x}^{\infty}$ Strichartz estimate and the conservation
of the $L^{2}$ norm. But in order to obtain the precise asymptotics,
one needs to \textquotedblleft pay the price of weights\textquotedblright{}, i.e. working with the weighted space $L^{2,1}$.

\smallskip

Dieng and McLaughlin in \cite{DM08} (see also an extended version \cite{DMM18}) developed a variant of Deift-Zhou method. In their approach 
rational approximation of the reflection coefficient is replaced by some 
non-analytic extension of the jump matrices off the real axis, which  leads to a $\bar{\partial}$-problem to 
be solved in some regions of the complex plane. The
new  $\bar{\partial}$-problem can be reduced to an integral equation and is solvable through
Neumann series.   
These ideas were originally implemented by Miller and McLaughlin \cite{MM08} to the 
study the  asymptotics of orthogonal polynomials. This method has shown its robustness in its application to other integrable models. Notably, for focussing NLS and derivative NLS, they were 
successfully applied to address the soliton resolution in \cite{BJM16} and \cite{JLPS18} respectively. In this paper, we incorporate this approach into the framework of \cite{DZ93}  to calculate the long time behavior of the defocussing MKdV equation in weighted Sobolev spaces.
The soliton resolution of the focussing MKdV equation will be addressed in a forthcoming article \cite{CL19}.

\subsection{Direct and inverse scattering formalism} 

To describe our approach, we recall that \eqref{MKDV} generates an iso-spectral flow for the problem
\begin{equation}
\label{L}
\frac{d}{dx} \Psi = -iz \sigma_3 \Psi + U(x) \Psi
\end{equation}
where
$$ \sigma_3 = \diagmat{1}{-1}, \,\,\, U(x) = \offdiagmat{iu(x)}{\overline{iu(x)}}.$$
This is a standard AKNS system. If {$u \in L^1(\bbR) $}, equation \eqref{L} admits bounded 
solutions for $z \in \mathbb{R}$.   There exist unique solutions $\Psi^\pm$ of \eqref{L} obeying the the following space asymptotic conditions
$$\lim_{x \rarr \pm \infty} \Psi^\pm(x,z) e^{-ix z \sigma_3} = \diagmat{1}{1},$$
and there is a matrix $T(z)$, the transition matrix, with 
 $\Psi^+(x,z)=\Psi^-(x,z) T(z)$.
The matrix $T(z)$ takes the form
\begin{equation} \label{matrixT}
 T(z) = \twomat{a(z)}{\bb(z)}{b(z)}{\ba(z)} 
 \end{equation}
and  the determinant relation gives
$$ a(z)\ba(z) - b(z)\bb(z) = 1 $$
Combining this with the symmetry relations 
\begin{align} \label{symmetry}
\ba(z)=\overline{a( \zbar )}, \quad \bb(z) = \overline{ b(\zbar)}. 
\end{align}
we arrive at
$$|a(z)|^2-|b(z)|^2=1$$
and conclude that $a(z)$ is zero-free.

By the standard inverse scattering theory, we formulate the reflection coefficient:
\begin{equation}
\label{reflection}
r(z)=\bb(z)/a(z), \quad z\in\bbR
\end{equation}
The functions $r(z)$ is  called the \emph{scattering data} 
for the initial data $u_0$ satisfying the following symmetry relation:
\begin{equation}
\label{minus}
r(z)=-\overline{r(-z)}.
\end{equation}

We also have the following identity
\begin{equation*}
a(z) \ba(z)   = (1-|r(z)|^2)^{-1}  \quad z\in\bbR.
\end{equation*}

{In \cite{Zhou98}, it is shown that for $k, j$ integers with $k\geq 0$, $j\geq 1$, the direct scattering map $\mathcal{R}$ maps $H^{k,j}(\bbR)$ onto $H^{j,k}_1=H^{j,k}(\bbR)\cap \lbrace r: \norm{r}{L^\infty} <1\rbrace$ where $H^{j,k}$ norm is defined in \eqref{sp: weighted} and the map $\mathcal{R}: u_0 \mapsto r$ is Lipschitz continuous. }Since we are dealing with the defocussing MKdV,  only the reflection coefficient $r$ is needed for the reconstruction of the solution. The long-time behavior of the solution to the MKdV equation is obtained through a sequence of transformations of the following RHP:

\begin{problem}
\label{prob:DNLS.RH0}
Given $r \in H^{1,2}(\bbR)$ for $z \in \bbR$, 
find a $2\times 2$ matrix-valued function $m(z;x,t)$ on $\bbC \setminus \bbR$ with the following properties:
\begin{enumerate}
\item		$m(z;x,t) \rarr I$ as $|z| \rarr \infty$,
			\medskip
			
\item		$m(z;,x,t)$ is analytic for $z \in \bbC \setminus \bbR$ with continuous boundary values
			$$m_\pm(z;x,t) = \lim_{\eps \darr 0} m(z\pm i\eps;x,t),$$
			\medskip
			
\item		
The jump relation $m_+(z;x,t) = m_	-(z;x,t) e^{-i\theta \ad \sigma_3} v(z)$ holds, where
			\begin{equation}
			\label{mkdv.V}
			e^{-i\theta \ad \sigma_3} v(z)	=		\Twomat{1-|r(z)|^2}
											{-\overline{r(z)}e^{-2i\theta}}{r(z)e^{2i\theta}}{1}
			\end{equation}
and the real phase function $\theta$ is given by
\begin{equation}
\label{mkdv.phase}
\theta(z;x,t) = 4tz^3+xz
\end{equation}
with stationary points
\begin{equation}
    \label{stationary pt}
    \pm z_0=\pm \sqrt{ \dfrac{-x}{12t} }
\end{equation}
\end{enumerate}
\end{problem}
Note that the jump matrix $v$ admits the following factorization on $\bbR$:
\begin{align*}
e^{-i\theta \ad \sigma_3} v(z)	&=\twomat{1}{-\rbar  e^{-2i\theta}}{0}{1}
					\twomat{1}{0}{r  e^{2i\theta}}{1} \\
					   &=(1-w_\theta^-)^{-1}(1+ w_\theta^+). 
\end{align*}
We define
\begin{align*}
\mu= m_+(I+ w_\theta^+ )^{-1}=m_-(I - w_\theta^-)^{-1}
\end{align*}
 then it is well known that solvability of the RHP above is equivalent to the solvability of the following Beals-Coifman integral equation:
\begin{align}
\label{BC-int}
\mu(z; x,t) &= I +C_{w_\theta}\mu(z; x,t)\\
                &= I+C^+\mu w_\theta^- +C^-\mu w_\theta^+
\end{align}
Here $C^\pm$ is the Cauchy projection:
\begin{equation}
(C^\pm f)(z)= \lim_{z\to \Sigma_\pm}\dfrac{1}{2\pi i} \int_{\Sigma} \dfrac{f(s)}{s-z}ds
\end{equation}
and  $+(-)$ denotes taking limit from the positive (negative) side of the oriented contour.
\bigskip\noindent
From the solution of Problem \ref{prob:DNLS.RH0}, we recover
\begin{align}
\label{mkdv.q}
u(x,t) &= \lim_{z \rarr \infty} -2 z m_{12}(x,t,z)\\
\label{mkdv.BC}
        &=\left[ \dfrac{-i}{\pi}\int_\bbR \mu (w_\theta^-+w_\theta^+) \right]_{12}
\end{align}
where the limit is taken in $\bbC\setminus \bbR$ along any direction not tangent to $\bbR$.  

\subsection{Main results}
The central results of this paper are the following theorems
that give the long-time behavior of  the  solution $u(x,t)$
 of \eqref{MKDV} in different regions in the $(x,t)$ plane respectively.
 
%
%

\begin{figure}[h!]
\caption{Five Regions}
\vskip 15pt

\begin{tikzpicture}[scale=0.7]
\draw[fill] (0,0) circle[radius=0.075];
\draw[thick] 	(-8,0) -- (-2,0);
\draw[thick]		(-2,0) -- (0,0);
\draw[->,thick,>=stealth]	(0,0) -- (6,0);
\draw[thick]		(2,0) -- (7,0);
\draw		(-7.2, -0.2) -- (-3.5, -0.2); 
\node[below] at (-5.5, -0.2) {I};
\draw		(-5, 0.2) -- (-0.8, 0.2); 
\node[above] at (-2.5, 0.1) {II};

\draw	(-2, -0.2) -- (2, -0.2); 
\node[below] at (0, -.2){III};
\node[above] at (0,.0){0};

\draw	(1, 0.2) -- (5, 0.2); 
\node[above] at (3,0.1) {IV};

\node[right] at (7, 0) {$x$-axis};
\draw (4,-0.2)--(7, -0.2);
\node [below] at (6, -0.2) {V};

\end{tikzpicture}
\label{fig: regions}
\end{figure}
{
For $M>1$ and $z_0$ the stationary point given by \eqref{stationary pt} and $\tau$ a parameter given by \eqref{tau}, we define the regions as follows:
\begin{itemize}
    \item Region I: $M^{-1}<z_0<M,\,\tau\to\infty$;
    \item Region II: $M^{-1}\leq \tau $;
    \item Region III: $\tau \leq M$;
    \item Region IV: $|z_0|\leq M, \, \tau\geq M^{-1}$;
    \item Region V: $ |z_0| > M^{-1}, \, \tau\to \infty$.
\end{itemize}
\begin{remark}
    \label{rmk: regions}
    We give some remarks on the various regions above:
    \begin{enumerate}
       \item The three main regions of interest are Regions I, III and V. In the case of focussing mKdV, they are named as oscillatory region, self-similar region and soliton region respectively, see \cite{CL19}. And the remaining two regions, Region II and Region IV can be regarded as transitions. They are treated separately because the asymptotics are calculated differently. 
        \item The calculations for Region I involves the large parameter $\tau$ and we keep $z_0$ a positive real number. In Regions II and III, $z_0$ can decay to $0$ as $t\to \infty$ while $\tau$ is bounded above. The calculations instead depend on scaling out $z_0$ and $t^{-1/3}$ respectively. The matching of the asymptotics in Region I and Region II is discussed in \cite[Section 6]{DZ93} while the matching between Region II and Region III is given in remark \ref{rmk:I overlap}.
    \end{enumerate}
\end{remark}
}
%
%

\begin{theorem}
\label{thm:main1}
Given initial data $u_0 \in H^{2,1}(\bbR)$, let $u$ be the solution to the MKdV equation
\begin{equation}
	u_t + u_{xxx} - 6u^2u_x=0 \qquad (x, t)\in (\mathbb{R}, \mathbb{R}^+)
\end{equation} given by the reconstruction formula \eqref{mkdv.q}. Let $z_0$, \eqref{stationary pt}, be the stationary point of the phase function \eqref{mkdv.phase}, and define
\begin{equation}
    \label{tau}
    \tau=z_0^3 t
\end{equation}
and 
\begin{equation}
\label{kappa}
\kappa=-\dfrac{1}{2\pi}\log(1-|r(z_0)|^2).
\end{equation}
where $r$ is defined in \eqref{reflection}.
Then we have the following asymptotics
\begin{enumerate}
\item[(i)]  In Region  $\text{I}$, 
\begin{align*}
 u(x,t) &= \left( \dfrac{\kappa}{3tz_0}\right)^{1/2}\cos \left(16tz_0^3-\kappa\log(192tz_0^3)+\phi(z_0) \right) \\
 &\quad+  \mathcal{O}\left( (z_0 t)^{-3/4}\right) 
\end{align*}
where 
\begin{align*}
\phi(z_0) &=\arg \Gamma(i\kappa)-\dfrac{\pi}{4}-\arg r(z_0)\\
        &\quad +\dfrac{1}{\pi}\int_{-z_0}^{z_0}\log\left( \dfrac{1-|r(\zeta)|^2}{1-|r(z_0)|^2} \right)\dfrac{d\zeta}{\zeta-z_0}. 
\end{align*}
\item[(ii)] In Region  $\text{II}$, 
$$u(x,t)=\dfrac{1}{(3t)^{1/3}}P\left( \dfrac{x}{ (3t)^{1/3} } \right)+\mathcal{O} \left( (z_0 t)^{-3/4} \right).$$
\item[(iii)]  In Region  $\text{III}$, 
$$u(x,t)=\dfrac{1}{(3t)^{1/3}}P\left( \dfrac{x}{ (3t)^{1/3} } \right)+\mathcal{O} \left(  t^{ -1/2 } \right).$$

\item[(iv)] In Region  $\text{IV}$, 
$$u(x,t)=\dfrac{1}{(3t)^{1/3}}P\left( \dfrac{x}{ (3t)^{1/3} } \right)+\mathcal{O} \left( (t\tau)^{-1/2}+ \dfrac{e^{-16\tau^{2/3}\eta}} {t^{ 1/2 } }\right)$$
 where we let $0<\eta< (M)^{-1/3}$.
\item[(v)] In Region  $\text{V}$, 
$$u(x,t)=\mathcal{O} \left( t^{-1}\right). $$
\end{enumerate}
In the above asymptotics for Regions $\text{II}$, $\text{III}$, $\text{IV}$,  $P$ is a solution of the Painlev\'e $\text{II}$ equation$$P''(s)-sP(s)-2P^3(s)=0$$
determined by $r(0)$. Note that given $r(z)\in H^1(\bbR)$, $r$ is defined pointwise and $r(0)$ makes sense. Also note that in all asymptotics above, the implicit constants in the remainder terms depend only on
$\norm{r}{H^{1}(\bbR)}$.
\end{theorem}
We give several remarks for the  statements above.
\begin{remark}
	In this paper, to derive asymptotics, our main focus is to establish estimates for the error terms which only depend on $\norm{r}{H^{1}(\bbR)}$. We claim that this dependence is uniform in each of the five regions defined in figure \ref{fig: regions}. All  leading order terms from the asymptotic formulae in all regions are obtained from special functions, namely parabolic cylinder functions and Painlev\'e II. For brevity, we do not repeat lengthy identical steps. We refer to Deift-Zhou \cite{DZ93} for full details.
\end{remark}


\begin{remark}
	From the view of the scattering theory, it is natural to ask if one can determine the initial data uniquely from the asymptotics of a solution. Here we point out that in our asymptotics formulae, the solution $P$
	to the Painlev\'e II equation only depends on $r\left(0\right)$,  the reflection coefficient evaluated at the origin. {For an explicit relation between $r(0)$ and $P$ the solution to \textit{ Painlev\'e} \text{II}, see \cite[p.358-p.359]{DZ03}.}
	Therefore, if one only looks at the asymptotics in regions $\text{II},$ $\text{III},$ $\text{IV},$ and $\text{V}$,
	these pieces of information are not sufficient to determine the initial
	data which produce this solution. To obtain the full information of
	the initial data, we have to go to Region $\text{I}$ from which
	one can determine the phase and modulus of the reflection coefficient
	from the formulae given by the parabolic cylinder. For more details,
	see Deift-Zhou \cite{DZ93}. In this defocussing case,
	Region $\text{I}$ is the most physically interesting. But in the focussing problem, breathers can appear in all regions. For more details, see our forthcoming article \cite{CL19}.
	
\end{remark}
{
\begin{remark}In \cite{GPR} and  \cite{HGB}, the long-time asymptotics of small solutions to the mKdV are established.  Moreover, the decay of  the spacial derivatives of the solutions  are also obtained. In \cite{HGB}, the $L^2$ estimates of error terms are estimated. In the theorem above, we only compute the asymptotics in the pointwise sense. In principle, with the analysis of the $L^2$ mapping properties of the $\dbar$ problem, we can also obtain the $L^2$ estimates for error terms but we do not pursue it here since this will require a different argument. Taking $z_0=\sqrt{\frac{-x}{12t}}$ in the leading order terms in expressions from the theorem above, the resulting formulas are the same as the leading order terms in \cite{GPR} and  \cite{HGB}. Plugging $z_0$ into the error terms above, we observe that actually in the pointwise sense, the error terms are sharper than those in  \cite{GPR} and  \cite{HGB} .
\end{remark}
}

The paper ends with a section to extend the asymptotics from Theorem \ref{thm:main1} to rougher solutions. With the uniform estimates on error terms, we apply approximation arguments to study solutions in various low regularity spaces:  $H^{1}$, $H^{1/4}$ and $L^2$ with some weights. Using  the local
 well-posedness in $H^{k}\left(\mathbb{R}\right)$ with $k\ge\frac{1}{4}$
 obtained by Kenig-Ponce-Vega see \cite{KPV}, the growth estimates
 for the $H^{k}$ norm due to Colliander-Keel-Staffilani-Takaoka-Tao
 \cite{CKSTT}, Guo \cite{Guo} and Kishimoto \cite{Kis}, and the recent advance on globally well-posedness by Harrop-Griffiths-Killip-Visan \cite{HGKV} in $H^\tau(\mathbb{R})$, $\tau>-1/2$ we employ
 a global approximation argument to extend our long-time asymptotics
 to $H^{k,2}$ with $k\geq 0$. Then we can extend the results in the previous theorem and obtain the following:
\begin{theorem}
	\label{thm:main2} For any initial data $u_{0}\in H^{k,1}\left(\mathbb{R}\right)$ with $k\geq 0$, the solution\footnote{For the precise meaning of solutions, we refer to Theorem \ref{thm:KVPlocal} and Theorem \ref{thm:gwphgkv} for details.} to the MKdV equation
	\begin{equation}
		\label{MKDV1}
		u_t + u_{xxx} - 6u^2u_x=0 \qquad (x, t)\in (\bbR, \bbR^+)
	\end{equation}
has the same  asymptotics as in our main Theorem \ref{thm:main1}.
\end{theorem}

\smallskip
 We notice that one can trace all the details in our implementing of the nonlinear steepest descent and notice that actually
 it suffices to require the weights in $x$ to be $\left\langle x\right\rangle ^{s}$
 with $s>\frac{1}{2}$. Since  for the general case,  $s>\frac{1}{2}$ is sufficient
for us to apply the Sobolev embedding and the estimate of modulus
of continuity of the reflection coefficients in  the Riemann-Hilbert
problem.  After establishing the computations for $s=1$, to get the general results for $s>\frac{1}{2}$, one just needs to use the standard analysis of Jost functions and mollifiers.

\begin{corollary}
	\label{cor:main2} For any initial data $u_{0}\in H^{k,s}\left(\mathbb{R}\right)$ with
	$s>\frac{1}{2}$ and $k\geq 0$, the solution to the MKdV equation \eqref{MKDV1}
has the same leading order asymptotics as in  main Theorem \ref{thm:main1} and the error terms can be estimated as in Remark  \ref{extention}.
\end{corollary}

Since computations from $s=1$ to general $s>\frac{1}{2}$ are quite routine, see   Cuccagna-Pelinovsky \cite{CuPe} for computations for the cubic NLS. In particular, the direct scattering of the mKdV equation is same as the NLS.
Hereinafter, for the sake of simplicity, we just focus on the case where $s=1$.

\subsection{Notations}
Let $\sigma_3$ be the third Pauli matrix:
$$\sigma_3=\twomat{1}{0}{0}{-1}$$
and define the matrix operation 
$$e^{\ad\sigma_3}A=\twomat{a}{e^{2} b}{e^{-2}c}{d}$$

	We define Fourier transforms as 
	\begin{equation}
	\hat{h}\left(\xi\right)=\mathcal{F}\left[h\right]\left(\xi\right)=\frac{1}{2\pi}\int_\bbR e^{-ix\xi}h\left(x\right)\,dx.\label{eq:FT}
	\end{equation}
	Using the Fourier transform, one can define the fractional weighted
	Sobolev spaces:
	\begin{equation}
	H^{k,s}\left(\mathbb{R}\right):=\left\{ h:\,\left\langle 1+\left|\xi\right|^{2}\right\rangle ^{\frac{k}{2}}\hat{h}\left(\xi\right)\in L^{2}\left(\mathbb{R}\right),\:\left\langle 1+x^{2}\right\rangle ^{\frac{s}{2}}h\in L^{2}\left(\mathbb{R}\right)\right\} .\label{eq:weight}
	\end{equation}
	
As usual, $"A:=B"$
or $"B=:A"$
is the definition of $A$ by means of the expression $B$. We use
the notation $\langle x\rangle=\left(1+|x|^{2}\right)^{\frac{1}{2}}$.
For positive quantities $a$ and $b$, we write $a\lesssim b$ for
$a\leq Cb$ where $C$ is some prescribed constant. Also $a\simeq b$
for $a\lesssim b$ and $b\lesssim a$. Throughout, we use $u_{t}:=\frac{\partial}{\partial_{t}}u$,
$u_{x}:=\frac{\partial}{\partial x}u$.

\subsection{Some discussion}

To finish the introduction, we highlight certain features of this paper. 

\smallskip

Firstly, compared with the analysis of the nonlinear Schr\"odinger
equation in weighted Sobolev spaces 
\cite{DZ03}, the defocussing MKdV exhibits more complicated behavior in terms of long-time asymptotics. This follows from the fact that phase function for the nonlinear Schr\"odinger equation has one
single stationary point while the phase function for the
MKdV equation has two stationary points.  The MKdV equation has the oscillatory region (Region I), the self-similar region (Region II-IV) and the soliton region (Region V), each of which has different leading order terms and error terms. These two stationary
points, due to symmetry, will lead to a real-valued solution to
the equation plus a higher order correction term. More importantly, unlike the NLS equation, where we can build parametrices directly out of the parabolic cylinder functions , for the MKdV equation, extra terms have to be eliminated before arriving at the model problem.  Thus, due to the complicated structure of the MKdV equation, we will explore some new applications of the $\bar{\partial}$-steepest descent method. We instead conjugate the jump matrices by a diagonal matrix $\mathcal{P}$ (cf. \eqref{eq: para}). Meanwhile in certain self similar regions, the two stationary
points will approach each other as $t\to\infty$. In this case, the decay in time results from a scaling factor instead of oscillation. We believe that these are new applications of the $\overline{\partial}$ nonlinear steepest descent method and can be used to treat other integrable models.

\smallskip

Secondly, we extend the asymptotics of the MKdV equation to solutions with initial
data in lower regularity spaces using a global approximation via PDE
techniques. In Deift-Zhou \cite{DZ03}, due to the $L_{t}^{4}L_{x}^{\infty}$
Strichartz estimates for the linear Schr\"odinger equation and the
conservation of the $L^{2}$ norm, the authors can globally approximate
the solution to the nonlinear Schr\"odinger equation with data in
$L^{2,1}$ using the Beals-Coifman representation of solutions
directly. Unlike the Schr\"odinger equation, the smoothing estimates and
Strichartz estimates for the Airy equation and the MKdV equation are much more
involved. For example,  one needs $L_{x}^{4}L_{t}^{\infty}$ estimate which acts
like a maximal operator. To directly work on the
solution to the MKdV equation via inverse scattering to establish the smoothing
estimates and Strichartz estimates, one needs estimates for pseudo-differential
operators with very rough symbols. To avoid these technicalities,
we first identify the {solution by inverse scattering} with the solution
given by the Duhamel formula, which we call a strong solution. The
equivalence of these two types of solutions in $H^{2,1}\left(\mathbb{R}\right)$
is not transparent since there are not enough smoothness for taking derivatives.
Relying on smoothing estimates and the bijectivity of the scattering
and inverse scattering transforms by Zhou \cite{Zhou98}
which plays the role of Plancherel theorem in Fourier analysis, we
show these two types of solutions are the same at the level of $H^{2,1}\left(\mathbb{R}\right)$
which is necessary to construct the {solutions by inverse scattering}. Since
the strong solutions by construction enjoy Strichartz estimates and
smoothing estimates, by our identification, the {solutions by inverse scattering}
also satisfy these estimates. Then we can use Strichartz estimates
and smoothing estimates to pass limits of solutions by inverse scattering
to obtain the asymptotics for rougher initial data in $H^{1,1}\left(\mathbb{R}\right)$
and $H^{\frac{1}{4},1}\left(\mathbb{R}\right)$. To illustrate the
importance of $H^{1}\left(\mathbb{R}\right)$ and $H^{\frac{1}{4}}\left(\mathbb{R}\right)$,
we note that in $H^{1}\left(\mathbb{R}\right)$, the MKdV equation has the
energy conservation. On the other hand,  $H^{\frac{1}{4}}\left(\mathbb{R}\right)$
is the optimal space to use iterations to construct the solution to
the MKdV equation. { With the recent advances of globally well-posedness of mKdV equations, \cite{HGKV}, with appropriate notations of solutions, our results can be naturally extended to solutions with initial data in the  weighted  $L^2(\mathbb{R)}$ space. }
For details of the proof, we refer the reader to Section \ref{sec: approx}.

Finally, we give a general description of the  derivation of the long-time asymptotics and performing nonlinear steepest descent. The major part of this paper is devoted to the study of the Region I whose leading behavior is given by parabolic cylinder functions.

\smallskip
The first step (Section \ref{sec:prep}),  is to conjugate 
the matrix $m$ with a scalar function $\delta(z)$  
which solves the  scalar model RHP Problem  \ref{prob:RH.delta}. This conjugation leads to a new RHP, Problem \ref{prob:DNLS.RHP1}. The purpose of this is to prepare for the lower/upper factorization of the jump matrix on the part of the real axis between two stationary point. This is needed in the contour deformation described in Section \ref{sec:mixed}.

\smallskip

The second  step
( Section \ref{sec:mixed}) is a  deformation of contour from $\bbR$ to a new
contour  $\Sigma^{(2)}$ (Figure \ref{fig:contour-2}). It is to guarantee that the phase factors in the jump matrix \eqref{DNLS.V1}
have the desired exponential decay in time along the deformed contours. Inevitably this transformation will results in certain non-analyticity in the sectors $\Omega_1\cup \Omega_3 \cup \Omega_4 \cup \Omega_6\cup\Omega_7^\pm \cup\Omega_8^\pm$,  which leads to
a mixed $\dbar$--RHP-problem,  Problem \ref{prob:DNLS.RHP.dbar}.

\smallskip

The  third step is a `factorization' of  $m^{(2)}$ 
in the form $m^{(2)} =  m^{(3)} m^{\RHP}$ where $m^{\RHP}$ is the solution of a 
localized RHP,
Problem \ref{MKDV.RHP.local}, and $m^{(3)}$ 
a solution of $\bar\partial$ problem, Problem \ref{prob:DNLS.dbar}. 
The term "localized" means the reflection coefficient $r(z)$ is fixed at $\pm z_0$ along the deformed contours. We then solve this localized RHP whose solution is given by parabolic cylinder functions. Since we have to separate the contribution from two stationary points $\pm z_0$, some error terms appear alongside and their decay rate are estimated.

\smallskip

The fourth step (Section \ref{sec:dbar}) is the solution of the $\dbar$-problem through solving an integral equation. The integral operator has small $L^\infty$-norm  at  large $t$ allowing the use of Neumann series. The contribution of this $\dbar$-problem is another higher order error term.

\smallskip

The fifth step (Section \ref{sec:large-time})  is to  group together all the previous transformations to  derive the long time asymptotics of the solution of the MKdV equation in Region I,  using the large-$z$ behavior of the RHP solutions.  These five steps above are more or less standard, during the proof of which we mainly follow the outline of \cite{LPS}.

\smallskip

The sixth step is the study of Region II-V. The leading order term in these region are given by a solution to the Painlev\'e II equation and error estimates are obtained from scaling.

\subsection{Acknowledgement}
We would like to thank Professor Jean-Claude Saut for pointing out the references  \cite{S79}, \cite{ST} and  \cite{T69}. We are very grateful to  the anonymous referees whose detailed comments improve the presentation of the paper significantly.

\section{Conjugation}
\label{sec:prep}
We introduce a new matrix-valued function
\begin{equation}
\label{m1}
m^{(1)}(z;x,t) = m(z;x,t) \delta(z)^{-\sigma_3} 
\end{equation}
where $\delta(z)$  solves 
the scalar RHP 
Problem \ref{prob:RH.delta} below:

\begin{problem}
\label{prob:RH.delta}
Given $\pm z_0 \in \bbR$ and $r \in H^{1}(\bbR)$, find a scalar function 
$\delta(z) = \delta(z; z_0)$, analytic for
$z \in \bbC \setminus [-z_0, z_0]$ with the following properties:
\begin{enumerate}
\item		$\delta(z) \rarr 1$ as $z \rarr \infty$,
\item		$\delta(z)$ has continuous boundary values $\delta_\pm(z) =\lim_{\eps \darr 0} \delta(z \pm i\eps)$ for $z \in (-z_0, z_0)$,
\item		$\delta_\pm$ obey the jump relation
			$$ \delta_+(z) = \begin{cases}
											\delta_-(z)  \left(1 - \left| r(z) \right|^2 \right),	&	 z\in (-z_0, z_0)\\
											\delta_-(z), &	z \in \bbR\setminus (-z_0, z_0)
										\end{cases}
			$$
\end{enumerate}
\end{problem}

\begin{lemma}
\label{lemma:delta}
Suppose $r \in H^{1}(\bbR)$ and that $\kappa(s)$ is given by \eqref{kappa}. Then
\begin{itemize}
\item[(i)]		Problem \ref{prob:RH.delta} has the unique solution
\begin{equation}
\label{RH.delta.sol}
\delta(z) = \left( \dfrac{z-z_0}{z+z_0} \right)^{i\kappa} e^{\chi(z)}  
\end{equation}
where $\kappa$ is given by equation \eqref{kappa} and
\begin{equation}
\label{chi}
\chi(z)=\dfrac{1}{2\pi i}\int_{-z_0}^{z_0}\log\left( \dfrac{1-|r(\zeta)|^2}{1-|r(z_0)|^2} \right)\dfrac{d\zeta}{\zeta-z}
\end{equation}
$$ \left( \dfrac{z-z_0}{z+z_0} \right)^{i\kappa}=\exp\left( i\kappa \left( \log\left\vert \dfrac{z-z_0}{z+z_0} \right\vert +i\arg(z-z_0)-i\arg(z+z_0) \right) \right)$$
here we choose the branch of the logarithm with $-\pi  < \arg(z) < \pi$. 
\bigskip
\item[(ii)]
\begin{equation*}
\delta(z) =(\overline{\delta(\zbar)})^{-1}=\overline{\delta(-\zbar)}
\end{equation*}
\bigskip
\item[(iii)]
For $z\in\bbR$, $|\delta_\pm(z)|<\infty$; for $z\in \bbC\setminus\bbR$, $|\delta^{\pm 1}(z)|<\infty$

\bigskip

\item[(iv)]Along any ray of the form $\pm z_0+ e^{i\phi}\bbR^+$ with $0<\phi<\pi$ or $\pi < \phi < 2\pi$, 
				
				$$ 
						 \left| \delta(z) - \left( \dfrac{z-z_0}{z+z_0} \right)^{i\kappa} e^{\chi(\pm z_0)}  \right| 
						 		\leq C_r
						 |z \mp z_0|^{1/2}.$$
				The implied constant depends on $r$ through its $H^{1}(\bbR)$-norm  
				and is independent of $\pm z_0\in \bbR$.

\end{itemize}
\end{lemma}

\begin{proof}
The proofs of (i)-(iii) can be found in  \cite{DZ93}.  To establish (iv), we first note that 
$$ \left\vert \left( \dfrac{z-z_0}{z+z_0} \right)^{i\kappa}\right\vert \leq e^{\pi \kappa}.$$
To bound the difference $e^{\chi(z)}-e^{\chi(\pm z_0)}$, notice that
\begin{align*}
\left\vert e^{\chi(z)}-e^{\chi(\pm z_0)}\right\vert &\leq\left\vert e^{\chi(\pm z_0)}\right\vert 
\left\vert e^{\chi(z)-\chi(\pm z_0)}-1 \right\vert\\
     &\lesssim \left\vert \int_0^1 \dfrac{d}{ds} e^{s(  \chi(z)-\chi(\pm z_0) )ds } \right\vert\\
     &\lesssim  |z \mp z_0|^{1/2} \sup_{0\leq s\leq 1} \left\vert e^{s(  \chi(z)-\chi(\pm z_0) )}\right\vert\\
     &\lesssim |z \mp z_0|^{1/2}
    \end{align*}
    where the third inequality follows from \cite[Lemma 23]{BDT88}.
\end{proof}

It is straightforward to check that if $m(z;x,t)$ solves Problem \ref{prob:DNLS.RH0}, then the new matrix-valued function $m^{(1)}(z;x,t)=m(z;x,t)\delta(z)^{\sigma_3}$ is the solution to the  following RHP.  

\begin{problem}
\label{prob:DNLS.RHP1}
Given $r \in H^{1,0}(\bbR)$, find a matrix-valued function $m^{(1)}(z;x,t)$ on $\bbC \setminus \bbR$ with the following properties:
\begin{enumerate}
\item		$m^{(1)}(z;x,t) \rarr I$ as $|z| \rarr \infty$,
\item		$m^{(1)}(z;x,t)$ is analytic for $z \in  \bbC \setminus \bbR$
			with continuous boundary values
			$$m^{(1)}_\pm(z;x,t) 
				= \lim_{\eps \darr 0} m^{(1)}(z+i\eps;x,t).$$
\item		The jump relation $$m^{(1)}_+(z;x,t)=m^{(1)}_-(z;x,t)	
			e^{-i\theta\ad\sigma_3}v^{(1)}(z)$$
			 holds,
			 where $$v^{(1)}(z) = \delta_-(z)^{\sigma_3} v(z) \delta_+(z)^{-\sigma_3}.$$
			 \noindent
			The jump matrix $e^{-i\theta\ad\sigma_3} v^{(1)} $ is factorized as 
			\begin{align}
			\label{DNLS.V1}
			e^{-i\theta\ad\sigma_3}v^{(1)}(z)	=
			\begin{cases}
					\Twomat{1}{0}{\dfrac{\delta_-^{-2}  r}{1-|r|^2}  e^{2i\theta}}{1}
					\Twomat{1}{-\dfrac{\delta_+^2 \rbar}{1- |r|^2} e^{-2i\theta}}{0}{1},
						& z \in (-z_0, z_0),\\
						\\
						\Twomat{1}{-\rbar \delta^2 e^{-2i\theta}}{0}{1}
					\Twomat{1}{0}{r \delta^{-2} e^{2i\theta}}{1},
						& z \in(-\infty, -z_0)\cup (z_0,\infty) .
			\end{cases}
			\end{align}
		
\end{enumerate}
\end{problem}

\section{Contour deformation}
\label{sec:mixed}

We now perform contour deformation on Problem \ref{prob:DNLS.RHP1}, following the standard procedure outlined in \cite[Section 4]{LPS}.
Since the phase function \eqref{mkdv.phase} has two critical points
at $\pm z_0$, our new contour is chosen to be
\begin{equation}
\label{new-contour}
\Sigma^{(2)} = \Sigma_1 \cup \Sigma_2 \cup \Sigma_3 \cup \Sigma_4\cup  \Sigma_5 \cup \Sigma_6 \cup \Sigma_7 \cup \Sigma_8
\end{equation}
shown in Figure \ref{new-contour} and consists of rays of the form $\pm z_0+ e^{i\phi}\bbR^+$
where $\phi = \pi/4, 3\pi/4,5\pi/4, 7\pi/4$. 

\begin{figure}[H]
\caption{Deformation from $\mathbb{R}$ to $\Sigma^{(2)}$}
\vskip 15pt
\begin{tikzpicture}[scale=0.75]
\draw[dashed] 					(0, 2) -- (0,-2);	
\draw[dashed] 					(-6,0) -- (6,0);								
\draw[thick]		(2,0) -- (4,2);								
\draw[->,thick,>=stealth] 		(5, 3) -- (4,2);
\draw[thick] 	(-5,3) -- (-4,2);							
\draw[->,thick,>=stealth]  	(-2,0) -- (-4,2);
\draw[->,thick,>=stealth]		(-5,-3) -- (-4,-2);							
\draw[thick]						(-4,-2) -- (-2,0);
\draw[thick,->,>=stealth]		(2,0) -- (4,-2);								
\draw[thick]						(4,-2) -- (5,-3);
\draw[thick]	(-2,0) -- (-1,1);								
\draw[thick,->,>=stealth] (0,2) -- (-1, 1);
\draw[thick,->,>=stealth]		(-2,0) -- (-1,-1);								
\draw[thick]						(-1,-1) -- (0, -2);
\draw[thick]			(0,2) -- (1,1);								
\draw[thick,->,>=stealth]	(2,0) -- (1, 1);
\draw[thick,->,>=stealth]		(0,-2) -- (1,-1);								
\draw[thick]						(1,-1) -- (2, -0);
\draw	[fill]							(-2,0)		circle[radius=0.1];	
\draw	[fill]							(2,0)		circle[radius=0.1];
\draw							(0,0)		circle[radius=0.1];
\node[below] at (-2,-0.1)			{$-z_0$};
\node[below] at (2,-0.1)			{$z_0$};
\node[right] at (5,3)					{$\Sigma_1$};
\node[left] at (-5,3)					{$\Sigma_2$};
\node[left] at (-5,-3)					{$\Sigma_3$};
\node[right] at (5,-3)				{$\Sigma_4$};
\node[left] at (-1,1.2)					{$\Sigma_5$};
\node[left] at (-1,-1.2)					{$\Sigma_7$};
\node[right] at (1,1.2)					{$\Sigma_6$};
\node[right] at (1,-1.2)					{$\Sigma_8$};
\node[right] at (3.5,1)				{$\Omega_1$};
\node[above] at (0,2)			{$\Omega_2$};
\node[left] at (-3.5,1)				{$\Omega_3$};
\node[left] at (-3.5,-1)				{$\Omega_4$};
\node[below] at (0,-2)			{$\Omega_5$};
\node[right] at (3.5,-1)				{$\Omega_6$};
\node[above] at (0.8, 0.2)			{$\Omega_7^+$};
\node[below] at (0.8,-0.2)				{$\Omega_8^+$};
\node[above] at (-0.8, 0.2)			{$\Omega_7^-$};
\node[below] at (-0.8,-0.2)				{$\Omega_8^-$};
\end{tikzpicture}
\label{fig:contour-2}
\end{figure}

We now introduce another matrix-valued function $m^{(2)}$:
$$ m^{(2)}(z) = m^{(1)}(z)  \calR^{(2)}(z). $$
Here $\calR^{(2)}$ is chosen to remove the jump on the real axis and brings about new analytic jump matrices with the desired exponential decay 
along the contour $\Sigma^{(2)}$. Straight forward computation gives
\begin{align*}
m^{(2)}_+	&=m^{(1)}_+ \calR^{(2)}_+ \\
				&= m^{(1)}_- \left( e^{-i\theta\ad\sigma_3} v^{(1)} \right) \calR^{(2)}_+ \\
				&= m^{(2)}_- \left(\calR^{(2)}_-\right)^{-1}
						\left( e^{-i\theta\ad\sigma_3} v^{(1)} \right) \calR^{(2)}_+.
\end{align*}
We want to make sure that the following condition is satisfied
$$ 
(\calR^{(2)}_-)^{-1} \left( e^{-i\theta\ad\sigma_3} v^{(1)} \right) \calR^{(2)}_+ = I
$$
where $\calR_\pm^{(2)}$ are the boundary values of $\calR^{(2)}(z)$ as $\pm \Imag(z) \darr 0$. In this case the jump matrix associated to $m^{(2)}_\pm$ will be the identity matrix on $\bbR$ .

From the signature table \cite[Figure 0.1]{DZ93} we find that the function $e^{2i\theta}$ is exponentially decreasing on $\Sigma_3$  $\Sigma_4$, $\Sigma_5$, $\Sigma_6$  and increasing on $\Sigma_1$, $\Sigma_2$, $\Sigma_7$, $\Sigma_8$ away from the stationary point while the reverse is true of $e^{-2i\theta}$. 
Letting
\begin{equation} \label{eta}
\eta(z; z_0) = \left( \dfrac{ z-z_0}{z+z_0} \right)^{i\kappa}, 
\end{equation}
we define $\calR^{(2)}$ as follows (Figure \ref{fig R-2+}-\ref{fig R-2-}): 
 the functions $R_1$, $R_3$, $R_4$, $R_6$, $R_7^+$, $R_8^+$, $R_7^-$, $R_{8}^-$ satisfy 
\begin{align}
\label{R1}
R_1(z)	&=	\begin{cases}
						-{r(z)} \delta(z)^{-2}			
								&	z \in (z_0,\infty)\\[10pt]
						-{r(z_0 )} e^{-2\chi(z_0)} \eta(z; z_0)^{-2}
								&	z	\in \Sigma_1,
					\end{cases}\\[10pt]
\label{R3}
R_3(z)	&=	\begin{cases}
						-{r(z)} \delta(z)^{-2}		
								&	z \in (-\infty, -z_0)\\[10pt]
						-{r(-z_0 )} e^{-2\chi(-z_0)} \eta(z; z_0)^{-2}
								&	z	\in \Sigma_2,
					\end{cases}\\[10pt]
\label{R4}
R_4(z)	&=	\begin{cases}
						-\overline{r(z)} \delta(z)^{2}			
								&	z \in (-\infty, -z_0)\\[10pt]
						-\overline{r(-z_0)} e^{2\chi(-z_0)} \eta(z; z_0)^{2}
								&	z	\in \Sigma_3,
					\end{cases}\\[10pt]
\label{R6}
R_6(z)	&=	\begin{cases}
						-\overline{r(z)} \delta(z)^{2}			
								&	z \in (-\infty, -z_0)\\[10pt]
						-\overline{r(z_0)} e^{2\chi(z_0)} \eta(z; z_0)^{2}
								&	z	\in \Sigma_4,
					\end{cases}
\end{align}

\begin{align}
\label{R7+}
R_7^+(z)	&=	\begin{cases}
						\dfrac{\delta_-^{-2}(z)r(z)}{1-|r(z)|^2}		
								& z \in (-z_0, z_0)\\[10pt]
					\dfrac{e^{-2\chi(z_0)} \eta(z; z_0)^{-2} r(z_0)}{1-|r(z_0)|^2} \quad
                                  & z \in \Sigma_6,
					\end{cases}
					\\[10pt]
\label{R8+}
R_8^+(z)	&=	\begin{cases}
						\dfrac{\delta_+^{2}(z)  \overline{r(z)} }{1-|r(z)|^2}		
								& z \in (-z_0, z_0)\\[10pt]
					{\dfrac{e^{2\chi(z_0)} \eta(z; z_0)^{2} \overline{r(z_0)} }{1-|r(z_0)|^2}}
						& z \in \Sigma_8,
					\end{cases}
					\\[10pt]
\label{R7-}
R_7^-(z)	&=	\begin{cases}
						\dfrac{\delta_-^{-2}(z)r(z)}{1-|r(z)|^2}		
								& z \in (-z_0, z_0)\\[10pt]
					\dfrac{e^{-2\chi(-z_0)} \eta(z; z_0)^{-2} r(-z_0)}{1-|r(-z_0)|^2} \quad
                                  & z \in \Sigma_5,
					\end{cases}
					\\[10pt]
\label{R8-}
R_{8}^-(z)	&=	\begin{cases}
						\dfrac{\delta_+^{2}(z)  \overline{r(z)} }{1-|r(z)|^2}		
								& z \in (-z_0, z_0)\\[10pt]
					{\dfrac{e^{2\chi(-z_0)} \eta(z; z_0)^{2} \overline{r(-z_0)} }{1-|r(-z_0)|^2}}
						& z \in \Sigma_7.
					\end{cases}
\end{align}

{
\SixMatrix{The Matrix  $\calR^{(2)}$ for Region I, near $z_0$}{fig R-2+}
	{\twomat{1}{0}{R_1 e^{2i\theta}}{1}}
	{\twomat{1}{R_7^+ e^{-2i\theta}}{0}{1}}
	{\twomat{1}{0}{R_8^+ e^{2i\theta}}{1}}
	{\twomat{1}{R_6 e^{-2i\theta}}{0}{1}}
}

{
\sixmatrix{The Matrix  $\calR^{(2)}$ for Region I, near $-z_0$}{fig R-2-}
	{\twomat{1}{R_{7}^- e^{-2i\theta}}{0}{1}}
	{\twomat{1}{0}{R_3e^{2i\theta}}{1}}
	{\twomat{1}{R_4e^{-2i\theta}}{0}{1}}
	{\twomat{1}{0} {R_{8}^- e^{2i\theta}}{1}}
}

Each $R_i(z)$ in $\Omega_i$ is constructed in such a way that the jump matrices on the contour and $\dbar R_i(z)$ along with along with their relevant exponentials enjoys the property of exponential decay as $t\to \infty$.
We formulate Problem \ref{prob:DNLS.RHP1} into a mixed RHP-$\dbar$ problem. In the following sections we will separate this mixed problem into a localized RHP and a pure $\dbar$ problem whose long-time contribution to the asymptotics of $u(x,t)$ is of higher order than the leading term.

The following lemma (\cite[Proposition 2.1]{DM08}) will be used in the error estimates of 
$\bar \partial$-problem in Section \ref{sec:dbar}.

We first denote the entries that appear in \eqref{R1}--\eqref{R8-} by
\begin{align*}
p_1(z)=p_3(z)	&=	-r(z).	&
p_4(z)=p_6(z)	&=	- \overline{r(z)},&\\
p_{7^-}(z)=p_{7^+}(z)	&=	\dfrac{r(z)}{1-|r(z)|^2},& p_{8^-}(z)=p_{8^+}(z)	&= \dfrac{ \overline{r(z)}}{1 -|r(z)|^2}.
\end{align*}

\begin{lemma}
\label{lemma:dbar.Ri}
Suppose $r \in H^{1}(\bbR)$. There exist functions $R_i$ on $\Omega_i$, $i=1,3,4,6,7^\pm,8^\pm$ satisfying \eqref{R1}--\eqref{R8-}, so that
$$ 
|\dbar R_i(z)| \lesssim
	 |p_i'(\Real(z))| + |z-\xi|^{-1/2} , 	
				z \in \Omega_i
			$$ 
where $\xi=\pm z_0$  and the implied constants are uniform for $r $ in a bounded subset of $H^{1}(\bbR)$.
\end{lemma}

\begin{proof}
We only prove the lemma for $R_1$. Define $f_1(z)$ on $\Omega_1$ by
$$ f_1(z) = p_1(z_0) e^{-2\chi(z_0)} \eta(z; z_0)^{-2} \delta(z)^{2} $$
and let
\begin{equation}
\label{interpol}
\ R_1(z) = \left( f_1(z) + \left[ p_1(\Real(z)) - f_1(z) \right] \mathcal{K}(\phi) \right) \delta(z)^{-2} 
\end{equation}
where $\phi = \arg (z-\xi)$ and $\mathcal{K}$ is a smooth function on $(0, \pi/4)$ with
\begin{equation}
\label{cal-K}
\mathcal{K}(\phi)=
	\begin{cases}
			1
			&	z\in [0, \pi/12], \\
			0
			&	z \in[\pi/6, \pi/4]
	\end{cases}
\end{equation}

 It is easy to see that $R_1$ as constructed has the boundary values \eqref{R1}.
Writing $z-z_0= \rho e^{i\phi}$, we have
$$ \dbar = \frac{1}{2}\left( \frac{\dee}{\dee x} + i \frac{\dee}{\dee y} \right)
			=	\frac{1}{2} e^{i\phi} \left( \frac{\dee}{\dee \rho} + \frac{i}{\rho} \frac{\dee}{\dee \phi} \right).
$$
We calculate
$$ 
\dbar R_1 (z) =  \frac{1}{2}  p_1'(\Real z) \mathcal{K}(\phi)  ~ \delta(z)^{-2} -
		\left[ p_1(\Real z) - f_1(z) \right]\delta(z)^{-2}  \frac{ie^{i\phi}}{|z-\xi|}  \mathcal{K}'(\phi) . 
$$
It follows from Lemma \ref{lemma:delta} (iv) that
$$ 
 \left|\left( \dbar R_1 \right)(z)  \right| \lesssim
|p_1'(\Real z)| + |z-\xi|^{-1/2}
$$
where the implied constants depend on $\norm{r}{H^{1}}$ and the cutoff function $\mathcal{K}$. 
The estimates in the remaining sectors are identical.
\end{proof}

The unknown $m^{(2)}$ satisfies a mixed $\dbar$-RHP. We first identify the jumps of $m^{(2)}$ along the contour $\Sigma^{(2)}$. Recall that $m^{(1)}$ is analytic along the contour,  the jumps are determined entirely by 
$\mathcal{R}^{(2)}$, see \eqref{R1}--\eqref{R8-}. Away from $\Sigma^{(2)}$, using the triangularity of $\mathcal{R}^{(2)}$, we   have that 
\begin{equation}
\label{N2.dbar}
 \dbar m^{(2)} = m^{(2)} \left( \calR^{(2)} \right)^{-1} \dbar \calR^{(2)} = m^{(2)} \dbar \calR^{(2)} 
 \end{equation}
 
 \begin{remark}
Note that the interpolation defined through \eqref{interpol} introduce new jump on $\Sigma^{'(2)}_9$ of Figure \ref{fig:contour} with jump matrix given by 
\begin{equation}
\label{v_9}
v_9(z)=\begin{cases} I, & z\in \left(-iz_0 \tan(\pi/12), iz_0 \tan(\pi/12) \right)\\
\\
        \unitupper{ (R_7^--R_7^+)e^{-2i\theta} },  & z\in \left(iz_0 \tan(\pi/12), iz_0 \right) \\
        \\
         \unitlower{ (R_8^--R_8^+)e^{2i\theta} },  & z\in \left(-iz_0 , -iz_0 \tan(\pi/12), \right). 
\end{cases}
\end{equation}
But $v_9$ is exponentially small due to the construction of $\mathcal{K}(\phi)$ in \eqref{cal-K}.
\end{remark}
\begin{figure}[H]
\caption{$\Sigma'^{(2)}$}
\vskip 15pt
\begin{tikzpicture}[scale=0.65]
\draw	[->,thick,>=stealth] 	(2,0) -- (4,2);								
\draw	[thick]	(5, 3) -- (4,2);
\draw  [->,thick,>=stealth] 	(-5,3) -- (-4,2);							
\draw [thick]	(-2,0) -- (-4,2);
\draw[->,thick,>=stealth]		(-5,-3) -- (-4,-2);							
\draw[thick]						(-4,-2) -- (-2,0);
\draw[thick,->,>=stealth]		(2,0) -- (4,-2);								
\draw[thick]						(4,-2) -- (5,-3);
\draw[thick,->,>=stealth]	(-2,0) -- (-1,1);								
\draw  [thick]    (0,2) -- (-1, 1);
\draw[thick,->,>=stealth]		(-2,0) -- (-1,-1);								
\draw[thick]						(-1,-1) -- (0, -2);
\draw	[thick,->,>=stealth]		(0,2) -- (1,1);								
\draw	[thick]  (2,0) -- (1, 1);
\draw[thick,->,>=stealth]		(0,-2) -- (1,-1);								
\draw[thick]						(1,-1) -- (2, -0);
\draw	[fill]							(-2,0)		circle[radius=0.1];	
\draw	[fill]							(2,0)		circle[radius=0.1];
\draw[->,thick,>=stealth] 		(0, -2) -- (0,0);
\draw[thick]			(0,0) -- (0,2);		
\node[below] at (-2,-0.1)			{$-z_0$};
\node[below] at (2,-0.1)			{$z_0$};
\node[right] at (5,3)					{$\Sigma'^{(2)}_1$};
\node[left] at (-5,3)					{$\Sigma'^{(2)}_2$};
\node[left] at (-5,-3)					{$\Sigma'^{(2)}_3$};
\node[right] at (5,-3)				{$\Sigma'^{(2)}_4$};
\node[left] at (-1,1.2)					{$\Sigma'^{(2)}_5$};
\node[left] at (-1,-1.2)					{$\Sigma'^{(2)}_7$};
\node[right] at (1,1.2)					{$\Sigma'^{(2)}_6$};
\node[right] at (1,-1.2)					{$\Sigma'^{(2)}_8$};
\node[right] at (0,0)              {$\Sigma'^{(2)}_9$};
\end{tikzpicture}
\label{fig:contour}
\end{figure}
Now we arrive at the following Riemann-Hilbert-$\overline{\partial}$ problem
\begin{problem}
\label{prob:DNLS.RHP.dbar}
Given $r \in H^{1}(\bbR)$, find a matrix-valued function $m^{(2)}(z;x,t)$ on $\bbC \setminus \Sigma^{'(2)} $ with the following properties:
\begin{enumerate}
\item		$m^{(2)}(z;x,t) \rarr I $ as $|z| \rarr \infty$ in $\bbC \setminus \Sigma^{'(2)}$,
\item		$m^{(2)}(z;x,t)$ is continuous for $z \in  \bbC \setminus \Sigma^{'(2)}$
			with continuous boundary values 
			$m^{(2)}_\pm(z;x,t) $
			(where $\pm$ is defined by the orientation in Figure \ref{fig:contour})
\item		The jump relation $m^{(2)}_+(z;x,t)=m^{(2)}_-(z;x,t)	
			e^{-i\theta\ad\sigma}v^{(2)}(z)$ holds, where
			$e^{-i\theta\ad\sigma}v^{(2)}(z)	$ is given in Figure \ref{fig:jumps-1}-\ref{fig:jumps-2} and \eqref{v_9}.
\item		The equation 
			$$
			\dbar m^{(2)} = m^{(2)} \, \dbar \calR^{(2)}
			$$ 
			holds in $\bbC \setminus \Sigma^{'(2)}$, where
			$$
			\dbar \calR^{(2)}=
				\begin{doublecases}
					\Twomat{0}{0}{(\dbar R_1) e^{2i\theta}}{0}, 	& z \in \Omega_1	&&
					\Twomat{0}{(\dbar R_7^+)e^{-2i\theta}}{0}{0}	,	& z \in \Omega_7^+	\\
					\\
					\Twomat{0}{0}{(\dbar R_8^+)e^{2i\theta}}{0},	&	z \in \Omega_8^+	&&
					\Twomat{0}{(\dbar R_6)e^{-2i\theta}}{0}{0}	,	&	z	\in \Omega_6	 \\
					\\
					\Twomat{0}{0}{(\dbar R_3) e^{2i\theta}}{0}, 	& z \in \Omega_3	&&
					\Twomat{0}{(\dbar R_4)e^{-2i\theta}}{0}{0}	,	& z \in \Omega_4	\\
					\\
					\Twomat{0}{0}{(\dbar R_8^-)e^{2i\theta}}{0},	&	z \in \Omega_8^-	&&
					\Twomat{0}{(\dbar R_7^-)e^{-2i\theta}}{0}{0}	,	&	z	\in \Omega_7^- \\
					\\
					0	&\hspace{-5pt} z\in \Omega_2\cup\Omega_5	
				\end{doublecases}
			$$
\end{enumerate}
\end{problem}

The following picture is an illustration of the jump matrices of RHP Problem \ref{prob:DNLS.RHP.dbar}.

\begin{figure}[H]
\caption{Jump Matrices  $v^{(2)}$  for $m^{(2)}$ near $z_0$}
\vskip 15pt
\begin{tikzpicture}[scale=0.75]
\draw[dashed] 				(-6,0) -- (6,0);							
\draw [thick]	(0,0) -- (1.5,1.5);						
\draw[->,thick,>=stealth] (3,3) -- (1.5, 1.5);
\draw [thick]	(-3,3) -- (-1.5,1.5);					
\draw	[->,thick,>=stealth]		(0,0)--(-1.5,1.5);
\draw[->,thick,>=stealth]	(-3,-3) -- (-1.5,-1.5);					
\draw[thick]					(-1.5,-1.5) -- (0,0);
\draw[->,thick,>=stealth]	(0,0) -- (1.5,-1.5);					
\draw[thick]					(1.5,-1.5) -- (3,-3);
\draw[fill]						(0,0)	circle[radius=0.075];		
\node [below] at  			(0,-0.15)		{$z_0$};
\node[right] at					(3.2,3)		{$\unitlower{R_1 e^{2i\theta}}$};
\node[left] at					(-3.2,3)		{$\unitupper{R_7^+ e^{-2i\theta}}$};
\node[left] at					(-3.2,-3)		{$\unitlower{R_8^+ e^{2i\theta}}$};
\node[right] at					(3.2,-3)		{$\unitupper{R_6 e^{-2i\theta}}$};
\node[above] at 				(1.4,1.6)	{$-$};
\node[below] at				(1.7,1.4)	{$+$};
\node[above] at				(-1.4,1.6)	{$-$};
\node[below] at				(-1.7,1.4)	{$+$};
\node[above] at 				(-1.7,-1.4)	{$+$};
\node[below] at				(-1.4,-1.6)	{$-$};
\node[above] at				(1.7,-1.4)	{$+$};
\node[below] at				(1.4,-1.6)	{$-$};
\node[left] at					(2.5,3)		{$\Sigma_1$};
\node[right] at					(-2.5,3)		{$\Sigma_6$};
\node[right] at					(-2.5,-3)		{$\Sigma_8$};
\node[left] at					(2.5,-3)		{$\Sigma_4$};
\end{tikzpicture}
\label{fig:jumps-1}
\end{figure}

\begin{figure}[H]
\caption{Jump Matrices  $v^{(2)}$  for $m^{(2)}$ near $-z_0$}
\vskip 15pt
\begin{tikzpicture}[scale=0.75]
\draw[dashed] 				(-6,0) -- (6,0);							
\draw [thick]	(0,0) -- (1.5,1.5);						
\draw	[->,thick,>=stealth] 	(3,3)--(1.5,1.5);
\draw [thick]		(-3,3) -- (-1.5,1.5);					
\draw[->,thick,>=stealth]		 (0,0)--(-1.5,1.5); 
\draw[->,thick,>=stealth]	(-3,-3) -- (-1.5,-1.5);					
\draw[thick]					(-1.5,-1.5) -- (0,0);
\draw[->,thick,>=stealth]	(0,0) -- (1.5,-1.5);					
\draw[thick]					(1.5,-1.5) -- (3,-3);
\draw[fill]						(0,0)	circle[radius=0.075];		
\node [below] at  			(0,-0.15)		{$-z_0$};
\node[right] at					(3.2,3)		{$\unitupper{R_7^- e^{-2i\theta}}$};
\node[left] at					(-3.2,3)		{$\unitlower{R_3 e^{2i\theta}}$};
\node[left] at					(-3.2,-3)		{$\unitupper{R_4 e^{-2i\theta}}$};
\node[right] at					(3.2,-3)		{$\unitlower{R_{8}^- e^{2i\theta}}$};
\node[above] at 				(1.4,1.6)	{$-$};
\node[below] at				(1.7,1.4)	{$+$};
\node[above] at				(-1.4,1.6)	{$-$};
\node[below] at				(-1.7,1.4)	{$+$};
\node[above] at 				(-1.7,-1.4)	{$+$};
\node[below] at				(-1.4,-1.6)	{$-$};
\node[above] at				(1.7,-1.4)	{$+$};
\node[below] at				(1.4,-1.6)	{$-$};
\node[left] at					(2.5,3)		{$\Sigma_5$};
\node[right] at					(-2.5,3)		{$\Sigma_2$};
\node[right] at					(-2.5,-3)		{$\Sigma_3$};
\node[left] at					(2.5,-3)		{$\Sigma_7$};
\end{tikzpicture}
\label{fig:jumps-2}
\end{figure}

%
%

\section{The Localized Riemann-Hilbert Problem}
\label{sec:local}
We perform the following factorization of $m^{(2)}$:
\begin{equation}
\label{factor-LC}
m^{(2)} = m^{(3)} m^{\RHP}.
\end{equation}
Here we require that $m^{(3)} $ to be the solution of the pure $\dbar$-problem, hence no jump, and $ m^{\RHP}$ solution of the localized  RHP Problem \ref{MKDV.RHP.local} below
with the jump matrix $v^\RHP=v^{(2)}$.  The current section focuses on $ m^{\RHP}$.
\begin{problem}
\label{MKDV.RHP.local}
Find a $2\times 2$ matrix-valued function $m^\RHP(z; x,t)$, analytic on $\bbC \setminus \Sigma'^{(2)}$ (See Figure \ref{fig:contour}) ,
with the following properties:
\begin{enumerate}
\item	$m^\RHP(z;x,t) \rarr I$ as $|z| \rarr \infty$ in $\bbC \setminus \Sigma'^{(2)}$, where $I$ is the $2\times 2$ identity matrix,
\item	$m^\RHP(z; x,t)$ is analytic for $z \in \bbC \setminus \Sigma'^{(2)}$ with continuous boundary values $m^\RHP_\pm$
		on $\Sigma'^{(2)}$,
\item	The jump relation $m^\RHP_+(z;x,t) = m^\RHP_-(z; x,t ) v^\RHP(z)$ holds on $\Sigma'^{(2)}$, where
		\begin{equation*}	
		v^\RHP(z) =	v^{(2)}(z).
		\end{equation*}
\end{enumerate}
\end{problem}

\subsection{construction of the parametrix}

For some fixed $\rho>0$, we define 
\begin{align*}
L_\rho &=\lbrace z: z=z_0+u  e^{3i\pi/4},  \rho \leq u\leq \sqrt{2}z_0 \rbrace\\
           &\cup \lbrace z: z=z_0+u  e^{i\pi/4},  u \geq \rho \rbrace\\
           & \cup \lbrace z: z=-z_0+u  e^{i\pi/4},  \rho \leq u\leq \sqrt{2} z_0 \rbrace\\
           & \cup \lbrace z: z=-z_0+u  e^{3i\pi/4},  u \geq \rho \rbrace\\
           \Sigma'&=\Sigma'^{(2)}\setminus (L_\rho \cup \overline{L_\rho}\cup \Sigma'^{(2)}_9). 
\end{align*}

\begin{figure}[H]
\caption{ $\Sigma'=\Sigma'_{A}\cup \Sigma'_{B}$}
\vskip 15pt
\begin{tikzpicture}[scale=0.7, photon/.style={decorate,decoration={snake,post length=0.8mm}} ]
\pgfmathsetmacro{\c}{2}
    \pgfmathsetmacro{\d}{3.464} 
   \draw [green] [dashed] plot[domain=-1:1] ({\c*cosh(\x)},{\d*sinh(\x)});
    \draw  [green][dashed]  plot[domain=-1:1] ({-\c*cosh(\x)},{\d*sinh(\x)});
\draw[thick]		(2,0) -- (2.4,0.4);								
\draw[thick] (2.4, 0.4)--(2.8, 0.8);
\draw[thick] 					(1.2, 0.8) -- (1.8, 0.2);
\draw[thick] (1.8, 0.2)--(2, 0);
\draw[thick]		(2,0) -- (2.4, -0.4);							
\draw[thick] (2.4, -0.4)--(2.8, -0.8);
\draw[thick] 	(-1.2, 0.8)--(-1.7, 0.3) ; 
\draw[thick] 	(-2, 0)--(-1.7, 0.3);						
\draw[thick] 	(-1.2,- 0.8)--(-1.7, -0.3) ; 
\draw[thick] 	(-2, 0)--(-1.7, -0.3);						
\draw[thick] 					(1.2, 0.8) -- (1.8, 0.2);
\draw[thick] (1.8, 0.2)--(2, 0);
\draw[thick] 					(1.2, -0.8) -- (1.8, -0.2);
\draw[thick] (1.8, -0.2)--(2, 0);
\draw[thick] 					(-2.8, 0.8) -- (-2.2, 0.2);
\draw[thick] (-2.2, 0.2)--(-2, 0);
\draw[thick] 					(-2.8, -0.8) -- (-2.2, -0.2);
\draw[thick] (-2.2, -0.2)--(-2, 0);
\draw	[fill]						(-2,0)		circle[radius=0.1];	
\draw	[fill]					(2,0)		circle[radius=0.1];
\draw		[blue]				(-2,0)		circle[radius=1.131];	
\draw		[blue]					(2,0)		circle[radius=1.131];
\node[below] at (-2,-0.1)			{$-z_0$};
\node[below] at (2,-0.1)			{$z_0$};
\node[left] at (-4.56, 0.565)					{$\Sigma_{A'}$};
\node[right] at (4.56, 0.565)				{$\Sigma_{B'}$};

\draw[->,photon] ( 4.565,0.565) --  (2.565, 0.565); 
\draw[->,photon] ( -4.565,0.565) --  (-2.565, 0.565); 
\draw[->,photon] ( 1, 1.8) --  (2, 1.131); 
\draw[->,photon] ( -1, 1.8) --  (-2, 1.131); 
\node [above] at  ( 1, 1.75)  {$C_B$};
\node [above] at  ( -1, 1.75)  {$C_A$};

\end{tikzpicture}
\label{fig:contour-2}
\end{figure}

\begin{problem}
\label{prob:mkdv.A}
Find a matrix-valued function $m^{A'}(z;x,t)$ on $\bbC \setminus\Sigma_A'$ with the following properties:
\begin{enumerate}
\item		$m^{A'}(z;x,t) \rarr I$ as $ z \rarr \infty$.
\item		$m^{A' }(z;x,t)$ is analytic for $z \in  \bbC \setminus \Sigma_A' $
			with continuous boundary values
			$m^{A'}_\pm(z;x,t)$.
\item On $ \Sigma_A'$ we have the following jump conditions
$$m^{A'}_+(z;x,t)=m^{A'}_-(z;x,t)	
			e^{-i\theta\ad\sigma_3}v^{A'}(z)$$
			where $v^{A'}=v^{(2)}\restriction _{\Sigma_A' }$.
\end{enumerate}
\end{problem}			

\begin{problem}
\label{prob:mkdv.B}
Find a matrix-valued function $m^{B'}(z;x,t)$ on $\bbC \setminus\Sigma_B'$ with the following properties:
\begin{enumerate}
\item		$m^{B'}(z;x,t) \rarr I$ as $ z \rarr \infty$.
\item		$m^{B'}(z;x,t)$ is analytic for $z \in  \bbC \setminus \Sigma_B' $
			with continuous boundary values
			$m^{B’}_\pm(z;x,t)$.
\item On $ \Sigma_B'$ we have the following jump conditions
$$m^{B'}_+(z;x,t)=m^{B’}_-(z;x,t)	
			e^{-i\theta\ad\sigma_3}v^{B'}(z)$$
			where $v^{B’}=v^{(2)}\restriction_{\Sigma_B' }$.
\end{enumerate}
\end{problem}	
To  construct solutions to problems \ref{prob:mkdv.A}-\ref{prob:mkdv.B} we need the following matrix-valued function:
\begin{equation}
\label{eq: para}
\mathcal{P}=\begin{cases}
\twomat{\mathcal{P}_-}{0 }{0}{\mathcal{P}_-^{-1}}, \quad &\left\vert z + z_0 \right\vert<\rho \\
\twomat{\mathcal{P}_+}{0 }{0}{\mathcal{P}_+^{-1}}, \quad & \left\vert z - z_0 \right\vert <\rho \\
\twomat{1}{0}{0}{1},  \quad & \left\vert z \pm z_0 \right\vert \geq \rho
\end{cases}
\end{equation}
where 
\begin{align*}
\mathcal{P}_-&= (192 \tau)^{i \kappa / 2} e^{-8 i \tau} e^{\chi\left(-z_{0}\right)} \eta(z; -z_0)^{-1}(-\zeta_-)^{i\kappa} e^{i \zeta_-^{2} / 4}e^{i\theta} \\
\mathcal{P}_+&= (192 \tau)^{-i \kappa / 2} e^{8 i \tau} e^{\chi\left(z_{0}\right)} \eta(z; z_0)^{-1}\zeta_+^{i\kappa} e^{-i \zeta_+^{2} / 4}e^{i\theta} 
\end{align*}
with $\zeta _\mp =\sqrt{48 z_{0} t}\left(z \pm z_{0}\right)$. Then we further set 
\begin{equation}
\label{c}
 m^\RHP :=  \tilde{m_p}\mathcal{P}^{-1}
\end{equation}
where 
\begin{align}
\tilde{m_p}\restriction  \lbrace  z: \left\vert z + z_0 \right\vert < \rho \rbrace &=m^{A'} \twomat{\mathcal{P}_-}{0 }{0}{\mathcal{P}_-^{-1}}  := m^{A} ,\\
\tilde{m_p}\restriction  \lbrace  z: \left\vert z - z_0 \right\vert < \rho \rbrace &=m^{B'} \twomat{\mathcal{P}_+}{0 }{0}{\mathcal{P}_+^{-1}}  := m^{B}.
\end{align}
Set
\begin{align*}
\delta^0_A(z) &= (192\tau)^{i\kappa/2} e^{-8i\tau} e^{\chi(-z_0)}\\
\delta^0_B(z) &= (192\tau)^{-i\kappa/2} e^{8i\tau} e^{\chi(z_0)}\\
\end{align*}
 Let $\Sigma_A$ and $\Sigma_B$ denote the contours
\begin{equation*}
\lbrace z=u e^{\pm i\pi/4} : -\infty<u<\infty \rbrace
\end{equation*}
with the same orientation as those of $\Sigma_{A'}$ and $\Sigma_{B'}$ respectively. 
\begin{figure}[H]
\caption{$\Sigma_A,\Sigma_B$}
\vskip 15pt
\begin{tikzpicture}[scale=0.6]
\draw[dashed] 				(-4,0) -- (4,0);							
\draw [->,thick,>=stealth] 	(0,0) -- (1.5,1.5);						
\draw  [thick]  (3,3) -- (1.5, 1.5);
\draw 	[->,thick,>=stealth]	  (-3,3) -- (-1.5,1.5);					
\draw	 [thick]	(0,0)--(-1.5,1.5);
\draw[->,thick,>=stealth]	(-3,-3) -- (-1.5,-1.5);					
\draw[thick]					(-1.5,-1.5) -- (0,0);
\draw[->,thick,>=stealth]	(0,0) -- (1.5,-1.5);					
\draw[thick]					(1.5,-1.5) -- (3,-3);
\draw[fill]						(0,0)	circle[radius=0.075];		
\node [below] at  			(0,-0.15)		{$0$};
\node[left] at					(2.5,3)		{$\Sigma_A^1$};
\node[right] at					(-2.5,3)		{$\Sigma_A^2$};
\node[right] at					(-2.5,-3)		{$\Sigma_A^3$};
\node[left] at					(2.5,-3)		{$\Sigma_A^4$};
\end{tikzpicture}
\qquad
\begin{tikzpicture}[scale=0.6]
\draw[dashed] 				(-4,0) -- (4,0);							
\draw [->,thick,>=stealth] 	(0,0) -- (1.5,1.5);						
\draw  [thick]  (3,3) -- (1.5, 1.5);
\draw 	[->,thick,>=stealth]	  (-3,3) -- (-1.5,1.5);					
\draw	 [thick]	(0,0)--(-1.5,1.5);
\draw[->,thick,>=stealth]	(-3,-3) -- (-1.5,-1.5);					
\draw[thick]					(-1.5,-1.5) -- (0,0);
\draw[->,thick,>=stealth]	(0,0) -- (1.5,-1.5);					
\draw[thick]					(1.5,-1.5) -- (3,-3);
\draw[fill]						(0,0)	circle[radius=0.075];		
\node [below] at  			(0,-0.15)		{$0$};
\node[left] at					(2.5,3)		{$\Sigma_B^1$};
\node[right] at					(-2.5,3)		{$\Sigma_B^2$};
\node[right] at					(-2.5,-3)		{$\Sigma_B^3$};
\node[left] at					(2.5,-3)		{$\Sigma_B^4$};
\end{tikzpicture}
\label{fig:jumps-A-B}
\end{figure}
$m^A$ solves the following Riemann-Hilbert problem
\begin{equation}
\left\{\begin{array}{ll}m_{+}^{A}(\zeta) & =m_{-}^{A}(\zeta) v^{B}(\zeta), \quad \zeta \in \Sigma_{A} \\\\ m^{A}(\zeta) &  = I-\frac{m_{1}^{B}}{\zeta}+O\left(\zeta^{-2}\right), \quad \zeta \rightarrow \infty\end{array}\right.
\end{equation}
We have from the list of entries stated in \eqref{R1}, \eqref{R4}, \eqref{R7+} and \eqref{R8+} the rescaled jump matrices on $\Sigma_A$:
\begin{equation}
\label{jump-vA}
v^{A}=\begin{cases}
  \Twomat{1}{0}{(\delta^0_A(z) )^{-2} r(z_0) (-\zeta_-)^{ 2i\kappa} e^{- i\zeta_-^2/2} }{1}, \quad \zeta_- \in \Sigma_A^2 \\
  \Twomat{1}{0}{ \dfrac{   (\delta^0_A(z) )^{-2}  r(z_0)}{1-|r(z_0)|^2} (-\zeta_-)^{ 2i\kappa} e^{- i\zeta_-^2/2} }{1}, \quad \zeta_- \in \Sigma_A^4 \\
 \Twomat{1}{-(\delta^0_A(z) )^{2} \overline{ r(z_0)} (-\zeta_-)^{ -2i\kappa} e^{ i\zeta_-^2/2}  } {0}{0},\quad \zeta_- \in \Sigma_A^3 \\
  \Twomat{1}
{\dfrac{ (\delta^0_A(z) )^{2} \overline{r(z_0)}}{1-|r(z_0)|^2} (-\zeta_-)^{ -2i\kappa} e^{ i\zeta_-^2/2}  }{0}{1} ,\quad \zeta_- \in \Sigma_A^1 .
\end{cases}
\end{equation}
Similarly we have from the the rescaled jump matrices on $\Sigma_B$:
\begin{equation}
\label{jump-vB}
v^{B}=\begin{cases}
 \Twomat{1}{0}{(\delta^0_B(z) )^{-2} r(z_0) \zeta_+^{- 2i\kappa} e^{ i \zeta_+^2/2}  }{1}, \quad \zeta_+ \in \Sigma_B^1 \\
\Twomat{1}{0}{
\dfrac{ (\delta^0_B(z) )^{-2}  r(z_0)}{1-|r(z_0)|^2}  \zeta_+^{- 2i\kappa} e^{ i \zeta_+^2/2} }{1}, \quad \zeta_+ \in \Sigma_B^3 \\
\Twomat{1}{-(\delta^0_B(z) )^{2}   \overline{ r(z_0)}  \zeta_+^{2i\kappa} e^{- i \zeta_+^2/2}  } {0}{1}, \quad \zeta_+ \in \Sigma_B^4 \\
 \Twomat{1}
{\dfrac{ (\delta^0_B(z) )^{2}  \overline{r(z_0)}}{1-|r(z_0)|^2}  \zeta_+^{ 2i\kappa} e^{ -i \zeta_+^2/2} }{0}{1}, \quad \zeta_+ \in \Sigma_B^2 .
\end{cases}
\end{equation}
$m^B$ solves the following Riemann-Hilbert problem
\begin{equation}
\left\{\begin{array}{ll}m_{+}^{B}(\zeta) & =m_{-}^{B^{0}}(\zeta) v^{B}(\zeta), \quad \zeta \in \Sigma_{B} \\\\ m^{B}(\zeta) & = I-\frac{m_{1}^{B}}{\zeta}+O\left(\zeta^{-2}\right), \quad \zeta \rightarrow \infty\end{array}\right.
\end{equation}

The explicit form of $m^{B^0}_1$ is given as follows (see \cite[ Section 4]{DZ93}) :
\begin{equation}
\label{explicit-B0}
m^{B}_1=\twomat{0}{-( \delta^0_B)^2 i\beta_{12}}{ ( \delta^0_B)^{-2}  i\beta_{21}}{0}
\end{equation}
where
$$\beta_{12}=\dfrac{\sqrt{2\pi } e^{i\pi/4} e^{-\pi \kappa } }{r(z_0) \Gamma(-i\kappa)}, \qquad \beta_{21}=\dfrac{-\sqrt{2\pi } e^{-i\pi/4} e^{-\pi \kappa } }{ \overline{ r(z_0) } \Gamma(i\kappa)}$$
and $\Gamma(z)$ is the \textit{Gamma} function. Using the explicit form of $v^{B}$ given by \eqref{jump-vB}, symmetry reduction given by \eqref{minus} and their analogue for $v^{A}$, we verify that
\begin{equation}
v^{A}(z)=  \sigma_3 \overline{ v^{B}( - \overline{z} )  } \sigma_3  
\end{equation}
which  in turn implies by uniqueness that
\begin{equation}
m^{A}(z)= \sigma_3 \overline{ m^{B}( - \overline{z} )  } \sigma_3  
\end{equation}
and from this we deduce that
\begin{align}
\label{mA-mB}
m^{A}_1 &=-\sigma_3 \overline{ m^{B}_1  } \sigma_3  \\
\nonumber
&=\twomat{0}{ ( \delta^0_A)^{-2} i \overline{\beta}_{12}}{- ( \delta^0_A)^{-2} i \overline{\beta}_{21} }{0}.
\end{align}
Collecting all the computations above, we write down the asymptotic expansions of solutions to Problem \ref{prob:mkdv.A} and Problem \ref{prob:mkdv.B} respectively.
\begin{proposition}
\label{solution-A-B}
Recall $\zeta_-=\sqrt{48z_0 t}(z + z_0) $, the solution to RHP Problem  \ref{prob:mkdv.A}  $m^{A'}$ admits the following expansion:
\begin{equation}
\label{expansion-A}
m^{A'}(z(\zeta) ;x,t)=I +\dfrac{1}{\zeta_-}\twomat{0}{i ( \delta^0_A)^2 \overline{\beta}_{12}}{-i ( \delta^0_A)^{-2}\overline{\beta}_{21} }{0} +\mathcal{O}(t^{-1}).
\end{equation}
Similarly, for $\zeta_+=\sqrt{48z_0 t}(z - z_0) $, the solution to RHP Problem  \ref{prob:mkdv.B}  $m^{B'}$ admits the following expansion:
\begin{equation}
\label{expansion-B}
m^{B'}(z(\zeta) ;x,t)=I +\dfrac{1}{\zeta_+}\twomat{0}{-i ( \delta^0_B)^2 {\beta}_{12}}{i ( \delta^0_B)^{-2}{\beta}_{21} }{0} +\mathcal{O}(t^{-1}).
\end{equation}
\end{proposition}
Now we construct $m^\RHP$ needed in the factorization of $m^{(2)}$ in \eqref{factor-LC}. 
In Figure \ref{fig:contour-circle}, we let $\rho$ be the radius of the circle $C_A$ ($C_B$) centered at $z_0$ ($-z_0$). We seek a solution of the form
\begin{equation}
\label{parametrix}
m^\RHP(z)=\begin{cases}
E(z) \quad  &\left\vert z\pm z_0 \right\vert>\rho \\
E(z) m^{A'}(z) \quad &\left\vert z + z_0 \right\vert\leq\rho \\
E(z) m^{B'}(z) \quad & \left\vert z - z_0 \right\vert\leq \rho 
\end{cases}
\end{equation}
\begin{figure}[H]
\caption{ $\Sigma_E$}
\vskip 15pt
\begin{tikzpicture}[scale=0.75]

\draw [->,thick,>=stealth] (-2.5304, 0.5304)--(-4, 2);
\draw [thick] (-4, 2)--(-5, 3);

\draw [->,thick,>=stealth]  (-5,- 3)--(-4, -2);
\draw [thick] (-2.5304, -0.5304)--(-4, -2);

\draw [->,thick,>=stealth] (5, 3)--(4, 2);
\draw [thick] (4, 2)--(2.5304, 0.5304);

\draw [->,thick,>=stealth]  (2.5304, -0.5304)--(4,-2) ;
\draw [thick] (5, -3)--(4, -2) ;

\draw [->,thick,>=stealth]  (0, 2)--(-1. 4696, 0.5304);

\draw [->,thick,>=stealth]  (1. 4696, 0.5304)--(1,1);
\draw[thick] (1,1)--(0,2);

\draw [->,thick,>=stealth]  (-1. 4696, -0.5304)--(-1,-1);
\draw[thick] (-1,-1)--(0, -2);

\draw [->,thick,>=stealth]  (0, -2)--(1, -1);
\draw[thick] (1, -1) -- (1. 4696, -0.5304) ;

\draw	[fill]							(-2,0)		circle[radius=0.1];	
\draw	[fill]							(2,0)		circle[radius=0.1];

\draw							(-2,0)		circle[radius=0.75];	
\draw							(2,0)		circle[radius=0.75];	

\draw[->,thick,>=stealth] 	(0, -2) -- (0,0);
\draw[thick]	(0,0) -- (0,2);		
\node[below] at (-2,-0.1)			{$-z_0$};
\node[below] at (2,-0.1)			{$z_0$};
\node[right] at (5,3)					{$\Sigma'^{(2)}_1$};
\node[left] at (-5,3)					{$\Sigma'^{(2)}_2$};
\node[left] at (-5,-3)					{$\Sigma'^{(2)}_3$};
\node[right] at (5,-3)				{$\Sigma'^{(2)}_4$};
\node[left] at (-1,1.2)					{$\Sigma'^{(2)}_5$};
\node[left] at (-1,-1.2)					{$\Sigma'^{(2)}_7$};
\node[right] at (1,1.2)					{$\Sigma'^{(2)}_6$};
\node[right] at (1,-1.2)					{$\Sigma'^{(2)}_8$};
\node[right] at (0,0)              {$\Sigma'^{(2)}_9$};
\end{tikzpicture}
\label{fig:contour-circle}
\end{figure}
Since $ m^{A'}$ and $ m^{B'}$ solve Problem \ref{prob:mkdv.A} and Problem \ref{prob:mkdv.B} respectively, we can construct the solution $m^\RHP(z)$ if we find $E(z)$. Indeed,  $E$ solves the following Riemann-Hilbert problem:
\begin{problem}
\label{prob: E-2}
Find a matrix-valued function $E(z)$ on $\bbC \setminus \Sigma_E$ with the following properties:
\begin{enumerate}
\item		$E(z) \rarr I$ as $ z \rarr \infty$,
\item		$E(z) $ is analytic for $z \in  \bbC \setminus \left( C_A\cup C_B \right)$
			with continuous boundary values
			$E_{\pm}(z)$.
\item On $ C_A\cup C_B $we have the following jump conditions
$$E_{+}(z)=E_{-}(z) v^{ (E) }(z)$$
			where 
			\begin{equation}
			\label{v-E}
			 v^{ (E) }(z)=\begin{cases}
			m^{A'}(z(\zeta) ), \quad &z\in C_A\\
			m^{B'}(z(\zeta) ), \quad &z\in C_B\\
			v^{(2)} , \quad &z\in \Sigma_E\setminus\left(  C_A\cup C_B\right).
			\end{cases}
			\end{equation}
\end{enumerate}
\end{problem}
\begin{proposition}
 $E(z)$ admits a classical solution, i.e jump condition \eqref{v-E} holds pointwise on the contour $\Sigma_E$. 
\begin{proof}
Here we invoke to the well-established existence and uniqueness theory from \cite{Zhou89} (see also chapter 2 \cite{TO16}). First it is easy to check that 
$$v^{ (E) }(z)= v^{ (E) \dagger}(z)$$
where the $\dagger$ denotes the Hermitian conjugate of the given matrix. We then take care of the zero sum condition at the self-intersecting points of $\Sigma_E$. Since the remaining cases follows from symmetry, we will only look $\Sigma'^{(2)}_5 \cap \Sigma'^{(2)}_6\cap \Sigma'^{(2)}_{9} $ and $\Sigma'^{(2)}_6\cap C_A $. The zero sum condition holds at the first point by comparing \eqref{R8+} and \eqref{v_9}. For $\Sigma'^{(2)}_6\cap C_A $, (after adding contour with identity jumps and reorientation cf. P. 1058 \cite{DZ03} ) we explicitly compute
\begin{align*}
I & =  m^{A'}(z(\zeta) )  \left[  v^{(2)} \right]^{-1}  \left[ m^{A'}(z(\zeta) ) \right] ^{-1}\\
   &= m^{A'}_+(z(\zeta) )\left( v^{(2)} \right)^{-1} \left( m^{A'}_-(z(\zeta))\right)^{-1}.
\end{align*}
Since $v^{(2)}$ is smooth away from the intersections and zero sum conditions have been verified, this completes the proof.
\end{proof}
\end{proposition}

Setting 
$$\eta(z)=E_{-}(z)-I$$
then by standard theory, we have the following singular integral equation
$$  \eta=I+ C_{v^{(E)}}\eta $$
where the singular integral operator is defined by:
$$  C_{v^{(E)}}\eta =C^-\left( \eta \left( v^{(E)}-I  \right)  \right). $$
We first deduce from \eqref{expansion-A}-\eqref{expansion-B} that 
\begin{equation}
\norm{v^{(E)}-I}{L^\infty} \lesssim t^{-1/2}
\end{equation}
hence the operator norm of $C_{v^{(E)}} $ 
\begin{equation}
\label{norm-vE}
\norm{C_{v^{(E)}}f }{ L^2} \leq \norm{f}{L^2} \norm{v^{(E)}-I}{L^\infty} \lesssim t^{-1/2}.
\end{equation}
Then the resolvent operator $(1-C_{v^{(E)}})^{-1}$ can be obtained through \textit{Neumann} series and we obtain the unique solution to Problem \ref{prob: E-2}:
\begin{equation}
E(z)=I+ \dfrac{1}{2\pi i}\int_{C_A\cup C_B} \dfrac{ (1+\eta(s))(v^{(E)}(s)-I )   }{s-z}ds
\end{equation}
which admits the following asymptotic expansion in $z$:
\begin{equation}
\label{E2-expan}
E_2(z)=I+ \dfrac{E_{1}}{z} +\mathcal{O}\left( \dfrac{1}{z^2} \right).
\end{equation}
Using the bound on the operator norm \eqref{norm-vE}, we obtain
\begin{align}
\label{bound-E-1}
E_{1}(z) &=-\dfrac{1}{2\pi i}\int_{C_A\cup C_B} { (1+\eta(s))(v^{(E)}(s)-I )   }ds  \\
              &=-\dfrac{1}{2\pi i}\int_{C_A\cup C_B} { (v^{(E)}(s)-I )   }ds +\mathcal{O}(t^{-1}).
\end{align}
Given the form of $v^{(E)}$ in \eqref{v-E} and the asymptotic expansions \eqref{expansion-A}-\eqref{expansion-B}, an application of Cauchy's integral formula leads to
\begin{align}
\label{E-1-2 cauchy}
E_{1}&=\dfrac{1}{\sqrt{48 z_0 t}} \twomat{0}{-i ( \delta^0_B)^2 {\beta}_{12}}{i ( \delta^0_B)^{-2}{\beta}_{21} }{0}  + \dfrac{1}{\sqrt{48 z_0 t}}  \twomat{0}{ i ( \delta^0_A)^2 \overline{\beta}_{12} }{-i ( \delta^0_A)^{-2}\overline{\beta}_{21} }{0}\\
 \nonumber
     & \quad + \mathcal{O}(t^{-1}).
\end{align}

After possible reorientation of the contours, using the reconstruction formula given by \eqref{mkdv.q}, we expect that 
\begin{align}
\label{leading-I}
u(x,t)&=\left( \dfrac{\kappa}{3z_0 t}\right)^{1/2}\cos \left(16tz_0^3-\kappa\log(192tz_0^3)+\phi(z_0) \right)\\
\nonumber
&\quad + O\left(\dfrac{c(z_0)}{\sqrt{z_0 t} \tau^{1/2}}\right) + \mathcal{E}_1
\end{align}
where
$$\phi(z_0)=\arg \Gamma(i\kappa)-\dfrac{\pi}{4}-\arg r(z_0)+\dfrac{1}{\pi}\int_{-z_0}^{z_0}\log\left( \dfrac{1-|r(\zeta)|^2}{1-|r(z_0)|^2} \right)\dfrac{d\zeta}{\zeta-z_0} $$
and $\mathcal{E}_1$ is the error induced by a pure-$\dbar$ problem to be studied in the following section.
\section{The $\dbar$-Problem}
\label{sec:dbar}
 
From \eqref{factor-LC} we have matrix-valued function
\begin{equation}
\label{N3}
m^{(3)}(z;x,t) = m^{(2)}(z;x,t) m^\RHP(z; x,t)^{-1}. 
\end{equation}
The goal of this section is to show that $m^{(3)}$ only results in an error term $E_1$ with higher order decay rate than the leading order term of the asymptotic formula \eqref{leading-I}.

Since $m^\RHP(z; x, t)$ is analytic in $\bbC \setminus \Sigma^{'(2)}$, we may compute
\begin{align*}
\dbar m^{(3)}(z;x,t) 	&=	\dbar m^{(2)}(z;x,t) m^\RHP(z; x,t)^{-1}\\	
								&=	m^{(2)}(z;x,t) \, \dbar \calR^{(2)}(z) m^\RHP(z; x,t)^{-1}	&\text{(by \eqref{N2.dbar})}\\
								&=	m^{(3)}(z;x,t) m^{\RHP}(z; x, t) \, \dbar \calR^{(2)}(z) m^\RHP(z; x,t)^{-1}	& \text{(by \eqref{N3})}\\
								&=	m^{(3)}(z;x,t)  W(z;x,t)
\end{align*}
where
 \begin{equation}
 \label{W-bound}
 W(z;x,t) = m^{\RHP}(z;x,t ) \, \dbar \calR^{(2)}(z)  m^\RHP(z;x,t )^{-1}. 
 \end{equation}
We thus arrive at the following pure $\dbar$-problem:

\begin{problem}
\label{prob:DNLS.dbar}
Give $r \in H^{1}(\bbR)$, find a continuous matrix-valued function
$m^{(3)}(z;x,t)$ on $\bbC$ with the following properties:
\begin{enumerate}
\item		$m^{(3)}(z;x,t) \rarr I$ as $|z| \rarr \infty$, 
\medskip
\item		$\dbar m^{(3)}(z;x,t) = m^{(3)}(z;x,t) W(z;x,t)$.
\end{enumerate}
\end{problem}

It is well understood (see for example \cite[Chapter 7]{AF}) that the solution to this $\dbar$ problem is equivalent to the solution of a Fredholm-type integral equation involving the solid Cauchy transform
$$ (Pf)(z) = \frac{1}{\pi} \int_\bbC \frac{1}{\zeta-z} f(\zeta) \, d\zeta $$
where $d$ denotes Lebesgue measure on $\bbC$. 
\begin{lemma}
A bounded and continuous matrix-valued function $m^{(3)}(z;x,t)$ solves Problem \eqref{prob:DNLS.dbar} if and only if
\begin{equation}
\label{DNLS.dbar.int}
m^{(3)}(z;x,t) =I+ \frac{1}{\pi} \int_\bbC \frac{1}{\zeta-z} m^{(3)}(\zeta;x,t) W(\zeta;x,t) \, d\zeta.
\end{equation}
\end{lemma}

 Using the integral equation formulation \eqref{DNLS.dbar.int}, we will prove:

\begin{proposition}
\label{prop:N3.est}
Suppose that $r \in H^{1}(\bbR)$.
Then, for $t\gg 1$, there exists a unique solution $m^{(3)}(z;x,t)$ for Problem \ref{prob:DNLS.dbar} with the property that 
\begin{equation}
\label{N3.exp}
m^{(3)}(z;x,t) = I + \frac{1}{z}  m^{(3)}_1(x,t) + o\left( \frac{1}{z} \right) 
\end{equation}
for $z=i\sigma$ with $\sigma \rarr +\infty$. Here
\begin{equation}
\label{N31.est}
\left| m^{(3)}_1(x,t) \right| \lesssim (z_0 t)^{-3/4}
\end{equation}
 where the implicit constant in \eqref{N31.est} is uniform for 
$r$ in a bounded subset of $H^{1}(\bbR)$ .
\end{proposition}

\begin{proof} Given Lemmas \ref{lemma:dbar.R.bd}--\ref{lemma:N31.est},
 as in \cite{LPS}, we   first show that, for large $t$, the integral operator $K_W$ 
defined by
\begin{equation*}
\left( K_W f \right)(z) = \frac{1}{\pi} \int_\bbC \frac{1}{\zeta-z} f(\zeta) W(\zeta) \, d \zeta
\end{equation*}
is bounded by
\begin{equation}
\label{dbar.int.est1}
\norm{K_W}{L^\infty \rarr L^\infty} \lesssim (z_0 t)^{-1/4}
\end{equation}
where the implied constants depend only on $\norm{r}{H^{1}}$.  This is the goal of Lemma \ref{lemma:KW}.
It implies that 
\begin{equation}
\label{N3.sol}
m^{(3)} = (I-K_W)^{-1}I
\end{equation}
exists as an $L^\infty$ solution of \eqref{DNLS.dbar.int}.

We  then show in Lemma \ref{lemma:N3.exp} that the solution $m^{(3)}(z;x,t)$ has a large-$z$ asymptotic expansion of the form \eqref{N3.exp}
where $z \rarr \infty$ along the \emph{positive imaginary axis}. Note that, for such $z$, we can bound $|z-\zeta|$ below by a constant times $|z|+|\zeta|$.
Finally,  in Lemma \ref{lemma:N31.est} we  prove  estimate \eqref{N31.est}
where the constants are uniform in $r$ belonging to a bounded subset of $H^{1}(\bbR)$.
Estimates \eqref{N3.exp}, \eqref{N31.est}, and \eqref{dbar.int.est1} result from the  bounds obtained in the next four lemmas.
\end{proof}

\begin{lemma}
\label{lemma:dbar.R.bd}
Set $\xi=\pm z_0$ and
 $z=(u +\xi)+iv$. We have 
\begin{equation}
\label{dbar.R2.bd}
\left| \dbar \calR^{(2)} e^{\pm 2i\theta}  \right| 	\lesssim		
		\left( |p_i'(\Real (z))| + |z-\xi|^{-1/2}\right) e^{-z_0t|u||v|}, 			
			\end{equation}
\end{lemma}

\begin{proof}
We only show the inequalities above in $\Omega_1$ and $\Omega_7^+$. Recall that near $z_0$
$$i\theta(z; x, t)=4it \left( (z - z_0)^3 + 3z_0 (z - z_0)^2 - 2z_0^3  \right).$$
In $\Omega_1$, we use the facts that $u\geq 0$, $v\geq 0$ and $|u|\geq |v|$ to deduce
\begin{align*}
\text{Re}(2i\theta) &=8it(3iu^2v-iv^3+6iuvz_0)\\
                              &=8t(-3u^2v+v^3-6uvz_0)\\
                              &\leq 8t(-3u^2v +u^2v-6uvz_0)\\
                              &\leq 8t (-2u^2v-6uvz_0)\\
                              &\leq -8|u| |v| z_0 t.
\end{align*}
Similarly, in $\Omega_7^+$, we have $u\leq 0$, $v\geq 0$ and $|u|\geq |v|$, hence
\begin{align*}
\text{Re}(-2i\theta) &=-8it(3iu^2v-iv^3+6iuvz_0)\\
                              &=8t(3u^2v+6uvz_0)\\
                              &\leq 8t(-3u z_0v +6uvz_0)\\
                            &\leq -8|u| |v| z_0 t.
\end{align*}
Estimate \eqref{dbar.R2.bd} then follows from Lemma \ref{lemma:dbar.Ri}. The quantities $p_i'(\Real z)$ are all bounded uniformly for  $r$ in a bounded subset of $H^{1}(\bbR)$.  

\end{proof}

\begin{lemma}
\label{lemma:RHP.bd}For the localized Riemann-Hilbert problem from Problem \ref{MKDV.RHP.local}, we have
\begin{align}
\label{RHP.bd1}
\norm{m^\RHP(\dotarg; x, t)}{\infty}	&	\lesssim		1,\\[5pt]
\label{RHP.bd2}
\norm{m^\RHP(\dotarg; x,t )^{-1}}{\infty}	&	\lesssim	1.
\end{align}
All implied constants are uniform  for  $r$ in a bounded subset of $H^{1}(\bbR)$.
\end{lemma}

The proof of this lemma is a consequence of the previous section.

\begin{lemma}
\label{lemma:KW}
Suppose that $r\in H^{1}(\bbR)$.
Then, the estimate \eqref{dbar.int.est1}
holds, where the implied constants depend on $\norm{r}{H^{1}}$.
\end{lemma}

\begin{proof}
To prove \eqref{dbar.int.est1}, first note that
\begin{align}
 \norm{K_W f}{\infty} &\leq \norm{f}{\infty} \int_\bbC \frac{1}{|z-\zeta|}|W(\zeta)| \, dm(\zeta) 
                                \end{align}
so that we need only estimate the right-hand integral. We will prove the estimate in the region $ z\in\Omega_1$ since estimates for the remaining regions are identical. From \eqref{W-bound}
$$ |W(\zeta)| \leq \norm{m^{\RHP}}{\infty} \norm{(m^{\RHP})^{-1}}{\infty} \left| \dbar R_1\right| |e^{2i\theta}|.$$
Setting $z=\alpha+i\beta$ and $\zeta=(u+z_0)+iv$, the region $\Omega_1$ corresponds to $u\geq v \geq 0 $. We then have from \eqref{dbar.R2.bd} \eqref{RHP.bd1}, and \eqref{RHP.bd2} that
$$
 \int_{\Omega_1}  \frac{1}{|z-\zeta|} |W(\zeta)| \, d\zeta  \lesssim  I_1 + I_2 
$$
where
\begin{align*}
I_1 	&=	\int_0^\infty \int_v^\infty \frac{1}{|z-\zeta|} |p_1'(u)| e^{-tz_0uv} \, du \, dv \\[5pt]
I_2	&=	\int_0^\infty \int_v^\infty \frac{1}{|z-\zeta|} \left| u+iv \right|^{-1/2} e^{-t z_0 uv} \, du \, dv.
\end{align*}
It now follows from \cite[proof of Proposition D.1]{BJM16} that
$$
|I_1|, \, |I_2| \lesssim (z_0t)^{-1/4}.
$$
It then follows that
$$ \int_{\Omega_1} \frac{1}{|z-z_0|} |W(\zeta)| \, d\zeta \lesssim (z_0t)^{-1/4} $$
which, together with similar estimates for the integrals over the remaining $\Omega_i$s,  proves \eqref{dbar.int.est1}.
\end{proof}

\begin{lemma}
\label{lemma:N3.exp}
For $z=i\sigma$ with $\sigma \rarr +\infty$, the expansion \eqref{N3.exp} holds with 
\begin{equation}
\label{N3.1}
m^{(3)}_1(x,t) = \frac{1}{\pi} \int_{\bbC} m^{(3)}(\zeta;x,t) W(\zeta;x,t) \, d\zeta . 
\end{equation}
\end{lemma}

\begin{proof}
We write   \eqref{DNLS.dbar.int} as 
$$
m^{(3)}(z;x,t) = (1,0) + \frac{1}{z} m^{(3)}_1(x,t) + \frac{1}{\pi z} \int_{\bbC} \frac{\zeta}{z-\zeta} m^{(3)}(\zeta;x,t) W(\zeta;x,t) \, dm(\zeta)
$$
where $m^{(3)}_1$is given by \eqref{N3.1}. If $z=i\sigma$, it is easy to see that $|\zeta|/|z-\zeta|$ is bounded above by a fixed constant independent of $z$, while $|m^{(3)}(\zeta;x,t)| \lesssim 1$ by the remarks following \eqref{N3.sol}. If we can show that $\int_\bbC |W(\zeta;x,t)| \, d\zeta$ is finite, it will follow from the Dominated Convergence Theorem that 
$$
\lim_{\sigma \rarr \infty} \int_\bbC \frac{\zeta}{i\sigma-\zeta} m^{(3)}(\zeta;x,t) W(\zeta;x,t) \, d\zeta = 0 
$$ 
which implies the required asymptotic estimate. We will estimate $\dint_{\hspace{-1.25mm} \Omega_1} |W(\zeta)| \, dm(\zeta)$ since the other estimates are identical. One can write
$$\Omega_1= \left\{ (u+z_0,v): v \geq 0, \, v \leq u < \infty\right\}.$$ 
Using \eqref{dbar.R2.bd}, \eqref{RHP.bd1}, and \eqref{RHP.bd2}, we may then estimate
$$
\int_{\Omega_1} |W(\zeta;x,t)| \, d \zeta	\lesssim  I_1+I_2
$$
where
\begin{align*}
I_1	&=	\int_0^\infty \, \int_v^\infty \left| p_1'(u+z_0) \right| e^{-tz_0uv} \, du \, dv\\
I_2	&=	\int_0^\infty \int_v^\infty \left| u^2 + v^2 \right|^{-1/2} e^{-tz_0 uv} \, du \, dv. \\
\end{align*}
It now follows from \cite[Proposition D.2]{BJM16} that
$$I_1, \, I_2 \lesssim (z_0t)^{-3/4}.$$
These estimates together show that
\begin{equation}
\label{W.L1.est}
\int_{\Omega_1} |W(\zeta;x,t)| \, dm(\zeta)	\lesssim( z_0t)^{-3/4}
\end{equation}
and that the implied constant depends only on $\norm{r}{H^{1}}$.  In particular, the integral \eqref{W.L1.est} is bounded uniformly as $t \rarr \infty$.
\end{proof}

\begin{lemma}
\label{lemma:N31.est}
The estimate   \eqref{N31.est} 
 holds with constants uniform in $r$ in a bounded subset of $H^{1}(\bbR)$ .
\end{lemma}

\begin{proof}
From the representation formula \eqref{N3.1}, Lemma \ref{lemma:KW}, and the remarks following, we have
$$ \left|m^{(3)}_1(x,t) \right| \lesssim \int_\bbC |W(\zeta;x,t)| \, d\zeta. $$
In the proof of Lemma \ref{lemma:N3.exp}, we bounded this integral by $( z_0t)^{-3/4}$ modulo constants with the required uniformity.
\end{proof}
\section{Long-Time Asymptotics}
\label{sec:large-time}

We now put together our previous results and formulate the long-time asymptotics of $u(x,t)$  in Region I.  Undoing all transformations we carried out previously,  we get back $m$:
\begin{equation}
\label{N3.to.N}
m(z;x,t) = m^{(3)}(z;x,t) m^\RHP(z; z_0) \calR^{(2)}(z)^{-1} \delta(z)^{\sigma_3}.
\end{equation}
By stand inverse scattering theory, the coefficient of $z^{-1}$ in the large-$z$ expansion for $m(z;x,t)$ will be the solution to the MKdV equation:

\begin{lemma}
\label{lemma:N.to.NRHP.asy}
For $z=i\sigma$ and $\sigma \rarr +\infty$, the asymptotic relations
\begin{align}
\label{N.asy}
m(z;x,t) 				&=	I + \frac{1}{z} m_1(x,t) + o\left(\frac{1}{z}\right)\\
\label{N.RHP.asy}
m^\RHP(z;x,t)		&=	I + \frac{1}{z} m^\RHP_1(x,t) + o\left(\frac{1}{z}\right)
\end{align}
hold. Moreover,
\begin{equation}
\label{N.to.NRHP.asy} 
\left(m_1(x,t)\right)_{12} = \left(m^\RHP_1(x,t)\right)_{12} + \bigO{(z_0 t)^{-3/4}}.
\end{equation}
\end{lemma}

\begin{proof}
By Lemma \ref{lemma:delta}(iii), 
the expansion
\begin{equation}
\label{delta.sigma.asy} 
\delta(z)^{\sigma_3} = \twomat{1}{0}{0}{1} + \frac{1}{z} \twomat{\delta_1}{0}{0}{\delta_1^{-1}} + \bigO{z^{-2}} 
\end{equation}
holds, with the remainder in \eqref{delta.sigma.asy}
 uniform in $r$ in a bounded subset of $H^{1}$. 
\eqref{N.asy} follows from \eqref{N3.to.N},  \eqref{N.RHP.asy}, the fact that $\calR^{(2)} \equiv I$ in $\Omega_2$,
and \eqref{delta.sigma.asy}.
Notice the fact that the 
diagonal matrix in \eqref{delta.sigma.asy} does not affect the $12$-component of $m$. Hence, for 
$z=i\sigma$,
$$ 
\left(m(z;x,t)\right)_{12} = 
		\frac{1}{z}\left(m^{(3)}_1(x,t)\right)_{12} + 
		\frac{1}{z}\left(m^\RHP_1(x,t)\right)_{12} + o\left(\frac{1}{z}\right)
$$
and result now follows from \eqref{N31.est}. 
\end{proof}

We arrive at the asymptotic formula in Region I:

\begin{proposition}
\label{lemma:N.RHP.asy}
The function 
\begin{equation}
\label{q.recon.bis}
u(x,t) = -2 \lim_{z \rarr \infty} z\, m_{12}(z;x,t)
\end{equation}
takes the form 
$$ u(x,t) = u_{as}(x,t) +  \mathcal{O}\left( t^{-1} +(z_0 t)^{-3/4}\right) $$
where 
$$u_{as}(x,t)=\left( \dfrac{\kappa}{3tz_0}\right)^{1/2}\cos \left(16tz_0^3-\kappa\log(192tz_0^3)+\phi(z_0) \right)$$
with
$$\phi(z_0)=\arg \Gamma(i\kappa)-\dfrac{\pi}{4}-\arg r(z_0)+\dfrac{1}{\pi}\int_{-z_0}^{z_0}\log\left( \dfrac{1-|r(\zeta)|^2}{1-|r(z_0)|^2} \right)\dfrac{d\zeta}{\zeta-z_0}$$
is obtained from \eqref{leading-I}.
\end{proposition}
See Section 4 in Deift-Zhou \cite{DZ93} for full details on the derivation for the explicit formula of $u_{as}$.
\section{The regions II-V}
We now turn to the study of the Regions II-V. We first study Region III, then Region II and finally Region IV and Region V. Our starting point is RHP Problem \ref{prob:DNLS.RH0} and the strategy of the proof is as follows:
\begin{itemize}
\item[1.] We scale the RHP Problem \ref{prob:DNLS.RH0} by a factor determined by the region.
\item[2.] We use $\dbar$-steepest descent to study the scaled RHP and obtain both leading term and error term.
\item[3.] We multiply by the scaling factor to get the asymptotic formula for the original RHP Problem \ref{prob:DNLS.RH0}.
\end{itemize}
\subsection{Region III}
In this region, $\tau\leq M$.
\subsubsection{$x<0$}
We first notice that
$$z_0=(\tau/t)^{1/3}\leq (M)^{1/3} t^{-1/3} \to 0 \quad \text{as}\,\, t \to \infty$$
so we do not need the lower/upper factorization given by \eqref{DNLS.V1} for $|z|<z_0$ and are left with the following upper/lower factorization:
\begin{equation}
\label{v-ul}
e^{-i\theta\ad\sigma_3}v(z)	=\Twomat{1}{-\overline{r(z)}  e^{-2i\theta}}{0}{1} \Twomat{1}{0}{r(z) e^{2i\theta}}{1},
						\quad z \in\bbR 
\end{equation}
Now we carry out the following scaling:
\begin{equation}
\label{scale}
z\to \zeta t^{-1/3}
\end{equation}
and \eqref{v-ul} becomes
\begin{equation}
\label{v-scale1}
\Twomat{1}{-\overline{r(\zeta t^{-1/3 } ) }  e^{-2i\theta (\zeta t^{-1/3 } ) } }{0}{1} \Twomat{1}{0}{r(\zeta t^{-1/3 } ) e^{2i\theta (\zeta t^{-1/3 } )  } }{1},
						\quad z \in\bbR 
\end{equation}
where
						$$ \theta (\zeta t^{-1/3 } )=4 \zeta^3+x \zeta t^{-1/3}=4( \zeta^3-3\tau^{2/3} \zeta).$$
Note that the stationary points now become $\pm z_0 t^{1/3}$.

We now study the scaled Riemann-Hilbert problem with jump matrix \eqref{v-scale1}. We will again perform contour deformation and write the solution as a product of solution to a $\dbar$-problem and a "localized" Riemann-Hilbert problem.

\begin{figure}[H]
\caption{$\Sigma-\text{Region-III}$}
\vskip 15pt
\begin{tikzpicture}[scale=0.7]
\draw[thick]		(5, 3) -- (4,2);						
\draw[->,thick,>=stealth] 		(2,0) -- (4,2);		
\draw[thick] 			(-2,0) -- (-4,2); 				
\draw[->,thick,>=stealth]  	(-5,3) -- (-4,2);	
\draw[->,thick,>=stealth]		(-5,-3) -- (-4,-2);							
\draw[thick]						(-4,-2) -- (-2,0);
\draw[thick,->,>=stealth]		(2,0) -- (4,-2);								
\draw[thick]						(4,-2) -- (5,-3);
\draw[thick]		(0,0)--(2,0);
\draw[thick,->,>=stealth] (-2,0) -- (0, 0);
\draw	[fill]							(-2,0)		circle[radius=0.1];	
\draw	[fill]							(2,0)		circle[radius=0.1];
\draw [dashed] (2,0)--(6,0);
\draw [dashed] (-6,0)--(-2,0);
\node[below] at (-2,-0.35)			{$-z_0 t^{1/3}$};
\node[below] at (2,-0.35)			{$z_0 t^{1/3}$};
\node[right] at (5,3)					{$\Sigma^{( \text{III} )}_1$};
\node[left] at (-5,3)					{$\Sigma^{( \text{III})}_2$};
\node[left] at (-5,-3)					{$\Sigma^{( \text{III})}_3$};
\node[right] at (5,-3)				{$\Sigma^{(\text{III})}_4$};
\node[above] at (4.5,0)           {$\Omega_1$};
\node[above] at (-4.5,0)           {$\Omega_2$};
\node[below] at (-4.5,0)           {$\Omega_3$};
\node[below] at (4.5,0)           {$\Omega_4$};
\end{tikzpicture}
\label{fig:contour-scale-1}
\end{figure}
For brevity, we only discuss the $\dbar$-problem in $\Omega_1$. In $\Omega_1$, we write 
$$\zeta= u+z_0 t^{1/3} +iv $$ 
then
\begin{align*}
\text{Re} (2i\theta (\zeta t^{-1/3})) &=8\left( -3(u+z_0 t^{1/3})^2 v +v^3 +3\tau^{2/3}v\right) \\
      & \leq 8 \left( -3u^2v -6 uv z_0 t^{1/3}  +v^3 \right)\\
      &\leq -16u^2 v
\end{align*}
	\begin{align*}
	R_1	&=	\begin{cases}
						\twomat{0}{0}{r(\zeta t^{-1/3} ) e^{2i\theta (\zeta t^{-1/3 } )  }  }{0}		
								&	z \in (z_0t^{1/3},\infty)\\[10pt]
								\\
						\twomat{0}{0}{r( z_0 ) e^{2i\theta (\zeta t^{-1/3 } )  }  }{0}	
								&	z	\in \Sigma_1
					\end{cases}
	\end{align*}
and the interpolation is given by 
$$r(z_0)+ \left( r\left( \text{Re}\zeta t^{-1/3} \right) -r(z_0)  \right) \cos 2\phi$$
So we arrive at the $\dbar$-derivative in $\Omega_1$ in the $\zeta$ variable:
\begin{equation}
\dbar R_1	= \left( {t^{-1/3}} r'\left( u t^{-1/3} \right) \cos 2\phi- 2\dfrac{ r(ut^{-1/3} )  -r(z_0)  }{  \left\vert \zeta-z_0 t^{1/3} \right\vert   } e^{i\phi} \sin 2\phi  \right) e^{2i\theta},
\end{equation}
\begin{equation}
\label{R1.bd1}
\left| \dbar R_1 e^{\pm 2i\theta}  \right| 	\lesssim\left( |  t^{-1/3} r'\left( u t^{-1/3} \right) | +\dfrac{\norm{r'}{L^2} }{ t^{1/3} |\zeta t^{-1/3}-z_0  |^{1/2} }\right) e^{-16u^2 v}.
\end{equation}
We will derive an exactly solvable model problem before dealing with the $\overline{\partial}-$ error estimates. We apply the fundamental theorem of calculus to get
$$r(\zeta t^{-1/3}) e^{2i\theta}-r(0)e^{2i\theta} \leq \left\vert  \dfrac{\zeta}{t^{1/6}} e^{8i (\zeta^3-3 \tau^{2/3}\zeta)}\right\vert.$$
Given the fact that $z_0 t^{1/3}=\tau^{1/3}\leq (M)^{1/3}$, we have that
$$\norm{  \dfrac{\zeta}{t^{1/6}} e^{8i (\zeta^3-3 \tau^{2/3}\zeta)}}{L^1\cap L^2\cap L^\infty} \lesssim t^{-1/6}.$$
So we can reduce the problem to a problem on the contour given by figure \ref{fig:contour-scale-1} with the following jump matrices:
\begin{align}
\label{matrices- Painleve }
e^{-i\theta\ad\sigma_3 } v^{(2)}(\zeta) &=e^{-4i\left(\zeta^3+\left( x/(4 t^{1/3}) \right)  \zeta \right) \ad\sigma_3 }  \twomat{1} {0} { {r(0)} }{1} , \quad \zeta\in \Sigma^{(\text{III})}_1\cup \Sigma^{(\text{III})}_2 \\
\nonumber
                                                                                                                   &=e^{-4i\left(\zeta^3+\left( x/(4 t^{1/3}) \right)  \zeta \right) \ad\sigma_3 } \twomat{1}{-\overline{r(0)}}{0}{1}, \quad \zeta\in \Sigma^{(\text{III})}_3 \cup \Sigma^{(\text{III})}_4\\
     \nonumber
                                                                                                                   &=                                                                                                              e^{-4i\left(\zeta^3+\left( x/(4 t^{1/3}) \right)  \zeta \right) \ad\sigma_3 } v(0),  \quad \zeta\in 
                                                                                                                   [-z_0t^{1/3}, z_0t^{1/3}].
\end{align}
Following the same argument on P. 357 of \cite{DZ93}, the RH problem is further reduced to one defined on the following contour which will be related to solve a
Painlev\'e II equation:
\begin{figure}[H]
\caption{$\Sigma$-Painlev\'e}
\vskip 15pt
\begin{tikzpicture}[scale=0.7]
\draw[thick]		(3, 3) -- (2,2);						
\draw[->,thick,>=stealth] 		(0,0) -- (2,2);		
\draw[thick] 			(0,0) -- (-2,2); 				
\draw[->,thick,>=stealth]  	(-3,3) -- (-2,2);	
\draw[->,thick,>=stealth]		(-3,-3) -- (-2,-2);							
\draw[thick]						(-2,-2) -- (0,0);
\draw[thick,->,>=stealth]		(0,0) -- (2,-2);								
\draw[thick]						(2,-2) -- (3,-3);
\draw	[fill]							(0,0)		circle[radius=0.1];	
\node [below] at (0,0) {0};
\draw [dashed] (0,0)--(4,0);
\draw [dashed] (-4,0)--(0,0);
\node[right] at (3,3)					{$\Sigma^{(\text{P})}_1$};
\node[left] at (-3,3)					{$\Sigma^{(\text{P})}_2$};
\node[left] at (-3,-3)					{$\Sigma^{(\text{P})}_3$};
\node[right] at (3,-3)				{$\Sigma^{(\text{P})}_4$};
\end{tikzpicture}
\label{fig: Painleve}
\end{figure}
\begin{align}
\label{matrices- Painleve}
e^{-i\theta\ad\sigma_3 } v^{(2)}(\zeta) &=e^{-4i\left(\zeta^3+\left( x/(4 t^{1/3}) \right)  \zeta \right) \ad\sigma_3 }  \twomat{1} {0} { {r(0)} }{1} , \quad \zeta\in \Sigma^{(\text{P})}_1\cup \Sigma^{(\text{P})}_2 \\
\nonumber
                                                                                                                   &=e^{-4i\left(\zeta^3+\left( x/(4 t^{1/3}) \right)  \zeta \right) \ad\sigma_3 } \twomat{1}{-\overline{r(0)}}{0}{1}, \quad \zeta\in \Sigma^{(\text{P})}_3 \cup \Sigma^{(\text{P})}_4
\end{align}
which is exactly solvable

Let $P$ is a solution of the Painlev\'e $\text{II}$ equation$$P''(s)-sP(s)-2P^3(s)=0$$
	determined by $r(0)$. Then the reduced factorization problem above is related to the Painlev\'e II equation
	by an isomonodromy problem associated to the linear problem
	\[
	\frac{d\psi}{dz}=\left(\begin{array}{cc}
	-4iz^{2}-is-2iP^{2} & 4Piz-2P'\\
	-4Piz-2P' & 4iz^{2}+is+2iP^{2}
	\end{array}\right)\psi.
	\]
	with $s=x/t^{1/3}$ and as $\zeta\to\infty$:
$$\Psi_{i}(s, \zeta)\sim e^{-( [4i/3]\zeta^3+is\zeta  )\sigma_3}. $$
	Here over six sections (cf. \cite[Figure 5.7]{DZ93}), one has the jump relations
	\[
	\psi_{i+1}\left(s,z\right)=\psi_{i}\left(s,z\right)S_{i},\ 1\leq i\leq6,\ \psi_{7}=\psi_{1}
	\]
	where $S_{i}$'s are determined by three parameters $\left(\mathrm{p},\mathrm{q},\mathrm{r}\right)$
	satisfying 
	\[
	\mathrm{r}=\mathrm{p}+\mathrm{q}+\mathrm{pqr}.
	\]
In our setting, we have that 	
	\[
	\mathrm{p}=r\left(0\right),\ \mathrm{q}=-r\left(0\right),\ \mathrm{r}=\frac{\mathrm{p+\mathrm{q}}}{1-\mathrm{pq}}=0.
	\]
Then one can reconstruct $P$ from $\psi$ (\cite[(5.44)]{DZ93}):
$$P=P(x/t^{1/3})=\lim_{\zeta\to \infty} 2i \zeta \left( \Psi e^{ \left( (4i/3) \zeta^3   +is\zeta   \right)   \sigma} -I\right)_{12}.$$
Since this isomonodromy problem is standard, we refer Deift-Zhou \cite[Sec.5]{DZ93} for full details.

We then proceed as in the previous section and study the integral equation related to the $\dbar$ problem. Setting $z=\alpha+i\beta$ and $\zeta=(u+z_0t^{1/3})+iv$, the region $\Omega_1$ corresponds to $u\geq v \geq 0 $. We decompose the integral operator into three parts:
$$
 \int_{\Omega_1}  \dfrac{1}{|z-\zeta|} |W(\zeta)| \, d\zeta  \lesssim  I_1 + I_2 
$$
where
\begin{align*}
I_1 	&=	\int_0^\infty \int_v^\infty \dfrac{1}{|z-\zeta|} \left\vert t^{-1/3} r'\left( u t^{-1/3} \right)   \right\vert e^{-16u^2v} \, du \, dv \\[5pt]
I_2	&=	\int_0^\infty \int_v^\infty \frac{1}{|z-\zeta|}  \dfrac {1} { t^{1/3} \left| u t^{-1/3} +ivt^{-1/3}  \right|^{1/2} } e^{-16u^2v} \, du \, dv\\[5pt].
\end{align*}
We first note that 
$$ \left( \int_\bbR \left\vert t^{-1/3} r'\left( u t^{-1/3} \right)   \right\vert^2 du \right)^{1/2} = t^{-1/6}\norm{r'}{L^2}$$
Using this and the following estimate from \cite[proof of Proposition D.1]{BJM16} 
\begin{equation}
\label{BJM16.bd1}
\norm{\frac{1}{|z-\zeta|}}{L^2(v,\infty)} \leq \frac{\pi^{1/2}}{|v-\beta|^{1/2}}.
\end{equation}
 and \textit{Cauchy-Schwarz}'s inequality on the $u$-integration we may bound $I_1$  by constants times
$$
 t^{-1/6} \norm{r'}{2} \int_0^\infty \frac{1}{|v-\beta|^{1/2}} e^{-v^3} \, dv \lesssim t^{-1/6}.
$$
For $I_2$,  we estimate
\begin{align*}
\norm{ \dfrac {1} { t^{1/3} \left| u t^{-1/3} +ivt^{-1/3}  \right|^{1/2} } }{ L^p (v, \infty) } &\leq  \left( \int_v ^\infty t^{-p/3} \left( \dfrac{  1  }{  (u t^{-1/3} )^2 + (v t^{-1/3})^2  } \right)^{p/4} du \right)^{1/p}\\ 
  &=t^{-1/6} \left(  \int_{v }^\infty \left( \dfrac{  1  }{  u^2 + v^2  } \right)^{p/4} du \right)^{1/p}\\
   &\leq c t^{-1/6}v^{1/p-1/2}.
   \end{align*}
Now by \eqref{BJM16.bd1} and an application of the \textit{H\"older}'s inequality with $P>2$ we get
\begin{align*}
|I_2|  &\leq  \int_0^\infty \norm{ \dfrac {1} { t^{1/3} \left| u t^{-1/3} +ivt^{-1/3}  \right|^{1/2} } }{ L^p (v, \infty) }  \norm{ \dfrac{1}{|z-\zeta|}}{L^q (v, \infty)} e^{-16v^3} dv\\
   &\leq c \int_0^\infty t ^{-1/6 }v^{1/p-1/2} \left\vert v-\beta \right\vert^{1/q-1} e^{-16v^3} dv\\
   &\leq c t ^{-1/6 }.
\end{align*}
This proves that 
$$
 \int_{\Omega_1}  \dfrac{1}{|z-\zeta|} |W(\zeta)| \, d\zeta  \lesssim  t ^{-1/6 } .
$$
We now show that 
\begin{equation}
    \label{error:II}
     \int_{\Omega_1} |W(\zeta)| \, d\zeta  \lesssim  t ^{ -1/6 } 
\end{equation}
Again we decompose the integral above into two parts
\begin{align*}
I_1 	&=	\int_0^\infty \int_v^\infty  \left\vert t^{-1/3} r'\left( u t^{-1/3} \right)   \right\vert e^{-16u^2v} \, du \, dv \\[5pt]
I_2	&=	\int_0^\infty \int_v^\infty \dfrac {1} { t^{1/3} \left| u t^{-1/3} +ivt^{-1/3}  \right|^{1/2} } e^{-16u^2v} \, du \, dv.\\[5pt]
\end{align*}
By \textit{Cauchy-Schwarz}'s inequality:
\begin{align*}
I_1 &\leq \int_0^\infty t^{-1/6} \norm{r'}{2} \left(  \int_v^\infty e^{-16u^2 v} du\right)^{1/2} dv\\
 &\leq c t^{-1/6} \int_0^\infty \dfrac{e^{-16 v^3} }{\sqrt[4]{v}}dv\\
 &\leq c t^{-1/6}.
\end{align*}
By  \textit{H\"older}'s  inequality:
\begin{align*}
I_2 &\leq c t^{-1/6 } \int_0^\infty  v^{1/p-1/2} \left( \int_v^\infty e^{-16qu^2 v} du \right)^{1/q} dv\\
     &\leq c t^{-1/6 } \int_0^\infty v^{3/2p-1} e^{-16v^3}dv\\
     &\leq c t^{-1/6 }. 
\end{align*}

We now follows the argument of Section \ref{sec:large-time} and \cite[Section 5]{DZ93} to obtain the long-time asymptotic formula in Region III ($x<0$):
\begin{align}
\label{asym-III}
u(x,t)&= \lim_{z\to \infty}-2z m_{12}(x, t; z)\\
\nonumber
 &= \lim_{\zeta\to \infty}-2 t^{-1/3}\zeta m_{12}(x, t; \zeta)\\
 \nonumber
&=\dfrac{1}{(3t)^{1/3}}P\left( \dfrac{x}{ (3t)^{1/3} } \right)+\mathcal{O} \left(  t^{ -1/2 } \right)
\end{align}
where  $P$ is a solution of the Painlev\'e II equation 
$$P''(s)-sP(s)-2P^3(s)=0$$ determined by $r(0)$.
\subsubsection{$x>0$}
In this case, we have the stationary points 
$$\pm z_0=\pm \sqrt{\dfrac{-x }{ 12 t }}=\pm i  \sqrt{\dfrac{ |x| }{ 12 t }}$$
stay on the imaginary axis. Given the signature table of $\theta$ function (see \cite[Figure 5.9]{DZ93}), we again perform the scaling 
$$z\to \zeta t^{-1/3}$$
and contour deformation 
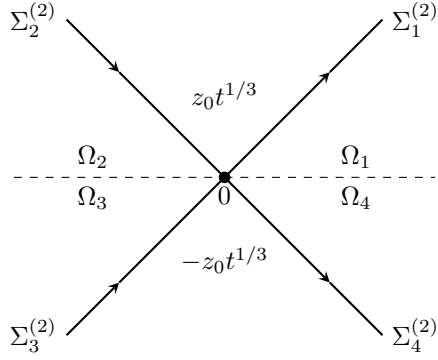
\begin{figure}[H]
\vskip 15pt
\begin{tikzpicture}[scale=0.7]
\draw[thick]		(3, 3) -- (2,2);						
\draw[->,thick,>=stealth] 		(0,0) -- (2,2);		
\draw[thick] 			(0,0) -- (-2,2); 				
\draw[->,thick,>=stealth]  	(-3,3) -- (-2,2);	
\draw[->,thick,>=stealth]		(-3,-3) -- (-2,-2);							
\draw[thick]						(-2,-2) -- (0,0);
\draw[thick,->,>=stealth]		(0,0) -- (2,-2);								
\draw[thick]						(2,-2) -- (3,-3);
\draw	[fill]							(0,0)		circle[radius=0.1];	
\node [below] at (0,0) {0};
\draw [dashed] (0,0)--(4,0);
\draw [dashed] (-4,0)--(0,0);
\node[below] at (0,2)			{$z_0 t^{1/3}$};
\node[above] at (0,-2)			{$-z_0 t^{1/3}$};
\node[right] at (3,3)					{$\Sigma^{(\text{III+})}_1$};
\node[left] at (-3,3)					{$\Sigma^{(\text{III+})}_2$};
\node[left] at (-3,-3)					{$\Sigma^{(\text{III+})}_3$};
\node[right] at (3,-3)				{$\Sigma^{(\text{III+})}_4$};
\node[above] at (2.5,0)           {$\Omega_1$};
\node[above] at (-2.5,0)           {$\Omega_2$};
\node[below] at (-2.5,0)           {$\Omega_3$};
\node[below] at (2.5,0)           {$\Omega_4$};
\end{tikzpicture}
\label{fig:contour-scale}
\end{figure}
We again only discuss the $\dbar$-problem in $\Omega_1$. In $\Omega_1$, we write 
$$\zeta= u+iv $$ 
then
\begin{align*}
\text{Re} (i\theta (\zeta t^{-1/3})) &=8\left( -3u^2 v +v^3 -xvt^{-1/3}\right) \\
      & \leq 8 \left( -3u^2v +u^2v \right)\\
      &\leq -16u^2 v.
\end{align*}
To apply the $\dbar$ method, we define
	\begin{align*}
	R_1	&=	\begin{cases}
						\twomat{0}{0}{r(\zeta t^{-1/3} ) e^{2i\theta (\zeta t^{-1/3 } )  }  }{0}		
								&	z \in (0, \infty)\\[10pt]
								\\
						\twomat{0}{0}{r( 0 ) e^{2i\theta (\zeta t^{-1/3 } )  }  }{0}	
								&	z	\in \Sigma^{(\text{III+})}_1
					\end{cases}
	\end{align*}
and the interpolation is given by 
$$r(0)+ \left( r\left( \text{Re}\zeta t^{-1/3} \right) -r(0)  \right) \cos 2\phi.$$
We can now repeat the analysis in the case above for $x<0$ and obtain the same long time asymptotics as \eqref{asym-III}.
\subsection{Region II} We follow the strategy of the previous subsection. We now scale
$$z\to \zeta z_0$$ and the jump matrix becomes
\begin{equation}
\label{v-scale}
\Twomat{1}{-\overline{r(\zeta z_0 ) }  e^{-2i\theta (\zeta z_0 ) } }{0}{1} \Twomat{1}{0}{r(\zeta z_0) e^{2i\theta (\zeta z_0 )  } } {1},
						\quad z \in\bbR 
\end{equation}
where
						$$ \theta (\zeta z_0 )=4 \tau \zeta^3+x \zeta z_0=4 \tau ( \zeta^3-3 \zeta).$$
\begin{figure}[H]
\caption{$\Sigma-\text{Region-II}$}
\vskip 15pt
\begin{tikzpicture}[scale=0.7]
\draw[thick]		(5, 3) -- (4,2);						
\draw[->,thick,>=stealth] 		(2,0) -- (4,2);		
\draw[thick] 			(-2,0) -- (-4,2); 				
\draw[->,thick,>=stealth]  	(-5,3) -- (-4,2);	
\draw[->,thick,>=stealth]		(-5,-3) -- (-4,-2);							
\draw[thick]						(-4,-2) -- (-2,0);
\draw[thick,->,>=stealth]		(2,0) -- (4,-2);								
\draw[thick]						(4,-2) -- (5,-3);
\draw[thick]		(0,0)--(2,0);
\draw[thick,->,>=stealth] (-2,0) -- (0, 0);
\draw	[fill]							(-2,0)		circle[radius=0.1];	
\draw	[fill]							(2,0)		circle[radius=0.1];
\draw [dashed] (2,0)--(6,0);
\draw [dashed] (-6,0)--(-2,0);
\node[below] at (-2,-0.35)			{$-1$};
\node[below] at (2,-0.35)			{$1$};
\node[right] at (5,3)					{$\Sigma^{(\text{II})}_1$};
\node[left] at (-5,3)					{$\Sigma^{(\text{II})}_2$};
\node[left] at (-5,-3)					{$\Sigma^{(\text{II})}_3$};
\node[right] at (5,-3)				{$\Sigma^{(\text{II})}_4$};
\node[above] at (4.5,0)           {$\Omega_1$};
\node[above] at (-4.5,0)           {$\Omega_2$};
\node[below] at (-4.5,0)           {$\Omega_3$};
\node[below] at (4.5,0)           {$\Omega_4$};
\end{tikzpicture}
\label{fig:contour-scale-III}
\end{figure}		
For brevity	, we again only discuss the $\dbar$-problem in $\Omega_1$. In $\Omega_1$, we write 
$$\zeta= u+1 +iv $$ 
then
\begin{align}
\label{Retheta: II}
\text{Re} (2i\theta (\zeta z_0)) &=8\tau \left( -3(u+1)^2 v +v^3 +3v\right) \\
\nonumber
      & \leq 8\tau \left( -3u^2v -6 uv  +v^3 \right)\\
      \nonumber
      &\leq -16 \tau u v
\end{align}
	\begin{align*}
	R_1	&=	\begin{cases}
						\twomat{0}{0}{r(\zeta z_0 ) e^{2i\theta (\zeta z_0 )  }  }{0}		
								&	z \in ( 1, \infty)\\[10pt]
								\\
						\twomat{0}{0}{r( z_0 ) e^{2i\theta (\zeta z_0 )  }  }{0}	
								&	z	\in \Sigma^{(\text{II})}_1
					\end{cases}
	\end{align*}
and the interpolation is given by 
$$r(z_0)+ \left( r\left( \text{Re}\zeta z_0 \right) -r(z_0)  \right) \cos 2\phi.$$
So we arrive at the $\dbar$-derivative in $\Omega_1$ in the $\zeta$ variable:
\begin{equation}
\dbar R_1	= \left( z_0 r'\left( u z_0\right) \cos 2\phi- 2\dfrac{ r(uz_0)  -r(z_0)  }{  \left\vert \zeta-1\right\vert   } e^{i\phi} \sin 2\phi  \right) e^{2i\theta   }
\end{equation}
\begin{equation}
\label{R1.bd2}
\left| \dbar R_1 e^{\pm 2i\theta}  \right| 	\lesssim\left( | z_0 r'\left( u z_0 \right) | +\dfrac{ z_0 \norm{r'}{L^2} }{  |\zeta z_0 -z_0  |^{1/2} }\right) e^{-16\tau u v}.
\end{equation}		
We now replace $t^{-1/3}$ and $e^{-16u^2 v}$ in the previous subsection with $z_0$ and  $e^{-16\tau u v}$ respectively and conclude that
\begin{equation}
\label{est: Omega-III}
     \int_{\Omega_1} |W(\zeta)| \, d\zeta  \lesssim  z_0^{1/2} \tau^{-3/4}. 
\end{equation}
and arrive at the following long-time asymptotics:
\begin{align}
\label{asym-II}
u(x,t) &=\dfrac{1}{(3t)^{1/3}}P\left( \dfrac{x}{ (3t)^{1/3} } \right)+\mathcal{O} \left( \tau^{-3/4}  z_0^{ 3/2 } \right) \\
      &=\dfrac{1}{(3t)^{1/3}}P\left( \dfrac{x}{ (3t)^{1/3} } \right)+\mathcal{O} \left( (z_0t) ^{-3/4} \right).
\end{align}
\begin{remark}
    \label{rmk:I overlap}
    In the overlap between Region II and III, we take  $\text{Re} (2i\theta (\zeta z_0))<-16\tau u^2v$ in \eqref{Retheta: II} and the corresponding estimate in \eqref{est: Omega-III} becomes
    \begin{equation}
        \int_{\Omega_1} |W(\zeta)| \, d\zeta  \lesssim  z_0^{1/2} \tau^{-1/2} 
    \end{equation}
    and the resulting asymptotics in Region II is:
   \begin{equation}
  u(x,t) =\dfrac{1}{(3t)^{1/3}}P\left( \dfrac{x}{ (3t)^{1/3} } \right)+\mathcal{O} \left( t^{-1/2}\right)     
   \end{equation}
\end{remark}
which matches up with \eqref{asym-III}.

\subsection{Region IV}
\begin{figure}[H]
\caption{$\Sigma-\text{Region-IV}$}
\vskip 15pt
\begin{tikzpicture}[scale=0.7]
\draw[thick]		(2, 3) -- (1,2);						
\draw[->,thick,>=stealth] 		(0,1) -- (1,2);		
\draw[thick] 			(-1,2) -- (0, 1); 				
\draw[->,thick,>=stealth]  	(-2,3) -- (-1,2);	
\draw[->,thick,>=stealth]		(-2,-3) -- (-1,-2);							
\draw[thick]						(-1,-2) -- (0, -1);
\draw[thick,->,>=stealth]		(2, -3) -- (1,-2);								
\draw[thick]						(1,-2) -- (0,-1);
\draw[dashed]		(0,0)--(4,0);
\draw[dashed] (-4,0) -- (0, 0);
\draw [dashed] (0,1)--(2,1);
\draw [dashed] (0,1)--(-2,1);
\node[right] at (2,3)					{$\Sigma^{(\text{IV})}_1$};

\node[right] at (2,-3)				{$\Sigma^{(\text{IV})}_2$};
\node[below] at (0,3)           {$\Omega_1$};
\node[above] at (2.5,0)           {$\Omega_2$};
\node[below] at (2.5,0)           {$\Omega_3$};
\node[above] at (0,-3)           {$\Omega_4$};
\node[below] at (0,1.1)      {$i\eta $};
\node[above] at (0,-1.1)      {$-i\eta$};
\node[above] at (-2, 0)  {(1)};
\node[above] at (1, 1)  {(2)};
\node[above] at (-1, 1)  {(3)};
\end{tikzpicture}
\label{fig:contour-scale-IV}
\end{figure}	
In this region, we have 
$$\tau=\left( \dfrac{x}{12t^{1/3}} \right)^{3/2}>(M)^{-1}>0$$
and choose a constant $\eta$ such that  $0<\eta< (M)^{-1/3}$
The contour deformation is given above and we carry out the same scaling 
$$z\to \zeta t^{-1/3}.$$
We extend $r$ to Part (1) of  $\Omega_2$ by setting $r=r( \text{Re}\zeta t^{-1/3} )$. Also in this region, 
\begin{align*}
\text{Re} (2i\theta (\zeta t^{-1/3})) &=8\left( -3 u^2v +v^3 \right) -2 (x t^{-1/3})v\\
& = 8 \left( -3u^2v +v^3 \right)-24 \tau^{2/3}v\\
& =-24u^2 v-16\tau^{2/3} v.
\end{align*}
We now integrate and find that
\begin{align}
\label{error-(1)}
 \int_{(1)} \left\vert t^{-1/3} r'(u t^{-1/3}) e^{2i\theta(\zeta t^{-1/3})} \right\vert d\zeta
 & =\int_0^\eta \int_{-\infty}^{\infty} \left\vert t^{-1/3} r'(u t^{-1/3}) e^{-24u^2v-16\tau^{2/3}v} \right\vert dudv\\
 \nonumber
 & \lesssim t^{-1/6}\tau^{-1/2}.
\end{align}
 In Part (2) , we write 
$$\zeta= u+i(v+\eta) $$ 
then
\begin{align}
\label{phase-iv}
\text{Re} (2i\theta (\zeta t^{-1/3})) &=8\left( -3 u^2 (v+\eta) +(v+\eta)^3 \right) -2 (x t^{-1/3}) (v+\eta )\\
\nonumber
      & \leq 8 \left( -3u^2v -3u^2 \eta  +v^3 +3 v^3\eta +3v\eta^2 +\eta^3 \right)\\
      \nonumber
      &\quad -24\tau^{2/3}(v+\eta)\\
      \nonumber
      &\leq -16(u^2 v+\tau^{2/3}\eta).
\end{align}
For the $\overline{\partial}$ problem, we set
	\begin{align*}
	R_1	&=	\begin{cases}
						\twomat{0}{0}{r(\zeta t^{-1/3} ) e^{2i\theta (\zeta t^{-1/3 } )  }  }{0}		
								&	z \in \mathbb{R}\\[10pt]
								\\
						\twomat{0}{0}{r( 0 ) e^{2i\theta (\zeta t^{-1/3 } )  }  }{0}	
								&	z	\in \Sigma^{(\text{IV})}_1
					\end{cases}
	\end{align*}
and the interpolation is given by 
$$r(0)+ \left( r\left( \text{Re}\zeta t^{-1/3} \right) -r(0)  \right) \cos 2\phi.$$
So we arrive at the $\dbar$-derivative in $\Omega_1$ in the $\zeta$ variable:
\begin{equation}
\dbar R_1	= \left( {t^{-1/3}} r'\left( u t^{-1/3} \right) \cos 2\phi- 2\dfrac{ r(ut^{-1/3} )  -r(0)  }{  \left\vert \zeta-i\eta \right\vert   } e^{i\phi} \sin 2\phi  \right) e^{2i\theta}.
\end{equation}
\begin{align*}
\label{R1.bd3}
\left| \dbar R_1 e^{\pm 2i\theta}  \right| 	&\lesssim\left( |  t^{-1/3} r'\left( u t^{-1/3} \right) | +\dfrac{\norm{r'}{L^2} }{ t^{1/3} |u t^{-1/3} +ivt^{-1/3}  |^{1/2} }\right) \\
            & \quad \times e^{-16(u^2 v + \tau^{2/3}\eta)}
\end{align*}
Following the same procedure, we  show that 
\begin{equation}
\label{error-(2)}
\int_{(2)} |W(\zeta)| \, d\zeta  \lesssim  t ^{-1/6 } e^{-16\tau^{2/3}\eta}
\end{equation}
 which is the error term resulting from the $\dbar$ estimate.  We can now combine \eqref{error-(1)} and \eqref{error-(2)} and follow the argument in Section \ref{sec:large-time} and \cite[Section 5]{DZ93} to obtain the long-time asymptotic formula in Region IV:
 \begin{equation}
\label{asym-IV}
u(x,t)=\dfrac{1}{(3t)^{1/3}}P\left( \dfrac{x}{ (3t)^{1/3} } \right)+\mathcal{O} \left( (t\tau)^{-1/2}+ \dfrac{e^{-16\tau^{2/3}\eta}} {t^{ 1/2 } }\right)
\end{equation}
where $P$ is a solution of the  Painlev\'e II equation 
$$P''(s)-sP(s)-2P^3(s)=0$$ determined by $r(0)$.
\subsection{Region V}
\begin{figure}[H]
\caption{$\Sigma-\text{Region-V}$}
\vskip 15pt
\begin{tikzpicture}[scale=0.7]

\draw[dashed]		(0,0)--(4,0);
\draw[dashed] (-4,0) -- (0, 0);
\draw	[fill]							(0, 2.5)		circle[radius=0.05];	
\draw	[fill]							(0,-2.5)		circle[radius=0.05];
\draw [thick] (0,1)--(2,1);
\draw [thick,->,>=stealth] (-2,1)--(0,1);
\draw [thick] (0,-1)--(2,-1);
\draw [thick,->,>=stealth] (-2,-1)--(0,-1);
\node[left] at (-2,1)					{$\Sigma^{(V)}_1$};
\node[right] at (4,0)					{$\bbR$};
\node[left] at (-2,-1)				{$\Sigma^{(V)}_2$};
\node[left] at (0,2.5)           {$z_0$};
\node[above] at (2.5,0)           {$\Omega_1$};
\node[below] at (2.5,0)           {$\Omega_2$};
\node[left] at (0,-2.5)           {$-z_0$};
\node[above] at (0,1.1)      {$ih$};
\node[below] at (0,-1)      {$-ih$};

\end{tikzpicture}
\label{fig:contour-scale-IV}
\end{figure}	
Given $|z_0|>M^{-1}$, let $h=M^{-1}/2$, then
we can directly read off that for $z=u+iv\in \Omega_1$
\begin{align}
\operatorname{Re}(2 i \theta(z)) &=2 t\left(4\left(-3 u^{2} v+v^{3}\right)-\frac{x}{t} v\right) \\
& \leq-24 u^{2} v t+2\left(4 h^{2}-\frac{x}{t}\right) v t \\
& \leq-24 u^{2} v t-2 c v t
\end{align}
So we simply factorize 
\begin{align*}
    \label{s}
    e^{-i\theta \ad \sigma_3} v(z)	=\twomat{1}{-\rbar  e^{-2i\theta}}{0}{1}
					\twomat{1}{0}{r e^{2i\theta}}{1}
\end{align*}
and deform $\bbR$ to $\Sigma^{(5)}_1$ and $\Sigma^{(5)}_2$. We only study the case of $\Omega_1$
It is obvious that $r(u)e^{2i\theta}$ decays exponentially on $\Sigma^{(5)}_1$, so we are only left with the error term
\begin{align*}
\int_{\Omega_{2}}\left|r^{\prime}(u) e^{2 i \theta(z)}\right| d z &=\int_{0}^{\eta} \int_{-\infty}^{\infty}\left|r^{\prime}(u) e^{-\left(24 u^{2} v+2 c v\right) t}\right| d u d v \\
&<\int_{0}^{\infty} \frac{e^{-2 c v t}}{\sqrt{v t}} d v \\
& \leq t^{-1}
\end{align*}
and the analysis in $\Omega_2$ is identical. So we obtain in Region V
\begin{equation}
\label{asym-V}
u(x,t)=\mathcal{O} \left( t^{-1}\right). 
\end{equation}

\begin{remark}
    \label{extention}
    If we instead let the initial condition $u_0\in H^{2,s}(\bbR)$ where $s>1/2$, then following a similar and simpler argument as in \cite[section 3]{CLL}, we can deduce that the reflection coefficient $r\in H^{s'}(\bbR)$ for all $1/2<s'<s$. Then replacing $r(z)$ by the convolution form as given in  \cite[(5.15)]{CLL}, we can deduce that the resulting error terms in \eqref{W.L1.est},\eqref{asym-III} \eqref{asym-II}, \eqref{asym-IV} and \eqref{asym-V} become
    $$
    \mathcal{O}\left( (z_0 t)^{-(1+2s')/4} \right), ~ \mathcal{O}\left( t^{-(2+s')/6} \right), ~ \mathcal{O}\left( (z_0 t)^{-(1+2s')/4} \right), ~ \mathcal{O} \left( (t\tau)^{-(2+s')/6}\right), ~\mathcal{O}\left( t^{-s'} \right).
    $$
\end{remark}

\section{Global approximation of solutions}
\label{sec: approx}
The goal of this section is to extend our long-time asymptotics given by Theorem \ref{thm:main1} to the MKdV equation with rougher initial data.  Three important spaces are { $H^{1}$, $H^{\frac{1}{4}}$ and $L^2$}. In $H^{1}$, the MKdV equation has certain conserved quantities (cf. Subsection \ref{subsec:H1}). For $H^{\frac{1}{4}}$, this space is the lowest regularity that the solution can be constructed by iterations (cf. Theorem \ref{thm:KVPlocal} and Subsection \ref{H1/4}). { Finally, in $L^2$, the mKdV enjoys the conservation of mass. }  We will show that the long-time asymptotics remain valid in these spaces after we introduce decay at $\pm\infty$.

We first sketch the local
existence and uniqueness of the strong solution in $H^{s}$ for $s\ge\frac{1}{4}$.
We mainly follow Kenig-Ponce-Vega \cite{KPV} and Linares-Ponce \cite{LP}. 

First of all, we define the solution operator to the linear Airy equation
by
\[
W\left(t\right)u_{0}=e^{-t\partial_{xxx}}u_{0}.
\]
In other words, using the Fourier transform, one has
\[
\mathcal{F}_{x}\left[W\left(t\right)u_{0}\right]\left(\xi\right)=e^{it\xi^{3}}\hat{u}_{0}\left(\xi\right).
\]
\begin{definition}

The \emph{strong solution} is defined in the following integral sense:
we say the function $u\left(x,t\right)$ is a strong solution
in $H^{s}\left(\mathbb{R}\right)$ to
\begin{equation}
\partial_{t}u+\partial_{xxx}u - 6u^{2}\partial_{x}u=0,\quad u\left(0\right)=u_{0}\in H^{s}\left(\mathbb{R}\right)\label{eq:IVP}
\end{equation}
if and only if $u\in C\left(I,H^{s}\left(\mathbb{R}\right)\right)$
satisfies
\begin{equation}
u=W\left(t\right)u_{0} + \int_{0}^{t}W\left(t-s\right)\left(6u^{2}\partial_{x}u\left(s\right)\right)\,ds.\label{eq:mild}
\end{equation}

\end{definition}
We also define
\[
\mathcal{D}_{x}^{s}h\left(x\right)=\mathcal{F}^{-1}\left[\left|\xi\right|^{s}\hat{h}\left(\xi\right)\right]\left(x\right).
\]
Then with notations introduced above, we have the classical local well-posedness results due to Kenig-Ponce-Vega \cite{KPV}.
\begin{theorem}[Kenig-Ponce-Vega]
\label{thm:KVPlocal}Let $s\ge\frac{1}{4}$. Then
for any $u_{0}\in H^{s}\left(\mathbb{R}\right)$ there is $T=T\left(\left\Vert \mathcal{D}_{x}^{\frac{1}{4}}u_{0}\right\Vert _{L^{2}}\right)\sim\left\Vert \mathcal{D}_{x}^{\frac{1}{4}}u_{0}\right\Vert _{L^{2}}^{-4}$
such that there exists a unique strong solution $u\left(t\right)$
to the initial value problem
\[
\partial_{t}u+\partial_{xxx}u - 6u^{2}\partial_{x}u=0,\,u\left(0\right)=u_{0}
\]
satisfying
\begin{equation}
u\in C\left(\left[-T,T\right]:H^{s}\left(\mathbb{R}\right)\right)\label{eq:E1}
\end{equation}
\begin{equation}
\left\Vert \mathcal{D}_{x}^{s}\partial_{x}u\right\Vert _{L_{x}^{\infty}\left(\mathbb{R}:L_{t}^{2}\left[-T,T\right]\right)}<\infty,\label{eq:E2}
\end{equation}
\begin{equation}
\left\Vert \mathcal{D}_{x}^{s-\frac{1}{4}}\partial_{x}u\right\Vert _{L_{x}^{20}\left(\mathbb{R}:L_{t}^{\frac{5}{2}}\left[-T,T\right]\right)}<\infty,\label{eq:E3}
\end{equation}
\begin{equation}
\left\Vert \mathcal{D}_{x}^{s}u\right\Vert _{L_{x}^{5}\left(\mathbb{R}:L_{t}^{10}\left[-T,T\right]\right)}<\infty,\label{eq:E4}
\end{equation}
and
\begin{equation}
\left\Vert u\right\Vert _{L_{x}^{4}\left(\mathbb{R}:L_{t}^{\infty}\left[-T,T\right]\right)}<\infty.\label{eq:E5}
\end{equation}
Moreover, there exists a neighborhood $\mathcal{N}$ of $u_{0}$ in
$H^{s}\left(\mathbb{R}\right)$ such that the solution map: $\tilde{u}_{0}\in\mathcal{N}\longmapsto\tilde{u}$
is smooth with respect to the norms given by \eqref{eq:E1}-\eqref{eq:E5}.
\end{theorem}

\begin{proof}
Given $T$ and $\mathcal{C}$, define the space
\begin{equation}
\mathcal{X}_{T}^{s}=\left\{ v\in C\left(\left[-T,T\right]:H^{s}\left(\mathbb{R}\right)\right):\vertiii{v}  _{\mathcal{X}_{T}^{s}}<\infty\right\} \label{eq:XT}
\end{equation}
and
\begin{equation}
\mathcal{X}_{T,\mathcal{C}}^{s}=\left\{ v\in C\left(\left[-T,T\right]:H^{s}\left(\mathbb{R}\right)\right):\vertiii{v}  _{\mathcal{X}_{T}^{s}}\leq\mathcal{C}\right\} \label{eq:XTC}
\end{equation}
where
\begin{align*}
\vertiii{v} _{\mathcal{X}_{T}^{s}} & =\left\Vert \mathcal{D}_{x}^{s}v\right\Vert _{L_{t}^{\infty}\left(\left[-T,T\right]:H^{s}\left(\mathbb{R}\right)\right)}+\left\Vert v\right\Vert _{L_{x}^{4}\left(\mathbb{R}:L_{t}^{\infty}\left[-T,T\right]\right)}\\
 & +\left\Vert \mathcal{D}_{x}^{s}v\right\Vert _{L_{x}^{5}\left(\mathbb{R}:L_{t}^{10}\left[-T,T\right]\right)}+\left\Vert \mathcal{D}_{x}^{s-\frac{1}{4}}\partial_{x}v\right\Vert _{L_{x}^{20}\left(\mathbb{R}:L_{t}^{\frac{5}{2}}\left[-T,T\right]\right)}+\left\Vert \mathcal{D}_{x}^{s}\partial_{x}v\right\Vert _{L_{x}^{\infty}\left(\mathbb{R}:L_{t}^{2}\left[-T,T\right]\right)}.
\end{align*}
To obtain a strong solution to the initial-value problem we need to find
appropriate $T$ and $\mathcal{C}$ such that the operator
\[
\mathcal{S}\left(v,u_{0}\right)=\mathcal{S}\left(v\right)=W\left(t\right)u_{0} + \int_{0}^{t}W\left(t-s\right)\left(6v^{2}\partial_{x}v\left(s\right)\right)\,ds
\]
is a contraction map on $\mathcal{X}_{T,\mathcal{C}}^{s}$.

Using linear estimates for $W\left(t\right)$ and the Leibniz rule
for fractional derivatives one can show that
\[
\vertiii{ \mathcal{S}\left(v\right)}_{\mathcal{X}_{T}^{s}}\leq c\left\Vert u_{0}\right\Vert _{H^{s}}+cT^{\frac{1}{2}}\vertiii{v}_{\mathcal{X}_{T}^{s}}^{3}
\]
where $c$ is from linear estimates etc independent of the initial
data. We refer the reader to Kenig-Ponce-Vega \cite{KPV} and Linares-Ponce \cite{LP}
for details. Then choose $\mathcal{C}=2c\left\Vert u_{0}\right\Vert _{H^{s}}$
and $T$ such that $c\mathcal{C}^{2}T^{\frac{1}{2}}<\frac{1}{4}$,
we obtain that
\[
\mathcal{S}\left(\cdot,u_{0}\right):\,\mathcal{X}_{T,\mathcal{C}}^{s}\rightarrow\mathcal{X}_{T,\mathcal{C}}^{s}.
\]
Similarly, one can also show
\begin{align*}
\vertiii{\mathcal{S}\left(v_{1}\right)-\mathcal{S}\left(v_{2}\right)}_{\mathcal{X}_{T}^{s}} & \leq cT^{\frac{1}{2}}\left(\vertiii {v_{1}}_{\mathcal{X}_{T}}^{2}+\vertiii {v_{2}} _{\mathcal{X}_{T}}^{2}\right)\vertiii{v_{1}-v_{2}}_{\mathcal{X}_{T}^{s}}\\
 & \leq2cT^{\frac{1}{2}}\mathcal{C}^{2}\vertiii{v_{1}-v_{2}} _{\mathcal{X}_{T}^{s}}.
\end{align*}
Therefore, with our choice of $T$ and $\mathcal{C}$, $\mathcal{S}\left(\cdot,u_{0}\right)$
is a contraction on $\mathcal{X}_{T,\mathcal{C}}^{s}$. So there is
a unique fixed point of this $\mathcal{S}\left(\cdot,u_{0}\right)$
in $\mathcal{X}_{T,\mathcal{C}}^{s}$. Hence we obtain the unique strong
solution:
\[
u=\mathcal{S}\left(u\right)=W\left(t\right)u_{0} + \int_{0}^{t}W\left(t-s\right)\left(6u^{2}\partial_{x}u\left(s\right)\right)\,ds.
\]
To check the dependence on the initial data, {using
arguments similar to those above}, one can show that
\begin{align*}
\vertiii{\mathcal{S}\left(u_{1},u_{1}\left(0\right)\right)-\mathcal{S}\left(u_{2},u_{2}\left(0\right)\right)} _{\mathcal{X}_{T_{1}}^{s}} & \leq c\left\Vert u_{1}\left(0\right)-u_{2}\left(0\right)\right\Vert _{H^{s}}\\
 & +cT_{1}^{\frac{1}{2}}\left(\vertiii{u_{1}}_{\mathcal{X}_{T_{1}}^{s}}^{2}+\vertiii{u_{2}} _{\mathcal{X}_{T_{1}}^{s}}^{2}\right)\vertiii{u_{1}-u_{2}}_{\mathcal{X}_{T_{1}}^{s}}.
\end{align*}
This can be used to show that for $T_{1}\in\left(0,T\right)$, the
solution map from a neighborhood $\mathcal{N}$ of $u_{0}$ depending
on $T_{1}$ to $\mathcal{X}_{T_{1},\mathcal{C}}^{s}$ is Lipschitz.
Further work can be used to show the solution map is actually smooth. For more details, see Kenig-Ponce-Vega \cite{KPV} and Linares-Ponce
\cite{LP}.
\end{proof}
Finally, we notice that if $u_{0}$ is Schwartz, then
the solution $u$ to the initial-value problem is also smooth and
hence a classical solution. The uniqueness of the classical solution
is well-known. We refer the reader to Bona-Smith \cite{BS}, Temam
\cite{T69} and Saut-Temam \cite{ST} for the KdV problem and Saut \cite{S79} for more general KdV type equations including the MKdV equation.

\subsection{Solutions of mKdV by inverse scattering and strong solutions}

As before, given $u_{0}\in H^{2,1}\left(\mathbb{R}\right)$, one can
solve the MKdV equation using the inverse scattering transform. 

Recall from \eqref{mkdv.BC}, we have the solution to the MKdV equation in terms of the {solution by inverse scattering}:
\begin{align}\label{eq:BC}
u &=  \left[ \dfrac{-i}{\pi}  \int\mu\left(w_{\theta}^{+}+w_{\theta}^{-}\right) \right]_{12} \\
\nonumber
  & =   \left[ \dfrac{-i}{\pi} \int\left(\mu-I\right)\left(w_{\theta}^{+}  +w_{\theta}^{-}\right)  \right]_{12} + \left[ \dfrac{-i}{\pi}  \int\left(w_{\theta}^{+}+w_{\theta}^{-}\right) \right]_{12} 
\end{align}
where $\mu$ is constructed using the reflection coefficients $r$.
But as we discussed above, using PDE techniques, one can construct
solutions with rougher data at least locally. Motivated by Deift-Zhou
\cite{DZ03}, we try to understand the relations between Beals-Coifman
solutions and strong solutions. First of all, if $u_{0}$ is
Schwartz, one can also show $u$ is Schwartz (cf.Deift-Zhou
\cite{DZ93}). So in this case, the strong solution is the same
as the solution via inverse scattering. Our goal is to identify the {solution by inverse scattering}
with the strong solution whenever the former
makes sense. Starting from the local construction, we will try to
extend these results globally later on.

Firstly, we show that one can always take the limit of a sequence of smooth solutions to the MKdV equation in weighted $L^2$ spaces without regularity assumptions.
\begin{lemma}\label{lem:limit}
Suppose there is a
sequence $\left\{ u_{0,k}\right\} $ of Schwartz functions which is a Cauchy sequence in $H^{j,1}\left(\mathbb{R}\right)$ and $u_{0,k}\rightarrow u_{0}$
in $H^{j,1}\left(\mathbb{R}\right)$ with $j\geq0$. Then for fixed $t>0$, one can always conclude that the sequence of solution $\lbrace u_k \rbrace$ to the MKdV equation with initial data $u_{0,k}$ obtained via inverse scattering in the sense of \eqref{eq:BC} has a $L^\infty$ limit.
\end{lemma}
\begin{proof}
	Since $u_{0,k}$ is Schwartz, from the inverse scattering transform, we can write down  the Beals-Coifman
	solutions
	\begin{equation}
	    \label{int-k}
	    u_{k}=  \left[ \dfrac{-i}{\pi}   \int\mu_{k}\left(w_{k,\theta}^{+}+w_{k,\theta}^{-}\right) \right]_{12}
	\end{equation}
	with initial data $u_{0,k}$. 
	Using the mapping properties of the direct scattering due to Zhou \cite[Theorem 1.8]{Zhou98} and Deift-Zhou \cite[Theorem 3.2]{DZ03}, in terms of reflection
	coefficients, we have that
	\[
	r_{k}=\mathcal{R}\left(u_{0,k}\right)\in H^{1},
	\]
	and by the Lipschitz continuity of the map, we have
	\[
	\left\Vert r_{k}-r_{\ell}\right\Vert _{H^{1}\left(\mathbb{R}\right)}\lesssim\left\Vert u_{0,k}-u_{0,\ell}\right\Vert _{H^{j,1}\left(\mathbb{R}\right)}.
	\]
	{By the integral representation of $u_k$ given in \eqref{int-k}, resolvent estimates in \cite{Zhou98} (see also \cite[(2.19) (2.21)]{DZ03}) and Lipschitz continuity of the direct and inverse scattering map, }one also has
	\[
	\left\Vert u_{\ell}-u_{k}\right\Vert _{L^{\infty}\left(\mathbb{R}\right)}\lesssim\left\Vert r_{k}-r_{\ell}\right\Vert _{H^{1}\left(\mathbb{R}\right)}.
	\]
	Since $r_{k}$ converges to a function $r_{\infty}$ in $H^{1}\left(\mathbb{R}\right)$,
	we claim that the corresponding {solution by inverse scattering} converges to a limit
	\[
	u_{\infty}=\lim_{k\rightarrow\infty}u_{k}
	\]
	in the sense of the $L^{\infty}$ norm. Indeed, we can write
	\begin{align*}
	u_{k} & =  \left[ \dfrac{-i}{\pi}   \int\mu_{k}\left(w_{k,\theta}^{+}+w_{k,\theta}^{-}\right) \right]_{12}\\
	& = \left[ \dfrac{-i}{\pi}   \int\left(\mu_{k}-I\right)\left(w_{k,\theta}^{+}+w_{k,\theta}^{-}\right) \right]_{12}  + \left[ \dfrac{-i}{\pi}   \int\left(w_{k,\theta}^{+}+w_{k,\theta}^{-}\right) \right]_{12}  \\
	& =\textrm{ I}_{k}+\textrm{ II}_{k}.
	\end{align*}
	Then due to the resolvent estimate, $\left(\mu_{k}-I\right)$ is bounded in the $L^{2}$ and the $L^{2}$ estimate for $w_{k,\theta}^{+}+w_{k,\theta}^{-}$
	is straightforward, so $\textrm{ I}_{k}$ makes sense pointwise. For $\textrm{ II}_{k}$,
	one simply notices that $\int\left(w_{k,\theta}^{\pm}\right)$ is
	proportional to $W\left(t\right)\check{r}_{k}=e^{-t\partial_{xxx}}\check{r}_{k}$,
	so by the standard stationary phase analysis, for $r_{k}\in H^{1}$,
	$\textrm{ II}_{k}$ is a function in $L^{\infty}\left(\mathbb{R}\right)$ for $t\geq0$ with
	the standard pointwise decay estimates for the Airy equation  (cf.\cite[Lemma 2.1]{GPR}).
	
	Hence for fixed $t\neq0$
	\[
	\left\Vert u_{k}(t)-u_{\infty}(t)\right\Vert _{L^{\infty}}\rightarrow0\,\ \text{as}\,\ \left\Vert r_{k}-r_{\infty}\right\Vert _{H^{1}}\rightarrow0
	\]as desired.
\end{proof}
\begin{remark}
	Note that a-priori, when we pass the solutions by inverse scattering to the
	pointwise limit above, it is not clear what the limit means since the limit is rougher than
	the required regularity  from the inverse scattering transform when $j<2$.
\end{remark}
In the following subsections, we use PDE techniques to conclude that indeed the limit constructed by the lemma above is a solution to the MKdV equation so long as we have enough regularity to perform the Picard iteration.
First of all, we illustrate that solutions we analyzed in earlier sections are strong solutions.
\begin{corollary}
\label{lem:identi}Suppose $u_{0}\in H^{2,1}\left(\mathbb{R}\right)$
then the {solution by inverse scattering} and the strong solution are the same
(up to a measure zero set)
\begin{align*}
u & =   \left[ \dfrac{-i}{\pi}  \int\mu\left(w_{\theta}^{+}+w_{\theta}^{-}\right) \right]_{12}    \\
 & =W\left(t\right)u_{0} + \int_{0}^{t}W\left(t-s\right)\left(6u^{2}\partial_{x}u\left(s\right)\right)\,ds
\end{align*}
in $\left[-T,T\right]$ where $T$ is given as in Theorem \ref{thm:KVPlocal}.
\end{corollary}

\begin{remark}
At such a high level of regularity, by the uniqueness of
weak solutions, see for example Ginibre-Tsutsumi \cite{GT} and Ginibre-Tsutsumi-Velo
\cite{GTV}, one might expect this identification. But here we provide
a direct approach in this specific situation.
\end{remark}

\begin{proof}
Suppose $u_{0}\in H^{2,1}\left(\mathbb{R}\right)$, we can find a
sequence $\left\{ u_{0,k}\right\} $ of Schwartz functions such that
it is a Cauchy sequence in $H^{2,1}\left(\mathbb{R}\right)$ and $u_{0,k}\rightarrow u_{0}$
in $H^{2,1}\left(\mathbb{R}\right)$.

We {may }assume that for all $k$, there is a uniform bound
\[
\left\Vert u_{0,k}\right\Vert _{\dot{H}^{2}\left(\mathbb{R}\right)}\lesssim\left\Vert u_{0,k}\right\Vert _{H^{2}\left(\mathbb{R}\right)}\lesssim\left\Vert u_{0,k}\right\Vert _{H^{2,1}\left(\mathbb{R}\right)}\leq C.
\]
Then applying Theorem \ref{thm:KVPlocal}, we can find a strong
solution $u_{k}$ with initial data $u_{0,k}$ in $\mathcal{X}_{T,\mathcal{C}}^{2}$
where $T$ and $\mathcal{C}$ are chosen as in Theorem \ref{thm:KVPlocal}. 

By Theorem \ref{thm:KVPlocal}, we also have
\[
\vertiii {u_{k}-u_{\ell}} _{\mathcal{X}_{T,\mathcal{C}}^{2}}\lesssim\left\Vert u_{0,k}-u_{0,\ell}\right\Vert _{H^{2}\left(\mathbb{R}\right)}.
\]
So in $\mathcal{X}_{T,C}^{2}$, $u_{k}$ converges to a limit $u_{\infty}$
which is a strong solution. Using the notation from above, we have
\[
u_{\infty}=\mathcal{S}\left(u_{\infty},u_{0}\right)\in\mathcal{X}_{T,C}^{2}.
\]
From the inverse scattering transform, we also have
solutions via inverse scattering
\[
\tilde{u}_{k}=  \left[ \dfrac{-i}{\pi}   \int\mu_{k}\left(w_{k,\theta}^{+}+w_{k,\theta}^{-}\right) \right]_{12}
\]
with initial data $u_{0,k}$. 

Since $u_{0,k}$ is Schwartz, so $u_{k}$ and $\tilde{u}_{k}$ are
also Schwartz. Therefore we have $u_{k}=\tilde{u}_{k}.$ 
By Lemma \ref{lem:limit}, one can conclude that there exists $\tilde{u}_{\infty}$ such that for $t\neq0$,
\[
\left\Vert \tilde{u}_{k}(t)-\tilde{u}_{\infty}(t)\right\Vert _{L^{\infty}}\rightarrow0.
\] 
By the convergence of the strong solutions, it follows that as $k\rightarrow\infty$, we have
\[
\vertiii{u_{k}-u_{\infty}} _{\mathcal{X}_{T,C}^{2}}=\vertiii{\tilde{u}_{k}-u_{\infty}} _{\mathcal{X}_{T,C}^{2}}\rightarrow0.
\]
In particular, as $k\rightarrow\infty$, one has
\[
\sup_{t\in\left[-T,T\right]}\left\Vert u_{k}-u_{\infty}\right\Vert _{H^{2}\left(\mathbb{R}\right)}=\sup_{t\in\left[-T,T\right]}\left\Vert \tilde{u}_{k}-u_{\infty}\right\Vert _{H^{2}\left(\mathbb{R}\right)}\rightarrow0.
\]
By construction, as $k\rightarrow\infty,$
\[
\sup_{t\in\left[-T,T\right]}\left\Vert \tilde{u}_{k}-\tilde{u}_{\infty}\right\Vert _{L^{\infty}\left(\mathbb{R}\right)}\rightarrow 0.
\]
Hence
\[
u_{\infty}=\tilde{u}_{\infty}
\]
up to a measure zero set. 

Therefore, we can conclude that 
\begin{align*}
u & = \left[ \dfrac{-i}{\pi}  \int\mu\left(w_{\theta}^{+}+w_{\theta}^{-}\right) \right]_{12}  \\
 & =W\left(t\right)u_{0}+ \int_{0}^{t}W\left(t-s\right)\left(6u^{2}\partial_{x}u\left(s\right)\right)\,ds
\end{align*}
in $\left[-T,T\right]$.
\end{proof}
Next, we will try to use this local identification to understand the
limits of  solutions via inverse scattering in various low regularity spaces.

\subsection{Approximation of solutions in $H^{1}\left(\mathbb{R}\right)$}\label{subsec:H1}

First of all, we consider
\[
\partial_{t}u+\partial_{xxx}u -6u^{2}\partial_{x}u=0,\,u\left(0\right)=u_{0}.
\]
with initial data in $H^{1}\left(\mathbb{R}\right)$.

The following three quantities are preserved by the solution flow:
\[
I_{1}\left(u\right)=\int_{-\infty}^{\infty}u\,dx,
\]
\[
I_{2}\left(u\right)=\int_{-\infty}^{\infty}u^{2}\,dx,
\]
\begin{align*}
E\left(u\right)=I_{3}\left(u\right) & =\int_{-\infty}^{\infty}\left[\left(\partial_{x}u\right)^{2} + u^{4}\right]\,dx.
\end{align*}
Using the local existence results and the conservation laws above,
we can extend a local solution to a global solution in $H^{1}\left(\mathbb{R}\right)$.

More precisely, using the Sobolev embedding, one has
\begin{align*}
E\left(u\right) & =\int_{-\infty}^{\infty}\left[\left(\partial_{x}u\right)^{2} + u^{4}\right]\,dx\\
 & \geq\left\Vert \partial_{x}u\right\Vert _{L^{2}\left(\mathbb{R}\right)}^{2} +\left\Vert u\right\Vert _{L^{4}\left(\mathbb{R}\right)}^{4}\\
 & \geq\left\Vert \partial_{x}u\right\Vert _{L^{2}\left(\mathbb{R}\right)}^{2} + c_{4}\left\Vert \partial_{x}u\right\Vert _{L^{2}\left(\mathbb{R}\right)}\left\Vert u\right\Vert _{L^{2}\left(\mathbb{R}\right)}^{3}.
\end{align*}
From $I_{2}$, we know the $L^{2}\left(\mathbb{R}\right)$ norm is
conserved.

If we denote
\[
f\left(t\right)=\left\Vert \partial_{x}u\left(t\right)\right\Vert _{L^{2}\left(\mathbb{R}\right)}
\]
then one has
\[
f^{2}\left(t\right) + c_{4}\left\Vert u\right\Vert _{L^{2}\left(\mathbb{R}\right)}^{3}f\left(t\right)\leq E\left(u_{0}\right)
\]
so $f\left(t\right)$ is bounded globally. In other words,
\[
\left\Vert \partial_{x}u\left(t\right)\right\Vert _{L^{2}\left(\mathbb{R}\right)}\lesssim E\left(u_{0}\right).
\]
Hence with the conserved $L^{2}\left(\mathbb{R}\right)$
norm, we conclude that
\begin{equation}
\left\Vert u\right\Vert _{H^{1}\left(\mathbb{R}\right)}\lesssim\left\Vert u_{0}\right\Vert _{H^{1}\left(\mathbb{R}\right)}.\label{eq:energybound}
\end{equation}
\begin{theorem}
\label{thm:H11}For $u_{0}\in H^{1,1}\left(\mathbb{R}\right)$, the
strong solution given by the Duhamel formulation \eqref{eq:mild} has
the same asymptotics as in our main Theorem \ref{thm:main1}.
\end{theorem}

\begin{proof}
We perform a construction similar to the construction in the proof of Corollary \ref{lem:identi}. Let $\left\{ u_{0,k}\right\} \in H^{2,1}\left(\mathbb{R}\right)$
be a Cauchy sequence in $H^{1,1}\left(\mathbb{R}\right)$ such that
\[
\lim_{k\rightarrow\infty}u_{0,k}\rightarrow u_{0}
\]in $H^{1,1}\left(\mathbb{R}\right)$ and $\sup_{k}\left\Vert u_{0,k}\right\Vert _{H^{1,1}\left(\mathbb{R}\right)}\leq C.$ 

Then we can use the inverse scattering transform to solve the initial-value
problem \eqref{eq:IVP} and obtain solutions $u_{k}$
by inverse scattering
\begin{equation}
u_{k}=  \left[ \dfrac{-i}{\pi}   \int\mu_{k}\left(w_{k, \theta}^{+}+w_{k, \theta}^{-}\right)\right]_{12}\label{eq:BCH1}
\end{equation}
with initial data $u_{0,k}$. 
By Lemma \ref{lem:limit}, one can conclude that there exists $u_{\infty}$ such that for $t\neq0$,
\[
\left\Vert u_{k}(t)-u_{\infty}(t)\right\Vert _{L^{\infty}}\rightarrow0.
\] For $t=0$, this convergence can be implied by Sobolev's embedding.

However,  by Corollary \ref{lem:identi}, we know $u_{k}$ is also
a strong solution, i.e.,
\[ u_{k}(t)=W\left(t\right)u_{0,k} +\int_{0}^{t}W\left(t-s\right)\left(6\left(u_{k}\right)^{2}\partial_{x}\left(u_{k}\right)\right)\,ds.
\]
Then we can use $T$ and $\mathcal{C}$ as in Theorem \ref{thm:KVPlocal}
to conclude that
\[
\vertiii{u_{k}-u_{\ell}} _{\mathcal{X}_{T,\mathcal{C}}^{1}}\lesssim\left\Vert u_{0,k}-u_{0,\ell}\right\Vert _{H^{1}\left(\mathbb{R}\right)}
\]
where $\mathcal{X}_{T,\mathcal{C}}^{1}$ is given as \eqref{eq:XTC}. 

Hence $\left\{ u_{k}\right\} $ is also a Cauchy sequence in $\mathcal{X}_{T,\mathcal{C}}^{1}$
which converges to $u$ satisfying
\[
u(t)=W\left(t\right)u_{0} + \int_{0}^{t}W\left(t-s\right)\left(6u^{2}\partial_{x}u\left(s\right)\right)\,ds
\]
by construction. So $u$ is a strong solution.

By the definition of space $\mathcal{X}_{T,\mathcal{C}}^{1}$ \eqref{eq:XTC},
we have
\[
\lim_{k\rightarrow\infty}\sup_{t\in\left[-T,T\right]}\left\Vert u_{k}-u\right\Vert _{H^{1}\left(\mathbb{R}\right)}=0.
\]
Combining
\[
\lim_{k\rightarrow\infty}\left\Vert u_{k}(t)-u_{\infty}(t)\right\Vert _{L^{\infty}\left(\mathbb{R}\right)}=0
\]
 we can conclude that $u(t)=u_{\infty}(t)$ pointwise (up to a measure zero
set) for $t\in[-T,T]$. 
Since the $H^{1}$ norms of $u$ is uniformly bounded as \eqref{eq:energybound}, 
we can repeat the above construct infinity many times to extend the
interval $\left[-T,T\right]$ to $\mathbb{R}$ and conclude that for
$t\in\mathbb{R}_{+}$
\[
u\left(t\right)=\tilde{u}_{\infty}\left(t\right).
\]
Since $u_{\infty}$ is the pointwise limit of solutions by inverse scattering
which have asymptotic behavior in our main theorem obtained from
the nonlinear steepest descent with uniform error terms estimates, $u_{\infty}$ also has the desired
asymptotics. 
More precisely, we can write
\[
u_{k}\left(x,t\right)=L_{k}\left(x,t\right)+E_{k}\left(x,t\right)
\]
where $L_{k}\left(x,t\right)$ gives the leading order behavior and
$E_{k}\left(x,t\right)$ collects the error term. By
the convergence of scattering data, we know
\[L_{k}\left(x,t\right)\rightarrow L_{\infty}\left(x,t\right)
\]
pointwise. Hence for an arbitrary fixed $t$, as the pointwise limit of $u_{k}\left(t\right)$,
one can write
\[
u(t)=u_{\infty}\left(t\right)=L_{\infty}\left(x,t\right)+E_{\infty}\left(x,t\right)
\]
where the decay estimates for $E_{\infty}\left(x,t\right)$ is the
same as $E_{k}\left(x,t\right)$ due to the uniform error estimates.
Therefore $u$ also has the asymptotic behavior as claimed.
\end{proof}
\begin{remark}
Similar to the situation of the NLS in Deift-Zhou \cite{DZ03}, the
solution $u$ as the limit of the sequences of solutions by inverse scattering also enjoys the conservation law
\begin{align*}
E\left(u\right)=I_{3}\left(u\right) & =\int_{-\infty}^{\infty}\left[\left(\partial_{x}u\right)^{2} + u^{4}\right]\,dx
\end{align*}
 since it is also a strong solution. It is not clear how to obtain
this conservation law using the inverse scattering transform due to
the low regularity.
\end{remark}

\subsection{Approximation of solutions in $H^{\frac{1}{4}}\left(\mathbb{R}\right)$}\label{H1/4}

For the MKdV equation, as in Theorem \ref{thm:KVPlocal}, Kenig, Ponce and
Vega obtained the lowest regularity for the local well-posedness in
$H^{s}\left(\mathbb{R}\right)$ , $s\ge\frac{1}{4},$ in \cite{KPV}.
They also showed in \cite{KPV2} that when $s<\frac{1}{4}$ the data-to-solution
map fails to be uniformly continuous as a map from $H^{s}$ to $C\left(\left[-T,T\right]H^{s}\left(\mathbb{R}\right)\right)$ (see also Christ-Colliander-Tao \cite{CCT}). These imply that the space
$H^{\frac{1}{4}}\left(\mathbb{R}\right)$ has the lowest regularity that the solution
can be obtained by iteration. These local results form the basis for
the global well-posedness. For example one can use the
energy conservation and the $L^{2}$ conservation to obtain the global
well-posedness. But in the space $H^{\frac{1}{4}}$, there is no conservation
laws allow us to do similar extensions. Then one needs to use the
\textquotedblleft I-method\textquotedblright{}, introduced by Colliander-Keel-Staffilani-Takaoka-Tao
\cite{CKSTT}, which plays a great role in constructing global solutions.
They obtained global well-posedness for KdV for $s>-\frac{3}{4}$
and then using the Miura transform to obtain the global well-posedness
for the MKdV equation in $H^{s}\left(\mathbb{R}\right)$ for $s>\frac{1}{4}$.
In Guo \cite{Guo} and Kishimoto \cite{Kis}, the authors use more delicate
spaces to handle "logarithmic divergence" and combine with the
I-method to conclude the global well-posedness for KdV in $H^{-\frac{3}{4}}$.
Then with the Miura transform given by \cite{CKSTT}, they also obtain
the global well-posedness for the MKdV equation in $H^{\frac{1}{4}}$. The most
important ingredient shown in these papers for the MKdV equation is that for some
$\kappa>0$, one has the following growth estimate
\[
\left\Vert u\left(t\right)\right\Vert _{H^{\frac{1}{4}}\left(\mathbb{R}\right)}\lesssim\left(1+t\right)^{\kappa}\left\Vert u_{0}\right\Vert _{H^{\frac{1}{4}}\left(\mathbb{R}\right)}.
\]
\begin{theorem}
\label{thm:H1/41}For $u_{0}\in H^{\frac{1}{4},1}\left(\mathbb{R}\right)$,
the strong solution given by the integral representation \eqref{eq:mild}
has the same asymptotics as in our main Theorem \ref{thm:main1}.
\end{theorem}

\begin{proof}
As in Theorem \ref{thm:H11}, we first show that locally the limit
of solutions by inverse scattering is the strong solution in $H^{\frac{1}{4}}\left(\mathbb{R}\right)$.
The difference here is that we use the growth rate estimate to extend the identification globally.

Let $\left\{ u_{0,k}\right\} \in H^{2,1}\left(\mathbb{R}\right)$
be a Cauchy sequence in $H^{\frac{1}{4},1}\left(\mathbb{R}\right)$
such that
\[
\lim_{k\rightarrow\infty}u_{0,k}\rightarrow u_{0}
\]
in $H^{\frac{1}{4},1}\left(\mathbb{R}\right)$ and $\sup_{k}\left\Vert u_{0,k}\right\Vert _{H^{\frac{1}{4},1}\left(\mathbb{R}\right)}\leq C.$ 

Using the inverse scattering transform to solve the initial-value
problem \eqref{eq:IVP}, we obtain a sequence of solutions
\begin{equation}
u_{k}= \left[ \dfrac{-i}{\pi}  \int\mu\left(w_{k, \theta}^{+}+w_{k, \theta}^{-}\right) \right]_{12}   .\label{eq:BCH1-1}
\end{equation}
By Lemma \ref{lem:limit}, one can conclude that there exists ${u}_{\infty}$ such that for $t\neq0$,
\[
\left\Vert {u}_{k}(t)-{u}_{\infty}(t)\right\Vert _{L^{\infty}}\rightarrow0.
\] For $t=0$, the pointwise convergence can be achieved by the standard $L^p$ spaces argument up to a subsequence.

Moreover, by Corollary \ref{lem:identi}, we also know $u_{k}$ is also
a strong solution, i.e.,
\[
u_{k}=W\left(t\right)u_{0,k} + \int_{0}^{t}W\left(t-s\right)\left(6\left(u_{k}\right)^{2}\partial_{x}\left(u_{k}\right)\right)\,ds.
\]
Then we can use $T$ and $\mathcal{C}$ as in Theorem \ref{thm:KVPlocal}
to conclude that
\[
\vertiii{u_{k}-u_{\ell}} _{\mathcal{X}_{T,\mathcal{C}}^{\frac{1}{4}}}\lesssim\left\Vert u_{0,k}-u_{0,\ell}\right\Vert _{H^{1}\left(\mathbb{R}\right)}
\]
where $\mathcal{X}_{T,\mathcal{C}}^{\frac{1}{4}}$ is given as \eqref{eq:XTC}. 

Hence $\left\{ u_{k}\right\} $ is also a Cauchy sequence in $\mathcal{X}_{T,\mathcal{C}}^{\frac{1}{4}}$
which converges to $u$ satisfying
\[
u=W\left(t\right)u_{0} + \int_{0}^{t}W\left(t-s\right)\left(6u^{2}\partial_{x}u\left(s\right)\right)\,ds
\]
by construction. 

By the definition of space $\mathcal{X}_{T,\mathcal{C}}^{\frac{1}{4}}$
\eqref{eq:XTC}, we have
\[
\lim_{k\rightarrow\infty}\sup_{t\in\left[-T,T\right]}\left\Vert u_{k}-u\right\Vert _{H^{\frac{1}{4}}\left(\mathbb{R}\right)}=0.
\]
 and combining
\[
\lim_{k\rightarrow\infty}\left\Vert u_{k}(t)-u_{\infty}(t)\right\Vert _{L^{\infty}\left(\mathbb{R}\right)}=0
\]
we can conclude that $u=u_{\infty}$ pointwise in $\left[-T,T\right]$
(up to a measure zero set). 

By the global well-posedness, $u$ exists in $H^{\frac{1}{4}}\left(\mathbb{R}\right)$
globally. By construction, one can also define $u_{\infty}\left(t\right)$
for all $t\in\mathbb{R}$.

By symmetry, we consider $t\ge0$. Suppose $u_{\infty}\left(t\right)= u\left(t\right)$ does not hold
for all $t\ge0$. Let
\[
t_{\star}=\inf\left\{ t\geq0|u_{\infty}\left(t\right)\neq u\left(t\right)\right\} .
\]
Clearly by the above argument, $T<t_{\star}<\infty$. 

By the growth rate estimate from Guo \cite{Guo} and Kishimoto \cite{Kis},
we have for $t\leq t_{\star}$
\[
\left\Vert u\left(t\right)\right\Vert _{H^{\frac{1}{4}}\left(\mathbb{R}\right)}\leq C\left(1+t_{\star}\right)^{\kappa}\left\Vert u_{0}\right\Vert _{H^{\frac{1}{4}}\left(\mathbb{R}\right)}.
\]
Also by construction, for $t<t_{\star},$
\[
u_{\infty}\left(t\right)=u\left(t\right).
\]
By Theorem \ref{thm:KVPlocal}, we can find $\mathcal{C}_{\star}$
and $T_{\star}$ depending on $C\left(1+t_{\star}\right)^{\kappa}\left\Vert u_{0}\right\Vert _{H^{\frac{1}{4}}\left(\mathbb{R}\right)}<\infty$
to construct $\mathcal{X}_{T_{\star},\mathcal{C}_{\star}}^{\frac{1}{4}}$.
Due to the explicit dependence of $T$ on the size of the initial
data in Theorem \ref{thm:KVPlocal}, $T_{\star}\geq\epsilon_{\star}>0$.

By the definition of $t_{\star}$, we have two situations: firstly 
\begin{equation}
u_{\infty}\left(t_{\star}\right)\neq u\left(t_{\star}\right)\label{eq:Situ1}
\end{equation}
or for any $\eta>0$, there exists $t_{\star}<t_{\eta}<t_{\star}+\eta$
such that 
\begin{equation}
u_{\infty}\left(t_{\eta}\right)\neq u\left(t_{\eta}\right)\label{eq:Situ2}
\end{equation}
in particular, we can take $\eta<\frac{\epsilon_{\star}}{8}$.

Again by construction, we have
\[
u_{\infty}\left(t_{\star}-\frac{\epsilon_{\star}}{8}\right)=u\left(t_{\star}-\frac{\epsilon_{\star}}{8}\right).
\]
Applying Theorem \ref{thm:KVPlocal} and the first part of this proof
using space $\mathcal{X}_{T_{\star},\mathcal{C}_{\star}}^{\frac{1}{4}}$,
we have
\[
u_{\infty}\left(t_{\star}-\frac{\epsilon_{\star}}{8}+s\right)=u\left(t_{\star}-\frac{\epsilon_{\star}}{8}+s\right)
\]
for $s\in\left[0,\epsilon_{\star}\right]\subset\left[0,T_{\star}\right].$
In particular, $u_{\infty}\left(t_{\star}\right)=u\left(t_{\star}\right)$
and $u_{\infty}\left(t_{\star}+s\right)=u\left(t_{\star}+s\right)$
for $s\in\left[0,\frac{\epsilon_{\star}}{4}\right].$ This is a contraction
with either \eqref{eq:Situ1} or \eqref{eq:Situ2}. So our assumption
for the existence of $t_{\star}$ fails.

Thus we can conclude that $u_{\infty}\left(t\right)=u\left(t\right)$
for all $t\ge0$. Then  the asymptotic behavior of $u$ is obtained as in Theorem \ref{thm:H11}.
\end{proof}
\begin{remark}
	For an alternative approach using low regularity conservation laws developed in Koch-Tataru \cite{KoTa} and Killip-Visan-Zhang \cite{KiViZh}, see our work on the focusing MKdV in \cite{CL19}.
\end{remark}

\subsection{Approximation of solutions in $L^2(\bbR)$}\label{subsec:L2app}

As we introduced before, it is known from \cite{CCT} and \cite{KPV2} that $H^{\frac{1}{4}}$ is the optimal space to perform the Picard iteration to construction the \emph{strong} solution in the sense of the Duhamel formula.  With appropriation notations and topology, in the work by Harrop-Griffiths-Killip-Visan\cite{HGKV}, the well-posedness of the mKdV equation can be obtained in $H^{\tau}(\mathbb{R})$ with $\tau>-\frac{1}{2}$.:

\begin{theorem}[\cite{HGKV}]\label{thm:gwphgkv}
Let $\tau>-\frac{1}{2}$. Then the mKdV equation \eqref{MKDV} is globally well-posed for all initial data in the sense that the solution map $\Phi$ extends uniquely from Schwartz spaces to a jointly continuous map $\Phi:\,\mathbb{R}\times H^{\tau}(\mathbb{R})\rightarrow H^{\tau}(\mathbb{R}).$
\end{theorem}
The notation of solution used above can be understood as the unique limit of Schwartz solutions. We also refer to Definition 1.1  in Kappeler-Topalov \cite{KaTo} for the interpretation of this notation of solution.  This notation is well-suited for our global approximation argument since the Schwartz solutions can be obtained via the inverse scattering and their asymptotics can be computed with uniform error estimates.


From the view of the standard analysis of Jost functions, it suffices to require the potential to be in $L^1$ which contains $L^{2,s}$ with $s>\frac{1}{2}$. The well-posedness theory above can allow use the extend the asymptotics of solutions of the mKdV equations with initial data in $L^{2,s}$. Again here we focus on  $s=1$.

\begin{theorem}
\label{thm:L21}For $u_{0}\in L^{2,1}\left(\mathbb{R}\right)$, the solution given by Theorem \ref{thm:gwphgkv}
has the same asymptotics as in  Theorem \ref{thm:main1}.
\end{theorem}

\begin{proof}
For any $u_0\in L^{2,1}(\mathbb{R})$, we pick a sequence of Schwartz functions $\{u_{0,k}\}$ 
such that
\begin{equation}\label{eq:convL21}
\lim_{k\rightarrow\infty}u_{0,k}\rightarrow u_{0}
\end{equation}
in $L^{2,1}\left(\mathbb{R}\right)$ and $\sup_{k}\left\Vert u_{0,k}\right\Vert _{L^{2,1}\left(\mathbb{R}\right)}\leq C.$ 

Let $u(x,t)$ be the solution to \eqref{MKDV} in the sense of \ref{thm:gwphgkv} with initial data $u_0$ and $\tau=0$.  Let $u_k(t)$ be the Schwartz solution to \eqref{MKDV} with initial data $u_{0,k}$.  By construction, we know that $\forall t\in \mathbb{R}$, $u_k(t)\rightarrow u(t)$ in $L^2(\mathbb{R})$. Then we also know that up to a subsequence,  $u_k(t)\rightarrow u(t)$ almost everywhere.

Now for each $u_k(t)$, via the nonlinear steepest descent, we can write
\[
u_{k}\left(x,t\right)=L_{k}\left(x,t\right)+E_{k}\left(x,t\right)
\]
where $L_{k}\left(x,t\right)$ gives the leading order behavior and
$E_{k}\left(x,t\right)$ collects the error term which only depend the $L^{2,1}(\mathbb{R})$ norm of $u_{0,k}$.  From the convergence \eqref{eq:convL21}, by the direct scattering, one has the convergence of the reflection coefficients $\lim_{k\rightarrow\infty} r_k=r$ in $H^1(\mathbb{R})$.   Then by
the convergence of reflection coefficients, we know
\[L_{k}\left(x,t\right)\rightarrow L_{\infty}\left(x,t\right)
\]
pointwise. Since the error term $E_k(x,t)$ is uniform in $k$, we can conclude that for an arbitrary fixed $t$, as the pointwise limit of $u_{k}\left(t\right)$ (up to measure zero set),
one can write
\[
u(t)=L_{\infty}\left(x,t\right)+E_{\infty}\left(x,t\right)
\]
where the decay estimates for $E_{\infty}\left(x,t\right)$ is the
same as $E_{k}\left(x,t\right)$ due to the uniform error estimates.
Therefore $u$ also has the asymptotic behavior as claimed.
\end{proof}

\end{document}